\documentclass{icmart}
\usepackage{color}
\numberwithin{equation}{section}
\usepackage{bbm}

\contact[m.dashti@sussex.ac.uk]{Masoumeh Dashti, Department of Mathematics, Sussex University,
Brighton BN1 9QH, UK}

\contact[a.m.stuart@warwick.ac.uk]{Andrew M Stuart,
Mathematics Institute, Warwick University, Coventry CV4 7AL, UK}

\newcommand{\Img}{{\rm Im}_{\tiny {G}}}
\newcommand{\cK}{\mathcal K}

\newcommand{\as}[1]{{\color{black}{#1}}}

\definecolor{darkred}{rgb}{.7,0,0}

\definecolor{darkgreen}{rgb}{0,0.7,0}

\definecolor{darkblue}{rgb}{0,0,0.7}

\newtheorem{theorem}{Theorem}[section]
\newtheorem{corollary}[theorem]{Corollary}
\newtheorem{lemma}[theorem]{Lemma}

\theoremstyle{definition}
\newtheorem{definition}[theorem]{Definition}
\newtheorem{remark}[theorem]{Remark}

\def\HH{\mathring \CH}

\def\Tr{\mathop{\mathrm{Tr}}}

\def\cBc{{\mathsf B}}

\def\cP{{\mathsf P}}

\def\T{\mathbb{T}}
\def\Z{\mathbb{Z}}

\def\hf{\textstyle{1\over 2}}

\def\*{|\!|\!|}

\newcommand\CD{{\mathcal D}}
\newcommand\CC{{C}}

\newcommand\C{C}
\newcommand\LC{\mathsf {L}}
\newcommand\cC{{\mathcal C}}
\newcommand\CM{{\mathcal M}}
\newcommand\CB{B}
\newcommand\CX{X}
\newcommand\CI{I}
\newcommand\CF{{\mathcal F}}
\newcommand\CL{{\mathcal L}}

\newcommand\theo[1]{Theorem~\ref{#1}}

\newcommand{\weakc}{\rightharpoonup}
\newcommand{\bu}{\overline{u}}
\newcommand{\bI}{\overline{I}}
\newcommand{\nx}{\|_X}
\newcommand{\ntu}{\nu^{\mathsf{T}}}
\newcommand{\rp}{\bbP}
\newcommand{\hmu}{\widehat{\mu}}
\newcommand{\hv}{\widehat{v}}
\newcommand{\hw}{\widehat{w}}
\newcommand{\bF}{\overline{f}}

\newcommand{\pdxi}[1]{\frac{\partial #1}{\partial x_i}}

\newcommand{\cN}{N}
\newcommand{\bbE}{\E}
\newcommand{\bbP}{\P}

\newcommand{\dhh}{d_{\mbox {\tiny{\rm Hell}}}}
\newcommand{\dtv}{d_{\mbox {\tiny{\rm TV}}}}
\newcommand{\dkl}{D_{\mbox {\tiny{\rm KL}}}}

\newcommand{\kk}{K}

\newcommand{\brho}{\overline {\rho}}
\newcommand{\nnabla}{D}

\newcommand{\cR}{\mathcal{R}}
\newcommand{\cH}{\mathcal{H}}
\newcommand{\CH}{\mathcal{H}}
\newcommand{\su}{\mathsf u}

\newcommand{\tr}{\mathrm{Tr\,}}

\newcommand{\eps}{\epsilon}

\newcommand{\bbQ}{\mathbb{Q}}
\newcommand{\bbT}{\mathbb{T}}
\newcommand{\bbZ}{\mathbb{Z}}

\newcommand{\RR}{\mathbb{R}}
\newcommand{\R}{\mathbb{R}}
\newcommand{\E}{\mathbb{E}}
\renewcommand{\P}{\mathbb{P}}
\newcommand{\EE}{\E}
\newcommand{\WW}{\mathbb{W}}

\newcommand{\PP}{\P}

\newcommand{\N}{\mathbb{N}}

\newcommand{\h}{\mathcal{X}}

\newcommand{\muy}{\mu^y}

\newcommand{\ud}{\mathrm{d}}
\newcommand{\e}{\mathrm{e}}

\newtheorem{remarks}[theorem]{Remarks}

\newtheorem{assumptions}[theorem]{Assumptions}
\newtheorem{assumption}[theorem]{Assumption}

\newtheorem{example}[theorem]{Example}

\newtheorem{algorithm}[theorem]{Algorithm}

\title[The Bayesian Approach to Inverse Problems]{The Bayesian Approach to Inverse Problems}

\author[Masoumeh Dashti and Andrew M Stuart]
{Masoumeh Dashti and Andrew M Stuart\thanks{The authors are grateful to EPSRC, ERC and ONR 
for financial support which led to the work described in these lecture
notes.}}

\begin{document}

\begin{abstract}
These lecture notes highlight the mathematical and computational
structure relating to the formulation of, and development of algorithms
for, the Bayesian approach to inverse problems in differential equations.
This approach is fundamental in the quantification of uncertainty
within applications involving the blending of mathematical models
with data.
The finite dimensional situation is described first, along with some
motivational examples. Then the development of probability
measures on separable Banach space is undertaken, using a random series 
over an infinite set of functions to construct draws; these probability
measures are used as \as{priors}\index{prior} in the Bayesian approach to inverse problems.
Regularity of draws from the \as{priors}\index{prior} is studied in the natural
Sobolev or Besov spaces implied by the choice of functions in the random
series construction, and the \as{Kolmogorov continuity}
\index{Kolmogorov!continuity criterion} theorem is used to
extend regularity considerations to the space of H\"older continuous
functions. Bayes' theorem is derived in this \as{prior}\index{prior}
setting, and here
interpreted as finding conditions under which the \as{posterior}\index{posterior} 
is absolutely continuous with respect to the \as{prior}\index{prior}, 
and determining 
a formula for the Radon-Nikodym derivative in terms of the 
\as{likelihood}\index{likelihood} of the data. 
Having established the form of the \as{posterior}\index{posterior}, 
we then describe various
properties common to it in the infinite dimensional setting. These
properties include well-posedness, approximation theory,
and the existence of maximum a posteriori estimators. We then
describe measure-preserving dynamics, again on the infinite dimensional
space, including Markov chain-Monte Carlo
and sequential Monte Carlo methods, and measure-preserving reversible
stochastic differential equations. By formulating the theory and
algorithms
on the underlying infinite dimensional space, we obtain a framework
suitable for rigorous analysis of the accuracy of reconstructions,
of computational complexity, as well
as naturally constructing algorithms which perform well under mesh 
refinement, since they are inherently well-defined in infinite dimensions. 
\end{abstract}



\maketitle
\vspace{0.3in}
\section{Introduction}
\label{sec:intro}
\renewcommand\theequation{\thesection.\arabic{equation}}
Many uncertainty quantification problems arising in the sciences and engineering
require the incorporation of data into a model; indeed doing so can 
significantly reduce the uncertainty in model predictions and is hence
a very important step in many applications. Bayes' formula provides
the natural way to do this. The purpose of these lecture notes is to
develop the Bayesian approach to inverse problems in order to provide
a rigorous framework for the development of uncertainty quantification
in the presence of data. Of course it is possible to simply discretize the
inverse problem and apply Bayes' formula on a finite dimensional space.
However we adopt a different approach: we formulate Bayes' formula on
a separable Banach space and study its properties in this infinite
dimensional setting. This approach, of course, requires considerably
more mathematical sophistication and it is important to ask whether
this is justified. The answer, of course, is ``yes''. The formulation
of the Bayesian approach on a separable Banach space has numerous
benefits: (i) it reveals an attractive well-posedness framework for the 
inverse problem, allowing for the study of robustness to changes in
the observed data, or to numerical approximation of the forward model;
(ii) it allows for direct links to be established with the classical
theory of regularization, which has been developed in a separable
Banach space setting; (iii) and it leads to new algorithmic 
approaches which build on the full power of analysis and numerical analysis
to leverage the structure of the infinite dimensional inference problem.

The remainder of this section contains a discussion of Bayesian inversion
in finite dimensions, for motivational purposes, and two examples
of partial differential equation (PDE) inverse problems. In section
\ref{sec:prior} we describe the construction of \as{priors}\index{prior}
on separable
Banach spaces, using random series and employing the
random series to discuss various Sobolev, Besov and H\"older regularity
results. Section \ref{sec:post} is concerned with the statement and
derivation of Bayes' theorem in this separable Banach space setting.
In section \ref{sec:common} we describe various properties common to
the \as{posterior}\index{posterior}, 
including well-posedness in the \as{Hellinger}\index{Hellinger distance} metric,
a related approximation theory which leverages well-posedness to
deliver the required stability estimate, and the existence of maximum
a posteriori (MAP) estimators; these address points (i) and (ii)
above, respectively. Then, in section \ref{sec:mpd}, we
discuss various discrete and continuous time Markov processes
which preserve the posterior probability measure, including 
Markov chain-Monte Carlo methods (MCMC),
sequential Monte-Carlo methods (SMC) and reversible stochastic partial
differential equations, addressing point (iii) above. 
The infinite dimensional perspective on
algorithms is beneficial as it provides a direct way to construct algorithms
which behave well under refinement of finite dimensional approximations 
of the underlying separable Banach space. 
We conclude in section \ref{sec:con} and then 
an appendix collects together
a variety of basic definitions and results from the theory of
differential equations and probability. Each section 
is accompanied by bibliographical notes connecting 
the developments herein to the wider literature.  
The notes complement and build on other overviews of Bayesian
inversion, and its relations to uncertainty quantification,
which may be found in \cite{article:Stuart2010, StuartICM}.
All results (lemmas, theorems etc.) which are quoted without
proof are given pointers to the literature, where proofs may be
found, within the bibliography of the section containing the result.

\vspace{0.3in}
\subsection{Bayesian Inversion on $\R^n$}
\label{ssec:1.1}

Consider the problem of finding $u \in \R^n$ from $y \in \R^J$
where $u$ and $y$ are related by the equation
$$y=G(u).$$
We refer to $y$ as {\em observed data} and to $u$ as the {\em unknown}.
This problem may be difficult for a number of reasons. We highlight
two of these, both particularly relevant to our future developments.

\begin{enumerate}

\item The first difficulty, which may be illustrated in the
case where $n=J$,
concerns the fact that often the equation is perturbed by noise
and so we should really consider the equation
\begin{equation}
\label{11} y=G(u)+\eta,
\end{equation}
where $\eta \in \R^J$ represents the {\em observational noise} which enters
the observed data. Assume further that $G$ maps $\R^J$ into a proper subset
of itself, $\Img$, and that $G$ has a unique inverse as a map from
$\Img$ into $\R^J$. It may then be the case that, because of the
noise, $y \notin \Img$ so that simply inverting
$G$ on the data $y$ will not be possible. 
Furthermore, the specific instance of $\eta$ 
which enters the data may not be known to us;
typically, at best, only the statistical properties of a typical
noise $\eta$ are known. Thus we cannot subtract $\eta$ from the
observed data $y$ to obtain something in $\Img$.
Even if $y \in \Img$ the uncertainty caused by the presence of
noise $\eta$ causes problems for the inversion.

\item The second difficulty is manifest
in the case where $n>J$ so that the system is {\em underdetermined}:
the number of equations is smaller than the number of unknowns.
How do we attach a sensible meaning to the concept of
solution in this case where, generically, there will be many solutions?

\end{enumerate}

Thinking probabilistically enables us to
overcome both of these difficulties. We will treat $u,y$ and $\eta$
as random variables and determine the joint probability distribution
of $(u,y)$. We then define the ``solution'' of the inverse
problem to be the probability distribution of $u$ given $y$, denoted
$u|y$. This allows us to model the noise via its statistical
properties, even if we do not know the exact instance of the
noise entering the given data. And it also allows us to specify
{\em a priori} the form of solutions that we believe to be
more likely, thereby enabling us to attach weights to multiple
solutions which explain the data. This is the
{\em Bayesian approach} to inverse problems.

To this end, we
define a random variable $(u,y)\in\R^n\times\R^J$
as follows. We let $u\in\R^n$ be a random variable with (Lebesgue) density
$\rho_0(u)$. Assume that $y|u$ ($y$ given $u$) is defined
via the formula \eqref{11}
where $G: \R^n\rightarrow\R^J$ is measurable, and
$\eta$ is independent of $u$ (we sometimes
write this as $\eta \perp u$) and
distributed according to measure $\bbQ_0$ with Lebesgue density $\rho(\eta)$. 
Then $y|u$ is simply found by shifting $\bbQ_0$ by $G(u)$
to measure $\bbQ_{u}$ with
Lebesgue density $\rho(y-G(u))$. It follows that 
$(u,y)\in\R^n\times\R^J$ is a random variable with
Lebesgue density $\rho(y-G(u))\rho_0(u)$.\\
The following theorem allows us to calculate the
distribution of the random variable $u|y$:

\begin{theorem} {\bf \as{Bayes' Theorem}\index{Bayes' Theorem}}.
\label{t:1.1}
Assume that
\begin{equation*}
Z:=\int_{\R^n}\rho\big(y-G(u)\big)\rho_0(u)du>0.
\end{equation*}
Then $u|y$ is a random variable with Lebesgue density $\rho^y(u)$
given by
\begin{equation*}
\rho^y(u)=\frac{1}{Z}\rho\big(y-G(u)\big)\rho_0(u).
\end{equation*}
\end{theorem}

\begin{remarks}
\label{r:1.2}
The following remarks establish the nomenclature of
Bayesian statistics, and also frame the previous theorem
in a manner which generalizes to the infinite dimensional
setting.
\begin{itemize}
\item $\rho_0(u)$ is the {\bf \as{prior}\index{prior} density}.

\item $\rho\big(y-G(u)\big)$ is the {\bf \as{likelihood}}\index{likelihood}.

\item $\rho^y(u)$ is the {\bf \as{posterior}\index{posterior} density}.

\item It will be useful in what follows to define
\begin{equation*}
\Phi(u;y)=-\log\rho \big(y-G(u)\big).
\end{equation*}
We call $\Phi$ the {\bf potential}.
This is the {\bf negative log \as{likelihood}.}\index{likeligood!log}

\item Note that $Z$ is the probability of $y$. \as{Bayes' formula}\index{Bayes' formula} expresses
$$\P(u|y)=\frac{1}{\P(y)}\P(y|u)\P(u).$$

\item Let $\mu^y$ be a measure on $\R^n$ with density $\rho^y$
and $\mu_0$ a measure on $\R^n$ with density $\rho_0$. Then
the conclusion of Theorem \ref{t:1.1} may be written as:
\begin{equation}\label{12}
\begin{split}
\frac{d\mu^y}{d\mu_0}(u)=\frac{1}{Z}\exp\big(-\Phi(u;y)\big),\\
Z=\int_{\R^n}\exp\big(-\Phi(u;y)\big)\mu_0(du).
\end{split}
\end{equation}
Thus the \as{posterior}\index{posteruir} is absolutely continuous with respect to
the \as{prior}\index{prior}, and the Radon-Nikodym derivative is proportional
to the \as{likelihood}\index{likelihood}. 
This is rewriting \as{Bayes' formula}\index{Bayes' formula} in the form
$$\frac{1}{\P(u)}\P(u|y)=\frac{1}{\P(y)}\P(y|u).$$

\item The expression for the Radon-Nikodym
derivative is to be interpreted as the statement that, for
all measurable $f: \R^n \to \R$,
$$\bbE^{\mu^y}f(u)=\bbE^{\mu_0}\Big(\frac{d\mu^y}{d\mu_0}(u)f(u)\Big).$$
Alternatively we may write this in integral form as
\begin{align*}
\int_{\R^n} f(u)\mu^y(du)
&=\int_{\R^n}\Bigl(\frac{1}{Z}\exp\bigl(-\Phi(u;y)\bigr)f(u)\Bigr)\mu_0(du)\\
&=\frac{\int_{\R^n}\exp\bigl(-\Phi(u;y)\bigr)f(u)\mu_0(du)}{\int_{\R^n}\exp\bigl(-\Phi(u;y)\bigr)\mu_0(du)}.
\end{align*}
\end{itemize}
\qed
\end{remarks}

\vspace{0.2in}
\subsection{Inverse Heat Equation}
\label{ssec:1.2}

This inverse problem illustrates the first difficulty, labelled
$1$.\ in the previous subsection,
which motivates the Bayesian approach to inverse problems.
Let $D\subset\R^d$ be a bounded open set, with Lipschitz  
boundary $\partial D$. Then define the Hilbert space $H$ and
operator $A$ as follows:
\begin{eqnarray*}
&&H=\Bigl(L^2(D),\langle\cdot,\cdot\rangle, \|\cdot\|\Bigr); \\
&&A=-\triangle, \quad \CD(A)=H^2(D)\cap H_0^1(D).
\end{eqnarray*}
We make the following assumption about the spectrum of $A$ which
is easily verified for simple geometries, but in fact holds quite
generally.

\begin{assumption}
\label{a:1.3}
The eigenvalue problem
\begin{equation*}
A\varphi_j=\alpha_j\varphi_j,
\end{equation*}
has a countably infinite set of solutions, indexed by
$j\in\mathbb{Z}^+$. They may be normalized to
satisfy the $L^2$-orthonormality condition
\begin{eqnarray*}
\langle\varphi_j,\varphi_k\rangle=\left\{
  \begin{array}{ll}
    1, \quad & j=k\\
    0, \quad & j\neq k,
  \end{array}
\right.
\end{eqnarray*}
and form a basis for $H$.
Furthermore, the eigenvalues are positive and, if ordered to be increasing,
satisfy $\alpha_j\asymp j^{\frac{2}{d}}$.$\qed$
\end{assumption}

Here and in the remainder of the notes, the notation $\asymp$ denotes the
existence of constants $C^{\pm}>0$ such that 
\begin{equation}
\label{eq:quid}
C^-j^{2/d} \le \alpha_j \le C^+ j^{2/d}
\end{equation}
for all $j \in \N$.

Any $w\in H$ can be written as
$$w=\sum_{j=1}^{\infty} \langle w, \varphi_j \rangle \varphi_j$$
and we can define the \as{Hilbert scale} of spaces\index{Hilbert scale} $\mathcal {H}^t=\CD(A^{t/2})$
as explained in Section \ref{sssec:sob} for any $t>0$ and with the norm
$$\|w\|_{\mathcal {H}^t}^2=\sum_{j=1}^\infty j^{\frac{2t}{d}}|w_j|^2$$
where $w_j=\langle w,\varphi_j \rangle$.

Consider the heat conduction equation on $D$, with Dirichlet
boundary conditions, writing it as an ordinary differential equation
in $H$:
\begin{equation}\label{13}
\frac{dv}{dt}+Av=0,\quad v(0)=u.
\end{equation}
We have the following:

\begin{lemma}
\label{l:1.4} Let Assumption \ref{a:1.3} hold.
Then for every $u\in H$ and every $s>0$
there is a unique solution $v$ of
equation \eqref{13} in the space $\CC([0,\infty); H)\cap \CC((0,\infty);\mathcal {H}^s)$. We write $v(t)=\exp(-At)u$.
\end{lemma}

To motivate this statement, and in particular the high degree
of regularity seen at each fixed $t$, we argue as follows.
Note that, if the initial condition is expanded in the
eigenbasis as
$$u=\sum\limits_{j=1}^\infty u_j\varphi_j, \quad u_j=\langle u,
\varphi_j\rangle,$$ then the solution of \eqref{13} has the form
$$v(t)=\sum\limits_{j=1}^\infty u_j
e^{-\alpha_j t}\varphi_j.$$
Thus 
\begin{align*}
\|v(t)\|_{\mathcal {H}^s}^2 & = \sum\limits_{j=1}^\infty j^{2s/d}e^{-2\alpha_j t}|u_j|^2
\asymp \sum\limits_{j=1}^\infty \alpha_j^s e^{-2\alpha_j t}|u_j|^2\\
& = t^{-s}\sum\limits_{j=1}^\infty (\alpha_j t)^s e^{-2\alpha_j t}|u_j|^2
\le Ct^{-s}\sum\limits_{j=1}^\infty |u_j|^2\\
& = Ct^{-s}\|u\|_{H}^2.
\end{align*}
It follows that $v(t) \in \mathcal {H}^s$ for any $s>0$, provided
$u \in{H}$.

We are interested in the inverse
problem of finding $u$ from $y$ where
$$y=v(1)+\eta = G(u)+\eta = e^{-A}u+\eta\;.$$
Here $\eta\in H$ is noise and $G(u):=v(1)=e^{-A}u$. Formally this
looks like an infinite dimensional linear version of the inverse problem
\eqref{11}, extended from finite dimensions to a Hilbert space
setting. However, the infinite dimensional setting throws up significant
new issues. To see this, assume that there is $\beta_c>0$ such that
$\eta$ has regularity $\mathcal {H}^\beta$ if and only if 
$\beta<\beta_c$. Then $y$ is not in the image space
of $G$ which is, of course, contained in $\cap_{s>0} \mathcal {H}^s$.
Applying the formal inverse of $G$ to $y$ results in an object which is
not in $H$.

To overcome this problem,
we will apply a Bayesian approach and hence will need to put probability
measures on the Hilbert space $H$; in particular we will want to study
$\P(u)$, $\P(y|u)$ and
$\P(u|y)$, all probability measures on $H$.

\vspace{0.2in}
\subsection{Elliptic Inverse Problem}
\label{ssec:1.3}

One motivation for adopting the Bayesian approach to inverse
problems is that \as{prior}\index{prior} modelling is a transparent approach
to dealing with under-determined inverse problems; it 
forms a rational approach to dealing with  
the second difficulty, labelled $2$. in subsection \ref{ssec:1.1}.
The elliptic inverse problem we now describe is a concrete
example of an under-determined inverse problem.

As in subsection \ref{ssec:1.2},
$D\subset\R^d$ denotes a bounded open set,
with Lipschitz boundary $\partial D$.
We define the Gelfand triple of Hilbert spaces $V\subset H\subset V^*$ 
by
\begin{equation}
H=\bigl(L^2(D),\langle\cdot,\cdot\rangle, \|\cdot\|\bigr)\;,\qquad 
V=\bigl(H_0^1(D),\langle \nabla \cdot,\nabla \cdot\rangle, \|\cdot\|_V=\|\nabla\cdot\|\bigr) \;. 
\end{equation}
and $V^*$ the dual of $V$ with respect to the pairing induced by $H$.
Note that $\|\cdot\|\leq C_{\mathsf p}\|\cdot\|_V$ for some constant $C_{\mathsf p}$: the Poincar\'e inequality.

Let $\kappa\in X:=L^\infty(D)$ satisfy
\begin{equation}\label{14}
\text{ess}\inf_{x\in D}\kappa(x)=\kappa_{\min}>0.
\end{equation}

Now consider the equation
\begin{subequations}
\label{15}
\begin{align}
-\nabla\cdot (\kappa\nabla p)&=f,\quad x\in D, \label{15a}
\\
p&=0,\quad x\in \partial D.\label{15b}
\end{align}
\end{subequations}
Lax-Milgram theory yields the following:

\begin{lemma} \label{l:1.5}
Assume that $f\in V^*$ and that $\kappa$ satisfies \eqref{14}. Then
(\ref{15}) has a unique weak solution $p\in V$. This solution
satisfies
\begin{equation*}
\|p\|_V\leq \|f\|_{V^*}/\kappa_{\min}
\end{equation*}
and, if $f\in H$,
\begin{equation*}
\|p\|_V\leq C_p\|f\|/\kappa_{\min}.
\end{equation*}
\end{lemma}

We will be interested in the inverse problem of finding $\kappa$
from $y$ where
\begin{equation}
\label{16} y_j=l_j(p)+\eta_j,\quad j=1,\cdots, J.
\end{equation}
Here $l_j\in V^*$ is a continuous linear functional on $V$ and $\eta_j$ is a
noise.

Notice that the unknown, $\kappa\in X$, is a function
(infinite dimensional) whereas the data from which we wish to
determine $\kappa$ is finite dimensional: $y\in\R^J$. The problem is severely under-determined,
illustrating point $2$. from subsection \ref{ssec:1.1}.
One way to treat such problems is by adopting  
the Bayesian framework, using \as{prior}\index{prior} 
modelling to fill-in missing
information. We will take the unknown function to be $u$ where
either $u=\kappa$ or $u=\log \kappa$. In either case, we will define
$G_j(u)=l_j(p)$ and, noting that $p$ is then a nonlinear function
of $u$, \eqref{16} may be written as
\begin{equation}
\label{17} y=G(u)+\eta
\end{equation}
where $y,\eta \in \R^J$ and $G: X^{+} \subseteq X \to \R^J$.
The set $X^{+}$ is introduced because
$G$ may not be defined on the whole of $X$. In particular, 
the positivity constraint \eqref{14} is only satisfied 
on 
\begin{equation}
X^{+} :=\Bigl\{u \in X: \text{ess}\inf_{x\in D}u(x)>0 \Bigr\}\subset X
\label{eq:Xp}
\end{equation}
in the case where $\kappa=u$.
On the other hand if $\kappa=\exp(u)$ then the positivity constraint \eqref{14}
is satisfied for any $u \in X$ and we may take $X^+=X.$ 

Notice that we again need probability measures on function space,
here the Banach space $X=L^\infty(D)$.
Furthermore, in the case where $u=\kappa$, these probability measures should
charge only positive functions, in view of the desired inequality
(\ref{14}). Probability on Banach spaces of
functions is most naturally developed in the setting of \as{separable
spaces}\index{separable space}, 
which $L^\infty(D)$ is not. This difficulty can be circumvented
in various different ways as we describe in what follows.

\vspace{0.2in}
\subsection{Bibliographic Notes}

\begin{itemize}

\item Subsection \ref{ssec:1.1}. See \cite{bs94} for a general overview of
the Bayesian approach to statistics in the finite dimensional
setting. The Bayesian approach to linear inverse problems
with Gaussian noise and \as{prior}\index{prior} in finite dimensions is
discussed in ~\cite[Chapters
2 and 6]{article:Stuart2010} and, with a more algorithmic flavour,
in the book ~\cite{KS05}.

\item Subsection \ref{ssec:1.2}. For details on the heat equation 
as an ODE in Hilbert space, and the 
regularity estimates of Lemma \ref{l:1.4}, see \cite{pazy,lunardi95}.
The classical approach to linear inverse problems is described
in numerous books; see, for example, \cite{kir96,ehn96}. The
case where the spectrum of the forward map $G$ decays exponentially,
as arises for the heat equation, is sometimes termed
{\em severely ill-posed}.
The Bayesian approach to linear inverse problems was developed
systematically in \cite{Man84,LPS89}, following from the seminal
paper \cite{Fr70} in which the approach was first described;
for further reading on
ill-posed linear problems see \cite[Chapters 3 and 6]{article:Stuart2010}.
Recovering the truth underlying the data from the Bayesian
approach, known as {\em Bayesian posterior consistency}, is the
topic of \cite{KVZ11a,ALS12}; generalizations to severely
ill-posed problems, such as the heat equation, may be found
in \cite{KVZ11b,ASZ12}.

\item Subsection \ref{ssec:1.3}. See \cite{evans}
for the Lax-Milgram theory which gives rise to Lemma
\ref{l:1.5}.  For classical inversion theory for the elliptic
inverse problem -- determining the permeability from the pressure in
a Darcy model of flow in a porous medium -- see \cite{richter81,BK};
for Bayesian formulations see \cite{DS11,DHS12}. For posterior
consistency results see \cite{Vol13}.

\end{itemize}

\vspace{0.3in}
\section{\as{Prior}\index{prior} Modeling}
\label{sec:prior}

In this section we show how to construct probability measures on a function
space, adopting a constructive approach based on random series. As explained
in section 6.2, the natural setting for probability in a function space is that
of a separable Banach space. A countable infinite sequence in the Banach space
$X$ will be used for our random series; in the case where $X$ is not separable 
the resulting probability measure will be constructed on a separable subspace 
$X'$ of $X$ (see the discussion in subsection 2.1).

Subsection \ref{ssec:2.1} describes this general setting, and subsections
\ref{ssec:2.2}, \ref{ssec:2.3} and \ref{ssec:2.4} consider, in turn,
three classes of \as{priors}\index{prior!uniform}\index{prior!Besov}\index{prior!Gaussian} termed uniform, Besov and Gaussian. In
subsection \ref{ssec:grf} we link the random series construction to the
widely used random field perspective on spatial stochastic
processes and we summarize in subsection \ref{ssec:2.5}.
We denote the \as{prior}\index{prior} 
measures constructed in this section by $\mu_0.$

\subsection{General Setting}
\label{ssec:2.1}

We let $\{\phi_j\}_{j=1}^\infty$ denote an infinite sequence in the
Banach space $X$, with norm $\|\cdot\|$, of $\R$-valued
functions defined on a domain $D$. We will either take
$D\subset\R^d$, a bounded, open set
with Lipschitz boundary, or $D=\bbT^d$ the $d$-dimensional torus.
We normalize these functions so that
$\|\phi_j\|=1$ for $j=1,\cdots,\infty$. We also introduce
another element $m_0 \in X$, not necessarily normalized to $1$. 
Define the function $u$ by
\begin{equation}\label{21}
u=m_0+\sum_{j=1}^\infty u_j\phi_j.
\end{equation}
By randomizing $\su:=\{u_j\}_{j=1}^\infty$ we create real-valued
\as{random functions}\index{random function} on $D$. 
(The extension to $\R^n$-valued
random functions is straightforward, but omitted for brevity.)

We now define the deterministic sequence
$\gamma=\{\gamma_j\}_{j=1}^\infty$ and the \as{i.i.d.\ random sequence}\index{i.i.d.\ sequence}
$\xi=\{\xi_j\}_{j=1}^\infty$, and set $u_j=\gamma_j\xi_j$. We assume that
$\xi$ is centred, i.e. that it has mean zero. Formally we
see that the average value of $u$ is then $m_0$ so that this element
of $X$ should be thought of as the {\em \as{mean function}}\index{mean function}. We
assume that $\gamma \in \ell_w^p$ for some $p \in [1,\infty)$ and
some positive weight sequence $\{w_j\}$ (see subsection \ref{sssec:elp}). 
We define $\Omega=\R^{\infty}$ and view $\xi$ as a random element 
in the probability space 
$\bigl(\Omega,\cBc(\Omega),\P\bigr)$ 
of \as{i.i.d.\ sequences}\index{i.i.d.\ sequence} 
equipped with the product $\sigma$-algebra; we let $\bbE$ denote expectation.
This sigma algebra can be generated by cylinder sets 
if an appropriate distance $d$ is defined
on sequences. However the distance $d$ captures
nothing of the properties of the \as{random function}\index{random function}
$u$ itself. For this reason
we will be interested in the pushforward of the measure $\P$
on the measure space $\bigl(\Omega,\cBc(\Omega)\bigr)$
into a measure $\mu$ on  $\bigl(X',\cBc(X')\bigr)$,
where $X'$ is a separable\index{separable space} Banach space and
$\cBc(X')$ denotes its Borel $\sigma$-algebra.
Sometimes $X'$ will be the same as $X$ but not always: 
the space $X$ may not be separable; and, although we have stated the
normalization of the $\phi_j$ 
in $X$, they may of course live in smaller spaces $X'$, and $u$ may do so too.
For either of these reasons $X'$ may be a proper subspace of $X$.

In the next three subsections we
demonstrate how this general setting may be adapted to create a
variety of useful \as{prior}\index{prior} measures on function space; the fourth subsection,
which follows these three, relates the random series construction, in the
Gaussian case, to the standard construction of Gaussian random fields.
We will express many of our results in terms of the probability measure $\P$
on i.i.d sequences, but all such results will, of course, have direct 
implications for the induced pushforward measures on the function spaces where 
the random functions $u$ live. We discuss this perspective in the summary 
section \ref{ssec:2.5}.  In dealing with the random series construction we will 
also find it useful to consider the truncated random functions
\begin{equation}\label{23a}
u^N=m_0+\sum_{j=1}^N u_j\phi_j, \quad u_j=\gamma_j\xi_j.
\end{equation}

\vspace{0.2in}
\subsection{Uniform \as{Prior}s\index{prior!uniform}}
\label{ssec:2.2}

To construct the \as{random function}s\index{random function} \eqref{21} we take 
$X=L^\infty(D)$, choose the deterministic sequence 
$\gamma=\{\gamma_j\}_{j=1}^\infty\in \ell^1$
and specify the i.i.d.\ sequence $\xi=\{\xi_j\}_{j=1}^\infty$ by $\xi_1\sim
U[-1,1]$, uniform random variables on $[-1,1].$ 
Assume further that there are finite, strictly positive constants $m_{\min}$,
$m_{\max}$, and $\delta$ such that
\begin{eqnarray*}
&&\text{ess}\inf_{x\in D}m_0(x)\geq m_{\min};\\
&&\text{ess}\sup_{x\in D}m_0(x)\leq m_{\max};\\
&&\|\gamma\|_{\ell^1}=\frac{\delta}{1+\delta}m_{\min}.
\end{eqnarray*}
The space $X$ is not \as{separable}\index{separable space} and so, instead, we work with
the space $X'$ found as the closure of the linear span of
the functions $(m_0,\{\phi_j\}_{j=1}^{\infty})$ with respect to
the norm $\|\cdot\|_{\infty}$ on $X$. The Banach space $\bigl(X',\|\cdot\|_{\infty}\bigr)$
is separable.\index{separable space}

\begin{theorem}
\label{t:2.1}
The following holds $\P$-almost
surely: the sequence of functions $\{u^N\}_{N=1}^\infty$ given by
\eqref{23a} is Cauchy in $X'$ and the limiting function $u$ given by (\ref{21})
satisfies
\begin{equation*}
\frac{1}{1+\delta}m_{\min}\leq
u(x)\leq m_{\max}+\frac{\delta}{1+\delta}m_{\min} \quad
\text{a.e.} \quad x\in D.
\end{equation*}
\end{theorem}

\noindent{\bf Proof.} Let $N>M$. Then, $\P$-a.s.,
\begin{align*}
\|u^N-u^M\|_{\infty}&=\Bigl\|\sum_{j=M+1}^N u_j\phi_j\Bigr\|_{\infty}\\
& \le \Bigl\|\sum_{j=M+1}^N \gamma_j \xi_j \phi_j\Bigr\|_{\infty}\\
& \le \sum_{j=M+1}^\infty |\gamma_j||\xi_j|\|\phi_j\|_{\infty}\\
& \le \sum_{j=M+1}^\infty |\gamma_j|.
\end{align*}
The right hand side tends to zero as $M \to \infty$ by the dominated
convergence theorem and hence the sequence is Cauchy in $X'$.

We have $\P$-a.s.\ and
for a.e.\ $x\in D$,
\begin{eqnarray*}
u(x)&\geq&m_0(x)-\sum_{j=1}^\infty |u_j|\|\phi_j\|_\infty\\
&\geq&\text{ess}\inf_{x\in D}m_0(x)-\sum_{j=1}^\infty
|\gamma_j|\\
&\geq&m_{\min}-\|\gamma\|_{\ell^1}\\
&=&\frac{1}{1+\delta}m_{\min}.
\end{eqnarray*}
Proof of the upper bound is similar. $\Box$\\

\begin{example} 
Consider the random function (\ref{21}) as
specified in this section. By Theorem \ref{t:2.1}
we have that, $\P$-a.s.,
\begin{equation}\label{22}
u(x)\geq\frac{1}{1+\delta}m_{\min}>0, \quad \text{a.e.}\quad x \in D.
\end{equation}
Set $\kappa=u$ in the elliptic equation (\ref{14}), so that
the coefficient $\kappa$ in the equation and the
solution $p$ are random variables on
$\bigl(\R^{\infty},{\mathcal B}(\R^{\infty}),\P\bigr)$.
Since (\ref{22}) holds $\P$-a.s., Lemma \ref{l:1.5} shows that,
again $\P$-a.s.,
\begin{equation*}
\|p\|_V\leq (1+\delta)\|f\|_{V^*}/m_{\min}.
\end{equation*}
Since the r.h.s.\ is non-random we have that for all $r\in
\mathbb{Z}^+$ the random variable $p\in L^r_{\P}(\Omega;V)$:
\begin{equation*}
\bbE\|p\|_V^r<\infty.
\end{equation*}
In fact $\bbE\exp(\alpha\|p\|_V^r)<\infty$ for all $r\in
\mathbb{Z}^+$ and $\alpha\in (0,\infty)$. $\quad\Box$
\label{ex:lpb}
\end{example}

We now consider the situation where the family $\{\phi_j\}_{j=1}^{\infty}$ 
have a uniform H\"older exponent
$\alpha$ and study the implications for H\"older continuity of the random
function $u$. Specifically we assume that 
there are $C,a>0$ and $\alpha \in (0,1]$
such that, for all $j \ge 1$, 
\begin{equation}
\label{eq:phih}
|\phi_j(x)-\phi_j(y)| \le Cj^{a}|x-y|^{\alpha},\, x,y \in D.
\end{equation}
and
\begin{equation}
\label{eq:phih22}
|m_0(x)-m_0(y)| \le C|x-y|^{\alpha},\, x,y \in D.
\end{equation}

\begin{theorem}
Assume that $u$ is given by \eqref{21} where the collection of
functions $(m_0,\{\phi_j\}_{j=1}^{\infty})$ satisfy \eqref{eq:phih} and
\eqref{eq:phih22}.
Assume further that
$\sum_{j=1}^{\infty}|\gamma_j|^2 j^{a\theta}<\infty$ for some
$\theta\in (0,2)$. Then $\P$-a.s. we have $u \in \CC^{0,\beta}(D)$ for all
$\beta<\frac{\alpha \theta}{2}$. 
\end{theorem}

\begin{proof}
This is an application of Corollary \ref{cor:Kol} of the \as{Kolmogorov 
continuity}
theorem\index{Kolmogorov!continuity criterion} 
and $S_1$ and $S_2$ are as defined there.
We use $\theta$ in place of the parameter  $\delta$ 
appearing in Corollary \ref{cor:Kol} in order to avoid 
confusion with $\delta$ appearing in Theorem \ref{t:2.1} above
and in (\ref{eq:tg}) below.
Note that, since $m_0$ has assumed H\"older regularity $\alpha$,
which exceeds $\frac{\alpha \theta}{2}$ since $\theta\in (0,2)$, 
it suffices to consider the centred case where $m_0\equiv 0$.
We let $f_j=\gamma_j\phi_j$ and complete the proof by noting that
$$S_1=\sum_{j=1}^{\infty} |\gamma_j|^2 \le S_2 \le \sum_{j=1}^{\infty}
|\gamma_j|^2 j^{a\theta}<\infty.$$

\end{proof}

\begin{example}
\label{ex:ihp}
Let $\{\phi_j\}$ denote the Fourier basis for $D=[0,1]^d$. Then
we may take $a=\alpha=1$. If $\gamma_j=j^{-s}$ then $s>1$ ensures
$\gamma \in \ell^1$. Furthermore
$$\sum_{j=1}^{\infty}|\gamma_j|^2 j^{a\theta}=\sum_{j=1}j^{\theta-2s}<\infty$$
for $\theta<2s-1$. We thus deduce that $u \in \CC^{0,\beta}([0,1]^d)$
for all $\beta<\min\{s-\frac12,1\}$.
\end{example}

\vspace{0.2in}
\subsection{Besov \as{Prior}s\index{prior!uniform}}
\label{ssec:2.3}

For this construction of \as{random functions}\index{random function} 
we take $X$ to be the Hilbert space
$$X:=\dot{L}^2(\mathbb{T}^d)=\Bigl\{u:\bbT^d \to \R\Big| \int_{\bbT^d} |u(x)|^2 dx<\infty, \int_{\bbT^d} u(x)dx=0\Bigr\}$$
of real valued periodic functions in dimension $d \le 3$
with inner-product and norm denoted by
$\langle \cdot, \cdot \rangle$ and $\|\cdot\|$ respectively.
We then set $m_0=0$ and let $\{\phi_j\}_{j=1}^\infty$ be an
orthonormal basis for $X$. 
Consequently, for any $u \in X$, we have for a.e. $x \in \bbT^d$,
\begin{equation}\label{212}
u(x)=\sum_{j=1}^\infty u_j\phi_j(x), \quad u_j=\langle u,\phi_j\rangle.
\end{equation}
Given a function $u: \bbT^d \to \R$ and the $\{u_j\}$ as
defined in \eqref{212} we define the Banach space $X^{t,q}$ by
\begin{equation*}
X^{t,q}=\Bigl\{u:\bbT^d \to \R\Big|
\|u\|_{X^{t,q}}<\infty, 
\int_{\bbT^d} u(x)dx=0\Bigr\}
\end{equation*}
where
\begin{equation*}
\|u\|_{X^{t,q}}=\Big(\sum_{j=1}^\infty
j^{(\frac{tq}{d}+\frac{q}{2}-1)}|u_j|^q\Big)^{\frac{1}{q}}
\end{equation*}
with $q \in [1,\infty)$ and $s>0$.
If $\{\phi_j\}$ form the Fourier basis and $q=2$ then
$X^{t,2}$ is the Sobolev space ${\dot H}^t(\bbT^d)$ of mean-zero
periodic functions with $t$ (possibly non-integer)
square-integrable derivatives; in particular 
$X^{0,2}=\dot{L}^2(\mathbb{T}^d)$.
On the other hand,
if the $\{\phi_j\}$ form certain wavelet bases, then
$X^{t,q}$ is the Besov space $B^{t}_{qq}$.

As described above, we assume that $u_j=\gamma_j\xi_j$ where
$\xi=\{\xi_j\}_{j=1}^\infty$ is an i.i.d.\ sequence and
$\gamma=\{\gamma_j\}_{j=1}^\infty$ is deterministic. Here we assume
that $\xi_1$ is drawn from the centred  measure on $\R$ with
density proportional to
$\exp\big(-\frac{1}{2}|x|^q\big)$ for some $1\leq q<\infty$ -- we
refer to this as a {\em $q$-exponential distribution}, noting that
$q=2$ gives a Gaussian and $q=1$ a Laplace-distributed random
variable. Then for $s>0$ and $\delta>0$ we define
\begin{equation}
\gamma_j=j^{-(\frac{s}{d}+\frac{1}{2}-\frac{1}{q})}\big(\frac{1}{\delta}\big)^{\frac{1}{q}}.
\label{eq:tg}
\end{equation}
The parameter $\delta$ is a key scaling parameter  which will appear
in the statement of exponential moment bounds below.

We now prove convergence of the
series (found from \eqref{23a} with $m_0=0$)
\begin{equation}\label{23}
u^N=\sum_{j=1}^N u_j\phi_j, \quad u_j=\gamma_j\xi_j
\end{equation}
to the limit function
\begin{equation}
u(x)=\sum_{j=1}^\infty u_j\phi_j(x), \quad u_j=\gamma_j\xi_j,
\label{eq:KL}
\end{equation}
in an appropriate space. To understand the sequence of functions $\{u^N\}$,
it is useful to introduce the following function space:
$$L^q_{\P}(\Omega;X^{t,q}):=\Big\{v:D \times \Omega  \to \R\Big|
\EE\bigl(\|v\|_{X^{t,q}}^{q}\bigr)<\infty\Big\}.$$
This is a Banach space, 
when equipped with the
norm $\Bigl(\EE\bigl(\|v\|_{X^{t,q}}\bigr)^q\Bigr)^{\frac{1}{q}}$.
Thus every Cauchy sequence is convergent in this space.

\begin{theorem} \label{t:2.4a}
For $t<s-\frac{d}{q}$ the sequence of functions
$\{u^N\}_{N=1}^\infty$, given by (\ref{23}) and (\ref{eq:tg}) with $\xi_1$ drawn from a centred $q$-exponential 
distribution,
 is Cauchy in
the Banach space
$L^q_{\P}(\Omega;X^{t,q})$.
Thus the infinite series \eqref{eq:KL}
exists as an $L^q_\P$-limit and takes values in $X^{t,q}$ almost surely,
for all $t<s-\frac{d}{q}$.
\end{theorem}

\noindent{\bf Proof.} For $N>M$,
\begin{eqnarray*}
\bbE\|u^N-u^M\|_{X^{t,q}}^{q}&=&\delta^{-1}\bbE\sum_{j=M+1}^N
j^{\frac{(t-s)q}{d}}|\xi_j|^q\\
&\asymp&\sum_{j=M+1}^N j^{\frac{(t-s)q}{d}}\leq\sum_{j=M+1}^\infty
j^{\frac{(t-s)q}{d}}.
\end{eqnarray*}
The sum on the right hand side tends to $0$ as $M\rightarrow\infty$,
provided $\frac{(t-s)q}{d}<-1$, by the dominated convergence
theorem. This completes the proof.
$\Box$

The previous theorem gives a sufficient condition, on $t$, for
existence of the limiting random function. The following theorem
refines this to an if and only if statement, in the
context of almost sure convergence. 

\begin{theorem}
\label{t:2.2}
Assume that $u$ is given by \eqref{eq:KL} and \eqref{eq:tg}
with $\xi_1$ drawn from a centred $q$-exponential distribution.  
Then the following are equivalent:
\begin{itemize}
\item[] i) $\|u\|_{X^{t,q}}<\infty$ $\P$-a.s.;

\item[] ii) $\bbE\big(\exp(\alpha\|u\|_{X^{t,q}}^q)\big)<\infty$ for
any $\alpha\in [0,\frac{\delta}{2})$;

\item[] iii) $t<s-\frac{d}{q}$.
\end{itemize}
\end{theorem}

\noindent{\bf Proof.} We first note that, for the random function in question,
\begin{eqnarray*}
\|u\|_{X^{t,q}}^q=\sum_{j=1}^\infty
j^{(\frac{tq}{d}+\frac{q}{2}-1)}|u_j|^q=\sum_{j=1}^\infty
\delta^{-1}j^{-\frac{(s-t)q}{d}}|\xi_j|^q.
\end{eqnarray*}
Now, for $\alpha<\frac{1}{2}$,
\begin{eqnarray*}
\bbE\exp\big(\alpha|\xi_1|^q\big)&=& \int_{\R}\exp\Big(-
\bigl(\frac{1}{2}-\alpha\bigr)|x|^q\Big)dx\Big/\int_{\R}\exp\Big(-\frac{1}{2}|x|^q\Big)dx\\
&=&(1-2\alpha)^{-\frac{1}{q}}.
\end{eqnarray*}

iii) $\Rightarrow$ ii).
\begin{eqnarray*}
\bbE\Big(\exp(\alpha\|u\|_{X^{t,q}}^q)\Big)&=&\bbE\Big(\exp(\alpha\sum_{j=1}^\infty
\delta^{-1}j^{-\frac{(s-t)q}{d}}|\xi_j|^q)\Big)\\
&=&
\prod_{j=1}^{\infty}\Big(1-\frac{2\alpha}{\delta}j^{-\frac{(s-t)q}{d}}\Big)^{-\frac{1}{q}}.
\end{eqnarray*}
For $\alpha<\frac{\delta}{2}$ the product converges if
$\frac{(s-t)q}{d}>1$ i.e. $t<s-\frac{d}{q}$ as required.\\

ii) $\Rightarrow$ i).\\

If (i) does not hold, $Z:=\|u\|_{X^{t,q}}^q$ is positive infinite on a set of positive measure $S$.
Then, since for $\alpha>0$, $\exp(\alpha Z)=+\infty$ if $Z=+\infty$, and
$\bbE \exp(\alpha Z)\ge \bbE(\mathbbm{1}_S\exp(\alpha Z))$ we get a contradiction.
\\

i) $\Rightarrow$ iii).\\

To show that (i) implies (iii) note that (i) implies that, almost
surely,
$$\sum_{j=1}^{\infty}j^{(t-s)q/d}|\xi_j|^q<\infty.$$
This implies that $t<s$. To see this assume for
contradiction that $t \ge s$. Then, almost surely,
$$\sum_{j=1}^{\infty}|\xi_j|^q<\infty.$$
Since there is a constant $c>0$ with $\EE|\xi_j|^q =c$ for any $j\in\N$, this contradicts the
law of large numbers.

Now define $\zeta_j=j^{(t-s)q/d}|\xi_j|^q$. Using the fact that the
$\zeta_j$ are non-negative and independent we deduce from Lemma
\ref{l:2.3} (below) that
$$\sum_{j=1}^{\infty}\EE\big(\zeta_j \wedge 1\big) =\sum_{j=1}^{\infty}
\EE\Big(j^{(t-s)q/d}|\xi_j|^q\wedge1\Big)<\infty.$$ 
Since $t<s$ we note that then
\begin{eqnarray*}
\EE\zeta_j&=&\EE\Big(j^{-(s-t)q/d}|\xi_j|^q \Big)\\
&=& \EE\Big(j^{-(s-t)q/d}|\xi_j|^q \mathbb{I}_{\{|\xi_j|\leq
j^{(s-t)/d}\}}\Big) +\EE\Big(j^{-(s-t)q/d}|\xi_j|^q \mathbb{I}_{\{|\xi_j| > j^{(s-t)/d}\}}\Big)\\
&\leq& \EE\Big(\big(\zeta_j \wedge 1\big) \mathbb{I}_{\{|\xi_j| \le j^{{(s-t)}/d}\}}\Big)+I\\
&\leq&\EE\Big(\zeta_j \wedge 1\Big)+I,
\end{eqnarray*}
where
$$I \propto j^{-(s-t)q/d}\int_{j^{(s-t)/d}}^{\infty} x^q e^{-x^q/2} dx.$$
Noting that, since $q \geq 1$, the function $x \mapsto  x^q
e^{-x^q/2}$ is bounded, up to a constant of proportionality, by the
function $x \mapsto e^{-\alpha x}$ for any $\alpha<\frac12$, we see
that there is a positive constant $K$ such that
\begin{eqnarray*}
I&\leq& Kj^{-(s-t)q/d}\int_{j^{(s-t)/d}}^{\infty} e^{-\alpha x}dx\\
&=& \frac{1}{\alpha}Kj^{-(s-t)q/d} \exp\big(-\alpha j^{(s-t)/d}\big)\\
&:=&\iota_j.
\end{eqnarray*}
Thus we have shown that
\begin{eqnarray*}
\sum_{j=1}^{\infty} \bbE \Big(j^{-(s-t)q/d}|\xi_j|^q \Big)
\leq \sum_{j=1}^\infty \bbE\Big(\zeta_j \wedge
1\Big)+\sum_{j=1}^{\infty} \iota_j
<\infty.
\end{eqnarray*}
Since the $\xi_j$ are i.i.d.\ this implies that
$$\sum_{j=1}^{\infty}j^{(t-s)q/d}<\infty,$$
from which it follows that $(s-t)q/d>1$ and (iii) follows. $\Box$

\begin{lemma}
\label{l:2.3}
Let $\{I_j\}_{j=1}^\infty$ be an
independent sequence of $\R^+$-valued random variables. Then
\begin{equation*}
\sum_{j=1}^\infty I_j<\infty \quad \text{a.s.} \Leftrightarrow
\sum_{j=1}^\infty\bbE(I_j\wedge 1)<\infty.
\end{equation*}
\end{lemma}

As in the previous subsection, we now study the situation where the 
family $\{\phi_j\}$ have a uniform H\"older exponent
$\alpha$ and study the implications for H\"older continuity of the random
function $u$. In this case, however, the basis functions are normalized
in $\dot{L}^2$ and not $L^{\infty}$; thus we must make additional
assumptions on the possible growth of the $L^{\infty}$ norms of $\{\phi_j\}$
with $j$. We assume that there are $C,a,b>0$ and $\alpha \in (0,1]$
such that, for all $j \ge 0$, 
\begin{subequations}
\label{eq:phih2}
\begin{align}
|\phi_j(x)|& =\beta_j \le  Cj^{b},\, x \in D.\\
|\phi_j(x)-\phi_j(y)| &\le  Cj^{a}|x-y|^{\alpha},\, x,y \in D.
\end{align}
\end{subequations}
We also assume that $a > b$ as, since $\|\phi_j\|_{L^2}=1$, 
it is natural that the pre-multiplication
constant in the H\"older estimate on the $\{\phi_j\}$ grows in $j$
at least as fast as the bound on the functions themselves.

\begin{theorem}\label{t:bcty1}
Assume that $u$ is given by \eqref{eq:KL} and \eqref{eq:tg}
with $\xi_1$ drawn from a centred $q$-exponential distribution.  
Suppose also that \eqref{eq:phih2}
hold and that $s>d\bigl(b+q^{-1}+\frac12\theta(a-b)\bigr)$ for some 
$\theta \in (0,2)$. Then $\P$-a.s.\ we have $u \in \CC^{0,\beta}(\bbT^d)$ for all
$\beta<\frac{\alpha \theta}{2}$. 
\end{theorem}

\begin{proof}
We apply Corollary \ref{cor:Kol} 
of the \as{Kolmogorov continuity theorem}
\index{Kolmogorov!continuity criterion}
and $S_1$ and $S_2$ are as defined there.
We use $\theta$ in place of the parameter $\delta$ 
appearing in Corollary \ref{cor:Kol} in order to avoid 
confusion with $\delta$ appearing in Theorem \ref{t:2.1} and (\ref{eq:tg})
above.  
Let $f_j=\gamma_j\phi_j$ and note that
\begin{subequations}
\begin{align*}
S_1 &=\sum_{j=1}^{\infty} |\gamma_j|^2 \beta_j^2 \lesssim \sum_{j=1}^{\infty}
j^{-c_1}\\
S_2 &=\sum_{j=1}^{\infty} |\gamma_j|^{2-\theta} \beta_j^{2-\theta} 
\gamma_j^{\theta}j^{a\theta} \lesssim \sum_{j=1}^{\infty}
j^{-c_2}.
\end{align*}
\end{subequations}
Short calculation shows that
\begin{subequations}
\begin{align*}
c_1 & =\frac{2s}{d}+1-\frac{2}{q}-2b,\\
c_2 & =\frac{2s}{d}+1-\frac{2}{q}-2b-\theta(a-b).
\end{align*}
\end{subequations}
We require $c_1>1$ and $c_2>1$ and since $a>b$ satisfaction of
the second of these will imply the first. Satisfaction of the second
gives the desired lower bound on $s$.
\end{proof}

We note that the result of Theorem \ref{t:bcty1} holds true when the 
\as{mean function}\index{mean function} is nonzero if it satisfies 
\begin{subequations}
\begin{align*}
|m_0(x)|& \le  C,\, x \in D.\\
|m_0(x)-m_0(y)| &\le  C|x-y|^{\alpha},\, x,y \in D.
\end{align*}
\end{subequations}

We have the following sharper result if  the family $\{\phi_j\}$ is regular enough to be a basis for $B^t_{qq}$ instead of satisfying (\ref{eq:phih2}):

\begin{theorem}
Assume that $u$ is given by \eqref{eq:KL} and \eqref{eq:tg}
with $\xi_1$ drawn from a centred $q$-exponential distribution.  
Suppose also that $\{\phi_j\}_{j\in\N}$ form a basis for $B^t_{qq}$
for some $t<s-\frac{d}{q}$. Then $u \in \CC^{0,t}(\mathbb{T}^d)$ $\P$-almost surely.
\end{theorem}

\begin{proof}
For any $m\ge 1$,
using the definition of $X^{t,q}$-norm we can write
\begin{align*}
\|u\|_{B^t_{m q,m q}}^{m q}
&=(\tfrac{1}{\delta})^m
\sum_{j=1}^\infty j^{\frac{mq t}{d}+\frac{m q}{2}-1}
       j^{-m q(\frac{s}{d}+ \frac{1}{2} -\frac{1}{q})}|\xi_j|^{m q}.
\end{align*}
For every $m\in\N$ there exists a constant $C_m$ with $\bbE |\xi_j|^{m q}=C_m$. 
Since each term of the above series is measurable we can swap the sum and the integration and write
\begin{align*}
\bbE\|u\|_{B^t_{m q,m q}}^{m q}
=C_m(\tfrac{1}{\delta})^m
\sum_{j=1}^\infty j^{\frac{m q}{d}(t-s)+m-1}
\,\le\, \tilde{C}_m,
\end{align*}
noting that the exponent of $j$ is smaller than $-1$ (since $t<s-d/q$).
Now for a given $t<s-d/q$, one can choose $m$ 
large enough so that $\frac{d}{m q}<s-d/q-t$. 
Then the embedding $B^{t_1}_{m q,m q}\subset C^t$ 
for any $t_1$ satisfying $t+\frac{d}{m q}<t_1<s-d/q$ implies that
$\bbE\|u\|_{C^t(\mathbb{T}^d)}^{mq}<\infty$. It follows that $u \in C^t$
$\mathbb{P}$-almost surely.

\end{proof}

If the \as{mean function}\index{mean function} $m_0$ is $t$-H\"older continuous, the result of the above theorem holds for a random series with nonzero mean function as well.

\vspace{0.2in}
\subsection{Gaussian Priors}
\label{ssec:2.4}

Let $X$ be a Hilbert space
$\mathcal {H}$ of real-valued functions on bounded open $D \subset \R^d$
with Lipschitz boundary, and 
with inner-product and norm denoted by
$\langle \cdot, \cdot \rangle$ and $\|\cdot\|$ respectively;
for example $\mathcal {H}=L^2(D;\R)$.
Assume that $\{\phi_j\}_{j=1}^\infty$ is
an orthonormal basis for $\mathcal {H}$.
We study the Gaussian case where $\xi_1 \sim N(0,1)$ and then
equation \eqref{21} with $u_j=\gamma_j\xi_j$ generates random draws from the
Gaussian measure $N(m_0,\cC)$ on $\mathcal{H}$ where the 
\as{covariance operator}\index{covariance operator} 
$\cC$ depends on the sequence $\gamma=\{\gamma_j\}_{j=1}^\infty$. 
See the Appendix
for background on {Gaussian measures} in a Hilbert space. 
As in Section \ref{ssec:2.3},
we consider the setting in which $m_0=0$ so that the function
$u$ is given by \eqref{212} and has mean zero. We thus focus
on identifying $\cC$ from the random series \eqref{212}, and
studying the regularity of random draws from $N(0,\cC)$.

Define the \as{Hilbert scale}\index{Hilbert scale} of spaces $\mathcal {H}^t$
as in Subsection \ref{sssec:sob} with, recall, norm
$$\|u\|_{\mathcal {H}^t}^2=\sum_{j=1}^\infty j^{\frac{2t}{d}}|u_j|^2.$$
We choose $\xi_1\sim N(0,1)$ and study convergence of the
series \eqref{23} for $u^N$ 
to a limit function $u$ given by \eqref{eq:KL}; the spaces in which
this convergence occurs will depend upon the sequence $\gamma.$
To understand the sequence of functions $\{u^N\}$,
it is useful to introduce the following function space:
$$L^2_{\P}(\Omega;\mathcal{H}^t):=\Big\{v:D \times \Omega\to \R\Big|
\EE\bigl(\|v\|_{\cH^t}\bigr)^2<\infty\Big\}.$$
This is in fact a Hilbert space, although we will not use the Hilbert
space structure. We will only use the fact that $L^2_{\P}$ is a 
Banach space when equipped with the
norm $\Bigl(\EE\bigl(\|v\|_{\cH^t}^2\bigr)\Bigr)^{\frac12}$ and that hence
every Cauchy sequence is convergent.

\begin{theorem} \label{t:2.4b}
Assume that $\gamma_j\asymp j^{-\frac{s}{d}}$. 
Then the sequence of functions
$\{u^N\}_{N=1}^\infty$ given by \eqref{23} is Cauchy in
the Hilbert space
$L^2_{\P}(\Omega;\mathcal{H}^t)$,
$t<s-\frac{d}{2}$.
Thus the infinite series \eqref{eq:KL}
exists as an $L^2_\P$-limit and takes values in $\cH^t$ almost surely, for $t<s-\frac{d}{2}$.
\end{theorem}

\noindent{\bf Proof.} For $N>M$,
\begin{eqnarray*}
\bbE\|u^N-u^M\|_{\cH^t}^2&=&\bbE\sum_{j=M+1}^N
j^{\frac{2t}{d}}|u_j|^2\\
&\asymp&\sum_{j=M+1}^N j^{\frac{2(t-s)}{d}}\leq\sum_{j=M+1}^\infty
j^{\frac{2(t-s)}{d}}.
\end{eqnarray*}
The sum on the right hand side tends to $0$ as $M\rightarrow\infty$,
provided $\frac{2(t-s)}{d}<-1$, by the dominated convergence
theorem. This completes the proof.
$\Box$

\begin{remarks}
\label{rem:28}

We make the following remarks concerning the Gaussian random
functions constructed in the preceding theorem.

\begin{itemize}

\item The preceding theorem shows that the sum \eqref{23} has an
$L^2_{\PP}$ limit in $\cH^t$ when $t<s-d/2$,
as one can also see from the following direct calculation
\begin{eqnarray*}
\bbE\|u\|^2_{\cH^t}
&=&\sum_{j=1}^\infty j^{\frac{2t}{d}}\bbE(\gamma_j^2\xi_j^2)\\
&=&\sum_{j=1}^\infty j^{\frac{2t}{d}}\gamma_j^2\\
&\asymp&\sum_{j=1}^\infty j^{\frac{2(t-s)}{d}}<\infty.
\end{eqnarray*}
Thus $u\in {\mathcal H}^t$ a.s., for $t<s-\frac{d}{2}$.

\item From the preceding theorem we see that, provided $s>\frac{d}{2}$,
the random function in \eqref{eq:KL} generates a mean zero
Gaussian measure on $\mathcal {H}$. The expression \eqref{eq:KL}
is known as the {\em Karhunen-Lo\`eve expansion}, and
the eigenfunctions $\{\phi_j\}_{j=1}^{\infty}$ as the 
{\em Karhunen-Lo\`eve basis}. 

\item The \as{covariance operator}\index{covariance operator} $\mathcal{C}$
of a measure $\mu$ on $\cH$ may then be viewed as a bounded linear
operator from $\cH$ into itself defined to satisfy 
\begin{equation}
\label{e:propCmu2}
\mathcal{C} \ell =\int_{\cH} \langle \ell,u \rangle u \,\mu(du)\;,
\end{equation}
for all $\ell \in \cH$.
Thus
\begin{equation}
\label{e:defVar2}
\mathcal{C}=\int_{\cH} u \otimes u\,\mu(du)\;.
\end{equation}
The following formal
calculation, which can be made rigorous if $\mathcal{C}$ is
trace-class on $\cH$, gives an expression for the 
\as{covariance operator}\index{covariance operator}:
\begin{align*}
\mathcal{C}&=\bbE u\otimes u\\
&=\bbE\Bigl(\sum_{j=1}^\infty \sum_{k=1}^\infty \gamma_j\gamma_k
\xi_j \xi_k \phi_j\otimes\phi_k\Bigr)\\
&=\Bigl(\sum_{j=1}^\infty \sum_{k=1}^\infty \gamma_j\gamma_k
\bbE(\xi_j \xi_k)\phi_j\otimes\phi_k\Bigr)\\
&=\Bigl(\sum_{j=1}^\infty \sum_{k=1}^\infty \gamma_j\gamma_k
\delta_{jk}\phi_j\otimes\phi_k\Bigr)\\
& =\sum_{j=1}^\infty
\gamma_j^2\phi_j\otimes\phi_j.
\end{align*}
From this expression for the covariance, we may find
eigenpairs explicitly:
\begin{eqnarray*}
\mathcal{C}\phi_k&=&\Big(\sum_{j=1}^\infty
\gamma_j^2\phi_j\otimes\phi_j\Big)\phi_k\\
&=&\sum_{j=1}^\infty \gamma_j^2\langle\phi_j,\phi_k\rangle\phi_j
=\sum_{j=1}^\infty \gamma_j^2\delta_{jk}\phi_k =\gamma_k^2\phi_k.
\end{eqnarray*}

\item The Gaussian measure is denoted by $\mu_0:=N(0,\mathcal{C})$,
a Gaussian with \as{mean function}\index{mean function}
$0$ and \as{covariance operator}\index{covariance operator} $\mathcal{C}.$
The eigenfunctions of $\mathcal{C}$, $\{\phi_j\}_{j=1}^\infty$, are
known as the {\em Karhunen-Lo\`eve} basis for measure $\mu_0$. The
$\gamma_j^2$ are the eigenvalues associated with this eigenbasis,
and thus $\gamma_j$ is the standard deviation of the Gaussian
measure in the direction $\phi_j$.
\end{itemize}
\end{remarks}

In the case where $\mathcal
{H}=\dot{L}^2(\mathbb{T}^d)$ we are in the setting of
Section \ref{ssec:2.3} and we briefly consider this case.
We assume that the $\{\phi_j\}_{j=1}^\infty$
constitute the Fourier basis. Let $A=-\triangle$ denote
the negative Laplacian equipped with periodic boundary conditions on $[0,1)^d$,
and restricted to functions which integrate to zero over
$[0,1)^d$. This operator is positive self-adjoint and has eigenvalues
which grow like $j^{2/d}$, analogously to the Assumption \ref{a:1.3} 
made in the case of Dirichlet boundary conditions. 
It then follows that $\mathcal {H}^t=\CD(A^{t/2})=\dot{H}^t(\mathbb{T}^d)$, 
the Sobolev space of periodic
functions on $[0,1)^d$ with spatial mean equal to zero and $t$
(possibly negative or fractional) square integrable
derivatives. Thus, by the preceding Remarks \ref{rem:28},
$u$ defined by \eqref{eq:KL} is in
the space $\dot{H}^t$ a.s., $t<s-\frac{d}{2}$.
In fact we can say more about regularity, 
using the \as{Kolmogorov continuity test}\index{Kolmogorov!continuity criterion} and Corollary \ref{cor:Kolmogorov}; this
we now do.

\begin{theorem} 
\label{t:gh}
Consider the Karhunen-Lo\`eve expansion \eqref{eq:KL}
so that $u$ is a sample from the measure $N(0,\mathcal{C})$ 
in the case where $\mathcal{C}=A^{-s}$ with $A=-\triangle$,
$\CD(A)=\dot{H}^2(\mathbb{T}^d)$ and $s>\frac{d}{2}$. 
Then, $\P$-a.s., $u \in \dot{H}^t$, $t<s-\frac{d}{2}$, and 
$u \in \CC^{0,t}(\bbT^d)$ a.s., $t<1 \wedge (s-\frac{d}{2})$.
\end{theorem}

\begin{proof}
Because of the stated properties of the eigenvalues of
the Laplacian, it follows that the eigenvalues of $\mathcal{C}$
satisfy $\gamma_j^2\asymp
j^{-\frac{2s}{d}}$ and the eigenbasis $\{\phi_j\}$ is the Fourier
basis. Thus we may apply the conclusions stated
in Remarks \ref{rem:28} to deduce that $u\in \dot{H}^t$,
$t<\alpha-\frac{d}{2}$. 
Furthermore we may apply Corollary \ref{cor:Kol} to obtain H\"older
regularity of $u$. To do this we note that the $\{\phi_j\}$ are
bounded in $L^{\infty}(\bbT^d)$, and are Lipschitz with constants
which grow like $j^{1/d}$. We apply that corollary with $\alpha=1$
and obtain
$$S_1=\sum_{j=1}^\infty \gamma_j^2, \quad S_2=\sum_{j=1}^{\infty}
\gamma_j^2 j^{\delta/d}.$$
The corollary delivers the desired result after noting that
any $\delta<2s-d$ will make $S_2$, and hence $S_1$, summable.
\end{proof}

The previous example illustrates the fact that, although
we have constructed Gaussian measures in a Hilbert space setting,
and that they are naturally defined on a range of Hilbert
(Sobolev-like) spaces
defined through fractional powers of the Laplacian,
they may also be defined on Banach spaces, such as the
space of H\"older continuous functions.
We now return to the setting of the general domain $D$, rather
than the $d$-dimensional torus. In this general context 
it is important to highlight the Fernique Theorem, here restated from
the Appendix because of its importance:

\begin{theorem}[{\bf Fernique Theorem}]
\label{t:2.5}
Let $\mu_0$ be a
Gaussian measure on the separable  Banach space $X$. Then there exists
$\beta_c\in (0,\infty)$ such that, for all $\beta\in
(0,\beta_c)$,
\begin{equation*}
\bbE^{\mu_0}\exp\big(\beta\|u\|^2_X\big)<\infty.
\end{equation*}
\end{theorem}

\begin{remarks} \label{r:star}

We make two remarks concerning the Fernique Theorem.

\begin{itemize}

\item Theorem \ref{t:2.5}, when combined with Theorem \ref{t:gh},
shows that, with $\beta$ sufficiently small,
$\bbE^{\mu_0}\exp\big(\beta\|u\|^2_X\big)<\infty$
for both $X=\dot{H}^t$ and $X=\CC^{0,t}(\bbT^d)$,
if $t<s-\frac{d}{2}$. 

\item Let $\mu_0=N(0,A^{-s})$ where $A$ is as in Theorem \ref{t:gh}.
Then Theorem \ref{t:2.2} proves the Fernique Theorem \ref{t:2.5}
for $X=X^{t,2}=\dot{H}^t$, if $t<s-\frac{d}{2}$; the proof in the
case of the torus is very different from the general proof 
of the result in the abstract setting of Theorem \ref{t:2.5}.

\item Theorem \ref{t:2.2}ii) gives, in the Gaussian case, the Fernique 
Theorem in the case that $X$ is the Hilbert space $X^{t,2}$. Furthermore, 
the constant $\beta_c$ is specified explicitly in that setting. More explicit 
versions of the general Fernique Theorem \ref{t:2.5} 
are possible, but the characterization of $\beta_c$ is more involved.

\end{itemize}
\end{remarks}

\begin{example} 
\label{ex:bndg}
Consider the random function (\ref{21}) in
the case where $\cH=\dot{L}^2(\mathbb{T}^d)$ and $\mu_0=N(0,A^{-s})$, $s>\frac{d}{2}$ as in the preceding
example. Then we know that, $\mu_0$-a.s.,
$u\in \CC^{0,t}$, $t<1\wedge(s-\frac{d}{2})$. 
Set $\kappa=e^u$ in the elliptic PDE (\ref{15}) so that
the coefficient $\kappa$, and hence the solution $p$, 
are random variables on the probability space
$(\Omega, \mathcal {F},\P)$. Then
$\kappa_{\min}$ given in (\ref{13}) satisfies
\begin{equation*}
\kappa_{\min}\geq \exp\big(-\|u\|_\infty\big).
\end{equation*}
By Lemma \ref{l:1.5} we obtain
\begin{equation*}
\|p\|_V\leq\exp\big(\|u\|_\infty\big)\|f\|_{V^*}.
\end{equation*}
Since $\CC^{0,t} \subset L^\infty(\mathbb{T}^d)$, $t\in(0,1)$, we deduce that,
$$\|u\|_{L^{\infty}} \le K_1 \|u\|_{\CC^{0,t}}.$$
Furthermore, for any $\eps>0$, there
is constant $K_2=K_2(\eps)$ such that
$\exp(K_1 r x) \le K_2\exp(\eps x^2)$ for all $x \ge 0$.
Thus
\begin{align*}
\|p\|_V^r&\leq\exp\big(K_1r\|u\|_{\CC^{0,t}}\big)\|f\|_{V^*}^r\\
&\leq K_2\exp\big(\eps\|u\|^2_{\CC^{0,t}}\big)\|f\|_{V^*}^r.
\end{align*}
Hence, by Theorem \ref{t:2.5}, we deduce that
\begin{equation*}
\bbE\|p\|_V^r<\infty,\quad \text{i.e.}\quad p\in
L^r_{\P}(\Omega;V)\quad \forall\,\, r \in\mathbb{Z}^+.
\end{equation*}
This result holds for any $r \ge 0.$
Thus, when the coefficient of the elliptic PDE is {\em log-normal},
that is $\kappa$ is the exponential of a Gaussian function,
moments of all orders exist for the random variable $p$.
However, unlike the case of the uniform
\as{prior}\index{prior!uniform}, we cannot obtain exponential moments
on $\bbE\exp(\alpha\|p\|_V^r)$ for any $(r,\alpha)
\in \mathbb{Z}^+ \times (0,\infty)$. This is because the
coefficient $\kappa$, whilst positive a.s., does not satisfy a uniform
positive lower bound across the probability space. $\quad\Box$
\end{example}

\vspace{0.2in}
\subsection{\as{Random Field}\index{random field} Perspective}
\label{ssec:grf}

In this subsection we link the preceding constructions of
\as{random functions}\index{random function}, through randomized series, to the
notion of {\em random fields}.
Let $(\Omega,{\cal F},\P)$ be a probability space,
with expectation denoted by $\bbE$,
and $D \subseteq \R^d$ an open set.
For the random series constructions developed in the preceding
subsections $\Omega=\R^{\infty}$ and ${\cal F}=\cBc(\Omega)$; however
the development of the general theory of random fields does not require
this specific choice. A {\em \as{random field}}\index{random field}
on $D$ is a measurable mapping $u:D \times \Omega \to \R^n$. 
Thus, for any $x \in D$,
$u(x;\cdot)$ is an $\R^n$-valued random variable; 
on the other hand,
for any $\omega \in \Omega$, $u(\cdot;\omega):D \to \R^n$ is
a vector field. 
In the construction of random fields it is commonplace
to first construct the {\em finite dimensional distributions}.
These are found by choosing any integer $\kk \ge 1$, 
and any set of points $\{x_k\}_{k=1}^{\kk}$ in $D$, and then
considering the
random vector $(u(x_1;\cdot)^*, \cdots, u(x_{\kk};\cdot)^*)^*
\in \R^{n\kk}$. 
From the finite dimensional distributions of
this collection of random vectors we
would like to be able to make sense of the
probability measure $\mu$ on $X$, a \as{separable}\index{separable space} 
Banach space equipped
with the Borel $\sigma$-algebra $\cBc(X)$,
via the formula
\begin{equation}
\label{eq:defprobf}
\mu(A)=\P(u(\cdot;\omega) \in A), \quad A \in \cBc(X),
\end{equation}
where $\omega$ is taken from a common probability space
on which the random element $u \in X$ is defined.
It is thus necessary to study the joint distribution of a set
of $\kk$ $\R^n$-valued random variables, all on a common probability
space. Such $\R^{n\kk}$-valued random variables 
are, of course, only defined up to a set of zero measure.
It is desirable that all such finite dimensional distributions
are defined on a common subset $\Omega_0 \subset \Omega$
with full measure, so that $u$ may be viewed as a function
$u:D\times\Omega_0\to \R^n;$ such a choice of random
field is termed a {\em modification}
When reinterpreting the previous subsections in terms of random
fields, statements about almost sure 
(regularity) properties should be
viewed as statements concerning the existence of
a modification possessing the stated almost sure regularity
property.

We may define the space of functions 
$$L^q_{\P}(\Omega;X):=\Big\{v:D \times \Omega \to \R^n\Big|
\EE\bigl(\|v\|_{X}^q\bigl)<\infty\Big\}.$$
This is a Banach space, when equipped with the
norm $\Bigl(\EE\bigl(\|v\|_{X}^q\bigr)\Bigr)^{\frac{1}{q}}$.
We have used such spaces in the preceding subsections 
when demonstrating
convergence of the randomized series.
Note that we often 
simply write $u(x)$, suppressing 
the explicit dependence on
the probability space.

A {\em \as{Gaussian random field}}\index{random field!Gaussian}\index{Gaussian!random field}
 is one where, for any integer $\kk \ge 1$, 
and any set of points $\{x_k\}_{k=1}^{\kk}$ in $D$, the
random vector $(u(x_1;\cdot)^*, \cdots, u(x_{\kk};\cdot)^*)^*
\in \R^{nK}$ is a Gaussian random vector.
The {\em \as{mean function}}\index{mean function} of a
Gaussian random field is $m(x)=\bbE u(x)$.
The {\em \as{covariance function}}\index{covariance function}
is $c(x,y)=\bbE \bigl(u(x)-m(x)\bigr)\bigl(u(y)-m(y)\bigr)^*$.
For Gaussian random fields the \as{mean function}\index{mean function}
$m: D \to \R^n$
and the \as{covariance function}\index{covariance function}
$c: D \times D \to \R^{n \times n}$ together 
completely specify the joint probability distribution for
$(u(x_1;\cdot)^*, \cdots, u(x_{\kk})^*)^* \in \R^{n\kk}$. 
Furthermore, if we view the Gaussian random field as
a Gaussian measure on $L^2(D;\R^n)$ then
the \as{covariance operator}\index{covariance operator}
can be constructed
from the \as{covariance function}\index{covariance function}
as follows. Without loss
of generality we consider the mean zero case; the more general
case follows by shift
of origin. Since the field has
mean zero we have, from \eqref{e:propCmu2}, that for all $h_1,h_2 \in
L^2(D;\R^n)$,
\begin{align*}
\langle h_1,\cC h_2 \rangle&=\bbE \langle h_1,u\rangle\langle u,h_2\rangle\\
&=\bbE \int_{D}\int_{D} h_1(x)^*\bigl(u(x)u(y)^*\bigr)h_2(y)dydx\\
&=\bbE \int_{D}h_1(x)^*\Bigl(\int_{D} \bigl(u(x)u(y)^*\bigr)h_2(y)dy
\Bigr)dx\\
&=\int_{D}h_1(x)^*\Bigl(\int_{D}c(x,y)h_2(y)dy\Bigr)dx
\end{align*}
and we deduce that, for all $\psi \in L^2(D;\R^n)$,
\begin{equation}
\label{eq:hereitis}
\bigl({\cal C}\psi\bigr)(x)=\int_{D} c(x,y)\psi(y)dy.
\end{equation}
Thus the \as{covariance operator}\index{covariance function}
of a Gaussian random field is an integral operator
with kernel given by the covariance function.
As such we may also view the \as{covariance function}\index{covariance function} as the \as{Green's function}\index{Green's function} of the
inverse covariance, or {\em \as{precision}}\index{precision operator}.

A mean-zero \as{Gaussian random field}\index{random field!Gaussian} is 
termed {\em \as{stationary}}\index{random field!stationary}
\index{random field!stationary} if $c(x,y)=s(x-y)$
for some matrix-valued function $s$,
so that shifting the field by a fixed random vector
does not change the statistics. It is {\em \as{isotropic}}\index{isotropic}
\index{random field!isotropic} if it is stationary and, in addition,
$s(\cdot)=\iota(|\cdot|)$, for some matrix-valued function $\iota$.

In the previous subsection we demonstrated how the regularity
of random fields maybe established from the properties of the
sequences $\gamma$ (deterministic, with decay) and $\xi$ (i.i.d.\ random). 
Here we show similar results but express them in terms
of properties of the \as{covariance function}\index{covariance function}
and \as{covariance operator}\index{covariance operator}.

\begin{theorem} \label{t:kol}
Consider an $\R^n$-valued Gaussian random field $u$ on
$D \subset \R^d$ with mean zero and with isotropic 
correlation function $c:D\times D
\to \R^{n \times n}$.
Assume that $D$ is bounded and that
$\Tr c(x,y)=k\bigl(|x-y|\bigr)$ where $k:\R^+\to\R$
is H\"{o}lder with any exponent $\alpha \le 1$.
Then $u$ is almost surely H\"{o}lder continuous on $D$ with
any exponent smaller than $\frac{1}{2}\alpha$.
\end{theorem}

\begin{proof} We have
\begin{align*}
\bbE|u(x)-u(y)|^2 &=\bbE|u(x)|^2+\bbE|u(y)|^2-2\bbE
\langle u(x),u(y) \rangle\\
&=\Tr \Bigl(c(x,x)+c(y,y)-2c(x,y)\Bigr)\\
&=2\Bigl(k\bigl(0\bigr)-k\bigl(|x-y|\bigr)\Bigr)\\
&\le C|x-y|^{\alpha}.
\end{align*}
Since $u$ is Gaussian it follows that, for any integer
$r>0$,
$$\bbE|u(x)-u(y)|^{2r} \le C_r|x-y|^{\alpha r}.$$
Let $p=2r$ and noting that 
$$\alpha r=p\Bigl(\frac{\alpha}{2}-\frac{d}{p}\Bigr)+d$$
we deduce from Corollary \ref{cor:Kolmogorov} that $u$ is
H\"{o}lder continuous on $D$ with any exponent
smaller than
$$\sup_{p \in \N}\min\Bigl\{1,\frac{\alpha}{2}-\frac{d}{p}\Bigr\}=
\frac{\alpha}{2}\;,$$
which is precisely what we claimed.
\end{proof}

It is often convenient both algorithmically and theoretically
to define the \as{covariance operator}\index{covariance operator}
through fractional inverse powers
of a differential operator. Indeed in the previous subsection
we showed that our assumptions on the random series construction
we used could be interpreted as having a covariance operator
which was an inverse fractional power of the Laplacian on
zero spatial average functions with periodic boundary conditions.
We now generalize this perspective and consider 
\as{covariance operators}\index{covariance operator}
which are a fractional power of an operator $A$
satisfying the following.

\begin{assumption}
\label{def:triangle}The operator $A$, densely
defined on the Hilbert space $\cH = L^2(D;\R^n)$,
satisfies the following properties:
\begin{enumerate}\item $A$ is positive-definite,
self-adjoint and invertible;
\item the eigenfunctions
$\{\phi_j\}_{j\in\N}$ of $A$
form an orthonormal basis for $\cH$;
\item the eigenvalues of $A$ satisfy $\alpha_j\asymp j^{2/d}$; 
\item there is $C>0$ such that
$$\sup_{j \in \N}\Bigl(\|\phi_j\|_{L^{\infty}}+\frac{1}{j^{1/d}}{\rm Lip}(\phi_j)\Bigr) \le C.$$
\end{enumerate}
\end{assumption}

These properties are satisfied by the Laplacian on a torus, when
applied to functions with spatial mean zero. But they are in fact
satisfied for a much wider range of differential operators which
are {\em Laplacian-like}. For example the Dirichlet Laplacian
on a bounded open set $D$ in $\R^d$, together with various Laplacian
operators perturbed by lower order terms; for example Schr\"odinger
operators. Inspection of the proof of Theorem \ref{t:gh} reveals
that it only uses the properties of Assumptions \ref{def:triangle}.
Thus we have:

\begin{theorem} 
\label{t:gh2}
Let $u$ be a sample from the measure $N(0,\mathcal{C})$ 
in the case where $\mathcal{C}=A^{-s}$ with $A$ satisfying
Assumptions \ref{def:triangle} and $s>\frac{d}{2}$. 
Then, $\P$-a.s., $u \in \dot{H}^t$, for $t<s-\frac{d}{2}$, and 
$u \in \CC^{0,t}({D})$, for $t<1 \wedge (s-\frac{d}{2})$.
\end{theorem}

\begin{example} \label{ex:xe}
Consider the case $d=2, n=1$ and $D=[0,1]^2$.
Define the Gaussian random field through the measure
$\mu=N(0,\bigl(-\triangle)^{-\alpha}\bigr)$
where $\triangle$ is the Laplacian with domain
$H_0^1(D)\cap H^2(D)$. Then Assumptions \ref{def:triangle}
are satisfied by $-\triangle$. By Theorem \ref{t:gh2} it follows that
choosing $\alpha>1$ suffices to ensure that draws from $\mu$ are
almost surely in $L^2(D)$. It also follows that, in fact, 
draws from $\mu$ are almost surely in $\CC(D)$.
\end{example}

\vspace{0.2in}
\subsection{Summary}
\label{ssec:2.5}

In the preceding four subsections we have shown how to create
\as{random functions}\index{random function} by randomizing the coefficients of a series
of functions. Using these random series we have also
studied the regularity properties
of the resulting functions. Furthermore we have extended
our perspective in the Gaussian case to determine regularity
properties from the properties of the \as{covariance function}
\index{covariance function} or the \as{covariance operator}\index{covariance operator}.

For the uniform \as{prior}\index{prior} we have
shown that the random functions all live in a subset of
$X=L^{\infty}$ characterized by the upper and lower bounds
given in Theorem \ref{t:2.1} and found as the closure of the
linear span of the set of functions $(m_0,\{\phi_j\}_{j=1}^{\infty})$; 
denote this subset, which is a separable Banach space, by $X'$.
For the Besov \as{prior}s\index{prior!Besov} we have shown in Theorem \ref{t:2.2}
that the random functions
live in the separable Banach spaces $X^{t,q}$ for all $t<s-d/q$;
denote any one of these Banach spaces by $X'$.
And finally for the Gaussian \as{prior}s\index{prior!Gaussian}
we have shown in
Theorem \ref{t:2.4b} that the random function exists as an
$L^2$-limit in any of the Hilbert spaces $\cH^t$ for
$t<s-d/2$. Furthermore, we have indicated that, by use
of the \as{Kolmogorov continuity theorem}\index{Kolmogorov!continuity criterion}, 
we can also show that the Gaussian
random functions lie in certain H\"older spaces; these
H\"older spaces are not separable but, by the discussion in
subsection \ref{sssec:ch}, we can embed the spaces $\CC^{0,\gamma'}$
in the separable uniform H\"older spaces $\CC^{0,\gamma}_0$ 
for any $\gamma<\gamma'$; since the upper bound on the range 
of H\"older exponents established by use of \as{Kolmogorov continuity theorem}
\index{Kolmogorov!continuity criterion} is open, this means we can work in the
same range of H\"older exponents, but restricted to uniform H\"older spaces,
thereby regaining separability. In this Gaussian case we
denote any of the separable Hilbert or Banach spaces where the 
Gaussian random function lives almost surely by $X'$. 

Thus, in all of these examples, we have created
a probability measure $\mu_0$ which is the pushforward of the
measure $\PP$ on the \as{i.i.d.\ sequence}\index{i.i.d.\ sequence}
$\xi$ under the map which takes
the sequence into the random function. The resulting measure 
lives on the separable Banach space $X'$, 
and we will often write $\mu_0(X')=1$ to denote this fact.
This is shorthand for saying that functions drawn from $\mu_0$ are
in $X'$ almost surely.
Separability of $X'$ naturally leads to the 
use of the Borel $\sigma$-algebra to define a canonical measurable
space, and to the development of an integration theory --
Bochner integration -- which is natural on this space; see subsection
\ref{ssec:measures}.

\vspace{0.2in}
\subsection{Bibliographic Notes}

\begin{itemize}

\item Subsection \ref{ssec:2.1}. For general discussion of the
properties of \as{random functions}\index{random function}
constructed via randomization of coefficients in a series expansion see
\cite{Kah85}. The construction of probability measure on infinite sequences
of i.i.d.\ random variables may be found in \cite{Giu06}.

\item Subsection \ref{ssec:2.2}. These uniform \as{prior}s\index{prior!uniform} have been extensively
studied in the context of the field of Uncertainty Quantification
and the reader is directed to \cite{CDS1,CDS2} for more details. Uncertainty
Quantification in this context does not concern inverse problems, but rather
studies the effect, on the solution of an equation, of randomizing the input
data. Thus the interest is in the pushforward of a measure on input parameter
space onto a measure on solution space, for a differential
equation. Recently, however, these priors have been used to study the
inverse problem; see \cite{ScSt11}.

\item Subsection \ref{ssec:2.3}. Besov \as{prior}s\index{prior!Besov} 
were introduced in the paper
\cite{las09} and Theorem \ref{t:2.2} is taken from that paper. We
notice that the theorem constitutes a special case of the Fernique
Theorem in the Gaussian case $q=2$; it is restricted to a specific
class of Hilbert space norms, however, whereas the Fernique Theorem
in full generality applies in all norms on Banach spaces which have full
Gaussian measure. See \cite{Fernique,hairernotes} 
for proof of the Fernique Theorem.
A more general Fernique-like property of the Besov
measures is proved in \cite{DHS12} but it remains open to determine
the appropriate complete generalization of the Fernique Theorem to
Besov measures. 
For proof of Lemma \ref{l:2.3} see \cite[Chapter 4]{Kall02}.
For properties of families of functions that can form a basis for a Besov space, and examples of such families see \cite{Dau92, Mey92}.

\item Subsection \ref{ssec:2.4}. The general theory of Gaussian measures
on Banach spaces is contained in \cite{Lif95,bog98}. The text
\cite{DapZab92}, concerning the theory of stochastic PDEs, also has
a useful overview of the subject. The {Karhunen-Lo\`eve
expansion} \eqref{eq:KL} is contained in \cite{adler}.
The formal calculation
concerning the \as{covariance operator}\index{covariance operator}
of the Gaussian measure
which follows Theorem \ref{t:2.4b} leads to the answer which may be
rigorously justified 
by using characteristic functions; see,
for example, Proposition 2.18 in \cite{DapZab92}.
All three texts include statement and proof of the
Fernique Theorem in the generality given here.
The \as{Kolmogorov continuity theorem}\index{Kolmogorov!continuity criterion} 
is discussed in
\cite{DapZab92} and \cite{adler}. Proof of H\"older regularity
adapted to the case of the periodic setting may be found
in \cite{hairernotes} and ~\cite[Chapter 6]{article:Stuart2010}. 
For further reading on Gaussian measures see \cite{Giu06}.

\item Subsection \ref{ssec:grf}. A key tool in making the random
field perspective rigorous is the \as{Kolmogorov Extension Theorem}
\index{Kolmogorov!extension theorem}
\ref{theo:extension}.

\item Subsection \ref{ssec:2.5}. For a discussion of measure
theory on general spaces see \cite{BogMT}. The notion of 
Bochner integral is introduced in \cite{Bochner}; we discuss
it in subsection \ref{ssec:measures}. 

\end{itemize}

\vspace{0.3in}
\section{\as{Posterior}\index{posterior} Distribution}
\label{sec:post}

In this section we prove a Bayes' theorem appropriate for 
combining a \as{likelihood}\index{likelihood} with 
\as{prior}\index{prior} measures on separable Banach spaces as
constructed in the previous section. In subsection \ref{ssec:3.1}
we start with some general remarks about conditioned random
variables. Subsection \ref{ssec:3.2} contains our statement
and proof of a Bayes' theorem, and specifically its application
to Bayesian inversion. We note here that, in our setting, the
\as{posterior}\index{posterior}
$\muy$ will always be absolutely continuous with respect to the
\as{prior}\index{prior} $\mu_0$, and we use the standard notation $\muy \ll \mu_0$ to denote 
this. It is possible to construct examples, for instance in the
purely Gaussian setting, where the \as{posterior}\index{posterior}
is not absolutely continuous with respect to the prior.
Thus it is certainly not necessary to work in the setting where
$\muy \ll \mu_0$. However it is
quite natural, from a modelling point of view, to work in this
setting: absolute continuity
ensures that almost sure properties built into the prior will be inherited
by the \as{posterior}\index{posteruir}. 
For these almost sure properties to be changed by
the data would require that the data contains an infinite ammount of
information, something which is unnatural in most applications.

In subsection \ref{ssec:3.3}
we study the example of the heat equation, introduced in
subsection \ref{ssec:1.2}, from the perspective of Bayesian inversion 
and in subsection \ref{ssec:3.4} we do the same for the elliptic
inverse problem of subsection \ref{ssec:1.3}.

\vspace{0.2in}
\subsection{Conditioned Random Variables}
\label{ssec:3.1}

Key to the development of \as{Bayes' Theorem}\index{Bayes' Theorem}, and the 
\as{posterior}\index{posterior} distribution, 
is the notion of conditional random variables.
In this section we state an important
theorem concerning conditioning.

Let $(X,A)$ and $(Y,B)$ denote a pair of measurable spaces and let
$\nu$ and $\pi$ be probability measures on $X\times Y$. We assume
that $\nu\ll\pi$. Thus there exists $\pi$-measurable $\phi: X\times
Y\rightarrow\R$ with $\phi \in L^1_{\pi}$ (see section \ref{sssec:other} for 
definition of $L^1_{\pi}$) and
\begin{equation}\label{31}
\frac{d\nu}{d\pi}(x,y)=\phi(x,y).
\end{equation}
That is, for $(x,y)\in X\times Y$,
\begin{equation*}
\bbE^\nu f(x,y) =\bbE^\pi\big(\phi(x,y)f(x,y)\big),
\end{equation*}
or, equivalently,
\begin{equation*}
\int_{X\times Y}f(x,y)\nu(dx,dy)=\int_{X\times
Y}\phi(x,y)f(x,y)\pi(dx,dy).
\end{equation*}

\begin{theorem} \label{t:3.1}
Assume that the conditional random
variable $x|y$ exists under $\pi$ with probability distribution
denoted $\pi^y(dx)$. Then the conditional random variable $x|y$
under $\nu$ exists, with probability distribution denoted by
$\nu^y(dx)$. Furthermore, $\nu^y\ll\pi^y$ and
if $c(y):=\int_X \phi(x,y)d\pi^y(x)>0$ then 
\begin{eqnarray*}
\frac{d\nu^y}{d\pi^y}(x)=
    \frac{1}{c(y)}\phi(x,y). 
\end{eqnarray*}
\end{theorem}

\begin{example} \label{ex:bb}
Let $X=C\big([0,1];\R\big)$,
$Y=\R$. Let $\pi$ denote the measure on $X\times Y$ induced
by the random variable $\big(w(\cdot),w(1)\big)$, where $w$ is a
draw from standard unit \as{Wiener measure}\index{Wiener!measure} on $\R$, starting
from $w(0)=z$.  Let $\pi^y$ denote measure on X found by conditioning Brownian
motion to satisfy $w(1)=y$, thus $\pi^y$ is a Brownian bridge
measure with $w(0)=z, w(1)=y$.

Assume that $\nu\ll\pi$ with
\begin{equation*}
\frac{d\nu}{d\pi}(x,y)=\exp\big(-\Phi(x,y)\big).
\end{equation*}
Assume further that
\begin{equation*}
\sup_{x\in {X}}\Phi(x,y)=\Phi^+(y)<\infty
\end{equation*}
for every $y\in \R$. Then 
\begin{equation*}
c(y)=\int_{\R}\exp\big(-\Phi(x,y)\big)d\pi^y(x)>\exp\big(-\Phi^+(y)\big)>0.
\end{equation*}
Thus $\nu^y(dx)$ exists and
\begin{equation*}
\frac{d\nu^y}{d\pi^y}(x)=\frac{1}{c(y)}\exp\big(-\Phi(x,y)\big).
\quad\Box
\end{equation*}
\end{example}

We will use the preceding theorem to go from a construction of the joint
probability distribution on unknown and data to the conditional distribution
of the unknown, given data. In constructing the joint probability
distribution we will need to establish measurability of the 
\as{likelihood}\index{likelihood}, for which the following will be useful:

\begin{lemma}
\label{l:meas}
Let $(Z,B)$ be a Borel measurable topological space and assume that
$G \in \CC(Z;\R)$ and that $\pi(Z)=1$ for some
probability measure $\pi$ on $(Z,B)$.
Then $G$ is a $\pi$-measurable function.
\end{lemma}

\vspace{0.2in}
\subsection{\as{Bayes' Theorem}\index{Bayes' Theorem} for Inverse Problems}
\label{ssec:3.2}

Let $X$, $Y$ be separable Banach spaces, equipped with the Borel
$\sigma$-algebra, and $G:X\rightarrow Y$ a
measurable mapping. We wish to
solve the inverse problem of finding $u$ from $y$ where
\begin{equation}\label{32}
y=G(u)+\eta
\end{equation}
and $\eta\in Y$ denotes noise.
We employ a Bayesian approach to this problem in which
we let $(u,y)\in X\times Y$ be a random variable and
compute $u|y$. We specify the random variable $(u,y)$ as
follows:

\begin{itemize}

\item {\bf Prior}: $u\sim\mu_0$ measure on $X$.

\item {\bf Noise}: $\eta\sim\mathbb{Q}_0$ measure on $Y$,
and (recalling that $\perp$ denotes
independence) $\eta\perp u$.

\end{itemize}

The random variable $y|u$ is then distributed
according to the measure
$\mathbb{Q}_u$, the translate of $\mathbb{Q}_0$ by $G(u)$.
We {\em assume} throughout the following
that $\mathbb{Q}_u\ll\mathbb{Q}_0$ for $u$ $\mu_0$- a.s. Thus,
for some {\bf potential} $\Phi: X\times Y\rightarrow\R$,
\begin{equation}
\label{insert3}
\frac{d\mathbb{Q}_u}{d\mathbb{Q}_0}(y)=\exp\big(-\Phi(u;y)\big).
\end{equation}
Thus, for fixed $u$, $\Phi(u;\cdot):Y \to \R$ is measurable
and $\bbE^{\mathbb{Q}_0} \exp\big(-\Phi(u;y)\big)=1$.
For given instance of the data $y$, $-\Phi(\cdot;y)$ is termed the {\bf 
log \as{likelihood}}\index{likelihood!log}. 

Define $\nu_0$ to be the product measure 
\begin{equation}
\nu_0(du,dy)=\mu_0(du)\mathbb{Q}_0(dy).
\label{insert1}
\end{equation}
We {\em assume} in what follows that
$\Phi(\cdot,\cdot)$ is $\nu_0$ measurable. 
Then the random variable $(u,y)\in X\times Y$ is distributed
according to measure $\nu(du,dy)=\mu_0(du)\mathbb{Q}_u(dy)$. 
Furthermore, it then follows that $\nu \ll \nu_0$ with 
\begin{equation*}
\frac{d\nu}{d\nu_0}(u,y)=\exp\big(-\Phi(u;y)\big).
\end{equation*}
We have the following infinite dimensional analogue of
Theorem \ref{t:1.1}.

\begin{theorem}[{\bf \as{Bayes' Theorem}\index{Bayes' Theorem}}]
\label{t:3.2} Assume that
$\Phi:X\times Y \to \R$ is $\nu_0$ measurable and that, for $y$
$\mathbb{Q}_0$-a.s.,
\begin{equation}
\label{insert33}
Z:=\int_{X}\exp\big(-\Phi(u;y)\big)\mu_0(du)>0.
\end{equation}
Then the conditional distribution of $u|y$ exists under $\nu$, and is denoted by
$\mu^y$. Furthermore $\mu^y\ll\mu_0$ and, for $y$ $\nu$-a.s.,
\begin{equation}
\frac{d\mu^y}{d\mu_0}(u)=\frac{1}{Z}\exp\big(-\Phi(u;y)\big).
\label{insert2}
\end{equation}
\end{theorem}

\noindent{\bf Proof.}
First note that the positivity of $Z$ holds for $y$ $\nu_0$-almost surely,
and hence by absolute continuity of $\nu$ with respect to $\nu_0$,
for $y$ $\nu$-almost surely.
The proof is an application of Theorem
\ref{t:3.1} with $\pi$ replaced by $\nu_0$,
$\phi(x,y)=\exp\big(-\Phi(u,y)\big)$ and $(x,y)=(u,y)$.
Since $\nu_0(du,dy)$ has product form, the conditional distribution
of $u|y$ under $\nu_0$ is simply $\mu_0$. The result follows. $\Box$

\begin{remarks} \label{r:implement}
In order to implement the derivation of \as{Bayes' formula}\index{Bayes' formula} \eqref{insert2}
four essential steps are required:

\begin{itemize}

\item Define a suitable \as{prior}\index{prior} measure $\mu_0$ and noise measure $\bbQ_0$
whose independent product form the reference measure $\nu_0$.

\item Determine the potential $\Phi$ such that formula \eqref{insert3} holds.

\item Show that $\Phi$ is $\nu_0$ measurable.

\item Show that the normalization constant $Z$ given by \eqref{insert33} is
positive almost surely with respect to $y \sim \bbQ_0$.

\end{itemize}
\end{remarks}

We will show how to carry out this program for two examples
in the following subsections. The following remark will also be
used in studying one of the examples.

\begin{remarks} \label{r:3.3}
The following comments on the set-up above may be useful.

\begin{itemize}

\item
In formula \eqref{insert2}
we can shift $\Phi(u,y)$ by any constant
$c(y)$, independent of $u$, provided the
constant is finite $\bbQ_0$-a.s. and hence $\nu$-a.s.
Such a shift can be absorbed into a redefinition of
the normalization constant $Z$.

\item
Our Bayes' Theorem only asserts that the \as{posterior}\index{posterior}
is absolutely continuous
with respect to the \as{prior}\index{prior} $\mu_0$. In fact equivalence (mutual
absolute continuity) will occur when $\Phi(\cdot;y)$ is finite
everywhere in $X$.

\end{itemize}

\end{remarks}

\vspace{0.2in}
\subsection{Heat Equation}
\label{ssec:3.3}

We apply Bayesian inversion to the heat equation from subsection \ref{ssec:1.2}.
Recall that for $G(u)=e^{-A}u$, we have the relationship
$$y=G(u)+\eta,$$
which we wish to invert.
Let $X=H$ and define
$$\cH^t=\CD(A^{t/2})=\Bigl\{w\big| w=A^{-t/2}w_0, w_0 \in H\Bigr\}.$$
Under Assumptions \ref{a:1.3} we have 
$\alpha_j\asymp j^{\frac{2}{d}}$ so that this family of
spaces is identical with the \as{Hilbert scale} of spaces\index{Hilbert scale}
$\mathcal{H}^t$ as defined in subsections \ref{ssec:1.2} and \ref{ssec:2.4}.

We choose the \as{prior}\index{prior} $\mu_0=N(0,A^{-\alpha}),\,\, \alpha>\frac{d}{2}$.
Thus $\mu_0(X)=\mu_0(H)=1$. Indeed the analysis in subsection \ref{ssec:2.4} shows that
$\mu_0(\mathcal{H}^t)=1$, $t<\alpha-\frac{d}{2}$.
For the \as{likelihood}\index{likelihood} we assume that $\eta\perp u$ with
$\eta\sim\mathbb{Q}_0=N(0,A^{-\beta})$, and
$\beta \in \R$. This measure satisfies $\bbQ_0(\cH^t)=1$ for
$t<\beta-\frac{d}{2}$ and we thus choose $Y=\cH^{t'}$ for
some $t'<\beta-\frac{d}{2}$. Notice that our analysis includes
the case of white observational noise, for which $\beta=0$.
The \as{Cameron-Martin\index{Cameron-Martin!theorem} Theorem} \ref{theo:CM}, together with the fact that
$e^{-\lambda A}$ commutes with arbitrary fractional powers of $A$, can be used
to show that $y|u\sim\mathbb{Q}_u:=N(G(u),A^{-\beta})$ where
$\mathbb{Q}_u\ll\mathbb{Q}_0$ with
$$\frac{d\mathbb{Q}_u}{d\mathbb{Q}_0}(y)=\exp\big(-\Phi(u;y)\big),$$
and
$$\Phi(u;y)=\frac{1}{2}\|A^{\frac{\beta}{2}}e^{-A}u\|^2-\langle
 A^{\frac{\beta}{2}}e^{-\frac{A}{2}}y, A^{\frac{\beta}{2}}e^{-\frac{A}{2}}u\rangle.$$
In the following we repeatedly use the fact that
$A^{\gamma}e^{-\lambda A}$, $\lambda>0$, is a bounded
linear operator from $\mathcal{H}^a$ to $\mathcal{H}^b$, any $a,b, \gamma
\in \R$.
Recall that $\nu_0(du,dy)=\mu_0(du)\bbQ_0(dy)$.
Note that $\nu_0(H \times \cH^{t'})=1$.
Using the boundedness of $A^{\gamma}e^{-\lambda A}$ it may be shown that
$$\Phi: H\times \cH^{t'}\rightarrow \R$$
is continuous, and hence $\nu_0$-measurable
by Lemma \ref{l:meas}.

Theorem \ref{t:3.2} shows that the \as{posterior}\index{posterior}
is given by $\mu^y$ where
$$\frac{d\mu^y}{d\mu_0}(u)=\frac{1}{Z}\exp\big(-\Phi(u;y)\big),$$
$$Z=\int_H\exp\big(-\Phi(u;y)\big)\mu_0(du),$$
provided that $Z>0$ for $y$ $\bbQ_0$-a.s.
We establish this positivity in the remainder of the proof.
Since $y \in \cH^t$ for any $t<\beta-\frac{d}{2}$, $\bbQ_0$-a.s., we
have that $y=A^{-t'/2}w_0$ for some $w_0 \in H$ and $t'<\beta-\frac{d}{2}$.
Thus we may write
\begin{equation}
\label{eq:nb}
\Phi(u;y)=\frac{1}{2}\|A^{\frac{\beta}{2}}e^{-A}u\|^2-\langle
 A^{\frac{\beta-t'}{2}}e^{-\frac{A}{2}}w_0, A^{\frac{\beta}{2}}e^{-\frac{A}{2}}u\rangle.
\end{equation}
Then, using the boundedness of $A^{\gamma}e^{-\lambda
A}$, $\lambda>0$, together with \eqref{eq:nb}, we have
$$\Phi(u;y)\leq C(\|u\|^2+\|w_0\|^2)$$
where $\|w_0\|$ is finite $\bbQ_0$-a.s. Thus
$$Z\geq\int_{\|u\|^2\leq 1}\exp\big(-C(1+\|w_0\|^2)\big)\mu_0(du)$$
and, since $\mu_0(\|u\|^2 \le 1)>0$ (by Theorem \ref{theo:balls}
all balls have positive measure for Gaussians on a separable 
Banach space) the required positivity follows.

\vspace{0.2in}
\subsection{Elliptic Inverse Problem}
\label{ssec:3.4}

We consider the elliptic
inverse problem from subsection \ref{ssec:1.3} from the Bayesian
perspective. We consider the use of both uniform and Gaussian
\as{prior}s\index{prior}. Before studying the inverse problem, however,
it is important to derive some continuity properties of
the forward problem.  Throughout this section
we consider equation \eqref{15}
under the assumption that $f\in V^*.$

\subsubsection{Forward Problem}
\label{sssec:f}
Recall that in subsection \ref{ssec:1.3}, equation \eqref{eq:Xp},
we defined
\begin{equation}
\label{eq:xp}
X^+=\Big\{v \in L^{\infty}(D)\Big|\text{ess}\inf_{x \in D} v(x)>0\Big\}.
\end{equation}
Then the map $\cR:X^+ \to V$ by $\cR(\kappa)=p$. This map
is well-defined by Lemma \ref{l:1.5} and we have the following result.

\begin{lemma}
\label{l:3.3}
For $i=1,2$, let
\begin{eqnarray*}
-\nabla\cdot(\kappa_i\nabla p_i) &=& f, \quad  x\in D,\\
p_i &=& 0,\quad   x\in\partial D.
\end{eqnarray*}
Then
$$\|p_1-p_2\|_V \le 
\frac{1}{\kappa_{\min}^2}\|f\|_{V^*}
\|\kappa_1-\kappa_2\|_{L^{\infty}}$$
where we assume that
$$\kappa_{\min}:={\rm ess}\inf_{x\in D}\kappa_1(x) \wedge {\rm ess}\inf_{x\in D}\kappa_2(x)>0.$$
Thus the function $\cR: X^+\rightarrow V$ is locally Lipschitz.
\end{lemma}

\noindent{\bf Proof.} 
Let $e=\kappa_1-\kappa_2$, $r=p_1-p_2$. Then
\begin{eqnarray*}
-\nabla\cdot(\kappa_1\nabla r) &=& \nabla\cdot\big((\kappa_1-\kappa_2)\nabla p_2\big), \quad  x\in D\\
r &=& 0,\quad   x\in\partial D.
\end{eqnarray*}
Multiplying by $r$ and integrating by parts on both sides of the identity
gives
$$\kappa_{\min}\int_{D} |\nabla r|^2\,dx \le \|(\kappa_2-\kappa_1)\nabla p_2\|
\|\nabla r\|.$$
Using the fact that $\|\varphi\|_{V}=\|\nabla \varphi\|$, and
applying Lemma \ref{l:1.5} to bound $p_2$ in $V$, we find that 
\begin{eqnarray*}
\|r\|_V &\leq& \|(\kappa_2-\kappa_1)\nabla p_2\|/\kappa_{\min}\\
& \leq &\|\kappa_2-\kappa_1\|_{L^{\infty}}\|p_2\|_V/\kappa_{\min}\\
& \leq & \frac{1}{\kappa_{\min}^2}\|f\|_{V^*}\|e\|_{L^{\infty}}.\\
\end{eqnarray*}
$\quad\Box$

\subsubsection{Uniform \as{Prior}s\index{prior!uniform}}

We now study the inverse problem of finding $\kappa$ from a finite
set of continuous linear functionals $\{l_j\}_{j=1}^J$  on $V$,
representing measurements of $p$; thus $l_j \in V^*$.
To match the notation from subsection \ref{ssec:3.2} we take $\kappa=u$
and we define the separable Banach space $X'$ as in subsection
\ref{ssec:2.2}. It is straightforward to see that Lemma \ref{l:3.3}
extends to the case where $X^+$ given by \eqref{eq:xp} is
replaced by
\begin{equation}
\label{eq:xp2}
X^+=\Big\{v \in X'\Big|\text{ess}\inf_{x \in D} v(x)>0\Big\}
\end{equation}
since $X' \subset L^{\infty}(D)$.
When considering uniform \as{prior}s\index{prior!uniform}
for the elliptic problem
we work with this definition of $X^+$.

We define $G:X^+ \to \R^J$ by
$$G_j(u)=l_j\bigl(\cR(u)\bigr),\quad j=1,\ldots, J$$
where, recall, the $l_j$ are elements of $V^*$: bounded
linear functionals on $V$.
Then $G(u)=\bigl(G_1(u),\cdots, G_J(u)\bigr)$
and we are interested in the inverse problem of finding 
$u \in X^+$ from $y$ where
$$y=G(u)+\eta$$
and $\eta$ is the noise.  We assume
$\eta\sim N(0,\Gamma)$, for positive symmetric $\Gamma\in \R^{J\times J}$.
(Use of other statistical assumptions on $\eta$ is a straightforward
extension of what follows whenever $\eta$ has a smooth
density on $\R^J.$)

Let $\mu_0$ denote the \as{prior}\index{prior}
measure constructed in subsection \ref{ssec:2.2}. 
Then $\mu_0$-almost surely we have, by Theorem \ref{t:2.1},
\begin{equation}
\label{eq:xp3}
u \in X^+_0:=\Big\{v \in X'\Big|
\frac{1}{1+\delta}m_{\min}\leq
v(x)\leq m_{\max}+\frac{\delta}{1+\delta}m_{\min} \quad
\text{a.e.} \quad x\in D\Bigr\}.
\end{equation}
Thus $\mu_0(X^+_0)=1$.

The \as{likelihood}\index{likelihood} is defined as
follows. Since $\eta\sim N(0,\Gamma)$
it follows that
$\mathbb{Q}_0=N(0,\Gamma)$,
$\mathbb{Q}_u=N\bigl(G(u),\Gamma\bigr)$ and
$$\frac{d\mathbb{Q}_u}{d\mathbb{Q}_0}(y)=\exp\big(-\Phi(u;y)\big),$$
$$\Phi(u;y)=\frac{1}{2}\bigl|\Gamma^{-\frac{1}{2}}(y-G(u))\bigr|^2-\frac{1}{2}
\bigl|\Gamma^{-\frac{1}{2}}y\bigr|^2.$$
Recall that $\nu_0(dy,du)=\mathbb{Q}_0(dy)\mu_0(du)$. 
Since $G: X^+\rightarrow \R^J$ is locally Lipschitz by 
Lemma \ref{l:3.3}, Lemma \ref{l:meas} implies that $\Phi: X^+\times
Y\rightarrow\R$ is $\nu_0$-measurable. Thus Theorem \ref{t:3.2}
shows that $u|y\sim\mu^y$ where
\begin{equation}
\label{eq:poste-}
\frac{d\mu^y}{d\mu_0}(u)=\frac{1}{Z}\exp\big(-\Phi(u;y)\big)
\end{equation}
$$Z=\int_{X^+}\exp\big(-\Phi(u;y)\big)\mu_0(du),$$
provided $Z>0$ for $y$ $\bbQ_0$-almost surely.
To see that $Z>0$ note that
$$Z=\int_{X^+_0}\exp\big(-\Phi(u;y)\big)\mu_0(du),$$
since $\mu_0(X^+_0)=1$. On $X^+_0$ we have that $\cR(\cdot)$ is bounded in $V$,
and hence $G$ is bounded in $\R^J$. Furthermore $y$ is finite $\bbQ_0$-almost surely. 
Thus $\bbQ_0$-almost surely with respect to $y$, $\Phi(\cdot;y)$ is bounded
on $X^+_0$; we denote the resulting bound by $M=M(y)<\infty$. Hence
$$Z \ge \int_{X^+_0}\exp(-M)\mu_0(du)=\exp(-M)>0.$$
and the result is proved.

We may use Remark \ref{r:3.3}
to shift $\Phi$ by $\frac12|\Gamma^{-\frac{1}{2}}y|^2$,
since this is almost
surely finite under $\bbQ_0$ and hence under
$\nu(du,dy)=\mathbb{Q}_u(dy)\mu_0(du)$.
We then obtain the equivalent form for the \as{posterior}\index{posterior}
distribution $\mu^y$:
\begin{subequations}
\label{eq:poste}
\begin{align}
\frac{d\mu^y}{d\mu_0}(u)&=\frac{1}{Z}\exp\Big(-\frac{1}{2}\bigl|
\Gamma^{-\frac{1}{2}}\bigl(y-G(u)\bigr)\bigr|^2\Big),\\
Z&=\int_{X}\exp\Big(-\frac{1}{2}|\Gamma^{-\frac{1}{2}}\bigl(y-G(u)\bigr)\bigr|^2\Big)\mu_0(du).\end{align}
\end{subequations}

\subsubsection{Gaussian \as{Prior}s\index{prior!Gaussian}}

We conclude this subsection by discussing the same inverse problem,
but using Gaussian \as{prior}s\index{prior!Gaussian}
from subsection \ref{ssec:2.4}.
We now set $X=\CC(\overline{D})$, 
$Y=\R^J$
and we note that $X$ embeds continuously into $L^{\infty}(D)$. 
We assume that we can find an operator $A$ 
which satisfies Assumptions \ref{def:triangle}.
We now take $\kappa=\exp(u)$, and
define $G:X \to \R^J$ by
$$G_j(u)=l_j\Bigl(\cR\bigr(\exp(u)\bigr)\Bigr),\quad j=1,\ldots, J.$$
We take as prior on $u$
the measure $N(0,A^{-s})$ with $s>d/2$. Then Theorem \ref{t:gh2} shows that 
$\mu(X)=1$. The \as{likelihood}\index{likelihood} is unchanged
by the \as{prior}\index{prior}, 
since it concerns $y$ given $u$, and is hence identical
to that in the case of the uniform \as{prior}\index{prior!uniform}, 
although the mean shift
from $\bbQ_0$ to $\bbQ_u$ by $G(u)$ now has a different interpretation
since $\kappa=\exp(u)$ rather than $\kappa=u$.
Thus we again obtain
\eqref{eq:poste-} for the \as{posterior}\index{posterior}
distribution (albeit with a different
definition of $G(u)$) provided that we can establish that, $\bbQ_0$-a.s.,
$$Z=\int_{X}\exp\Big(\frac{1}{2}\bigl|\Gamma^{-\frac{1}{2}}y\bigr|^2-\frac{1}{2}\bigl|\Gamma^{-\frac{1}{2}}\bigl(y-G(u)\bigr)\bigr|^2\Big)\mu_0(du)>0.$$
To this end we use the fact that the unit ball in $X$, denoted $B$,
has positive measure by Theorem \ref{theo:balls}, 
and that on this ball $\cR\bigl(\exp(u)\bigr)$
is bounded in $V$ by $e^{-a}\|f\|_{V^*}$, by Lemma \ref{l:1.5}, for
some finite positive constant $a$. This follows from the continuous embedding
of $X$ into $L^{\infty}$ and since the infimum of $\kappa=\exp(u)$
is bounded below by $e^{-\|u\|_{L^{\infty}}}$. 
Thus $G$ is bounded on $B$ and, noting that $y$ is
$\bbQ_0$-a.s. finite, we have for some $M=M(y)<\infty$,
$$\sup_{u \in B}\Bigl(\frac{1}{2}\bigl|\Gamma^{-\frac{1}{2}}\bigl(y-G(u)\bigr)\bigr|^2-
\frac12\bigl|\Gamma^{-\frac{1}{2}}y|^2\Bigr)<M.$$
Hence
$$Z \ge \int_{B} \exp(-R)\mu_0(du)=\exp(-R)\mu_0(B)>0$$
since all balls have positive measure for Gaussian measure
on a separable Banach space.
Thus we again obtain \eqref{eq:poste} for the \as{posterior}\index{posterior}
measure, now with the new definition of $G$, and hence $\Phi$.

\vspace{0.2in}
\subsection{Bibliographic Notes}

\begin{itemize}

\item Subsection \ref{ssec:3.1}. Theorem \ref{t:3.1} 
is taken from \cite{HSWV05} where it
is used to compute expressions for the measure induced by various
conditionings applied to SDEs. The existence of regular conditional
probability distributions is discussed in \cite{Kall02}, Theorem 6.3. The Example 
\ref{ex:bb}, concerning end-point conditioning
of measures defined via a density with respect to Wiener
measure, finds application to problems from molecular
dynamics in \cite{PS10,PST11}.
Further material concerning the equivalence of \as{posterior}\index{posterior}
with respect to the \as{prior}\index{prior} may be found in  
\cite[Chapters 3 and 6]{article:Stuart2010},~\cite{ALS12},~\cite{ASZ12}.
The equivalence of Gaussian measures is studied via the
\as{Feldman-H\'ajeki} theorem; \index{Feldman-H\'ajek Theorem} 
see  \cite{DapZab92} and the Appendix. 
A proof of Lemma \ref{l:meas} can be found in \cite[Chapter 1, Theorem 1.12]{Rudin87}. See also \cite[Lemma 1.5]{Kall02}.

\item Subsection \ref{ssec:3.2}. General development of 
\as{Bayes' Theorem}\index{Bayes' Theorem}
for inverse problems on function space, along the lines described here,
may be found in \cite{CDRS08,article:Stuart2010}. The reader is also directed to
the papers \cite{La02,La07} for earlier related
material, and to \cite{las11,las12,las12b} for recent developments.

\item Subsection \ref{ssec:3.3}. The inverse problem for the heat
equation was one of the first infinite dimensional inverse problems
to receive Bayesian treatment; see \cite{Fr70}, leading to further
developments in \cite{Man84,LPS89}. The problem is
worked through in detail in \cite{article:Stuart2010}. To fully understand
the details the reader will need to study the Cameron-Martin
theorem (concerning shifts in the mean of Gaussian measures)
and the Feldman-H\'ajek theorem\index{Feldman-H\'ajek Theorem} 
(concerning equivalence of Gaussian measures); 
both of these may be found in \cite{DapZab92,Lif95,bog98}
and are also discussed in \cite{article:Stuart2010}.

\item Subsection \ref{ssec:3.4}. The elliptic inverse problem with
the uniform \as{prior}\index{prior!uniform} 
is studied in \cite{ScSt11}. A Gaussian \as{prior}\index{prior!Gaussian}
is adopted in \cite{DS11}, and a Besov \as{prior}\index{prior!Besov}
in \cite{DHS12}.

\end{itemize}

\vspace{0.3in}
\section{Common Structure}
\label{sec:common}

In this section we discuss various common features of the 
\as{posterior}\index{posterior}
distribution arising from the Bayesian approach to inverse problems.
We start, in subsection \ref{ssec:4.1}, by studying the continuity
properties of the posterior with respect to changes in data, proving
a form of well-posedness; indeed we show that the \as{posterior}\index{posterior} is
Lipschitz in the data with respect to the \as{Hellinger metric}\index{Helliner distance}.
In subsection \ref{ssec:4.2} we 
use similar ideas to study the effect of approximation
on the posterior distribution, showing that small changes in the
potential $\Phi$ lead to small changes in the \as{posterior}\index{posterior}
distribution,
again in the \as{Hellinger metric}\index{Hellinger distance}; 
this work may be used to translate error
analysis pertaining to the forward problem into estimates on errors
in the \as{posterior}\index{posterior} distribution.
In the final subsection \ref{ssec:map}
we study an important link between
the Bayesian approach to inverse problems and classical regularization
techniques for inverse problems; specifically we link the
Bayesian MAP estimator to a Tikhonov-Phillips regularized least squares
problem. The first two subsections work with general \as{prior}s\index{prior},
whilst the final one is concerned with Gaussians only.

\vspace{0.2in}
\subsection{Well-Posedness}
\label{ssec:4.1}

In many classical inverse problems small changes in the data can induce
arbitrarily large changes in the solution, 
and some form of regularization is needed
to counteract this ill-posedness. We illustrate this effect with
the inverse heat equation example. We then proceed to show that 
the Bayesian approach to inversion has the property that small changes in the
data lead to small changes in the \as{posterior}\index{posterior} distribution. 
Thus working with probability measures on the solution space, and
adopting suitable \as{prior}s\index{prior}, 
provides a form of regularization.

\begin{example} 
Consider the heat equation introduced in
subsection \ref{ssec:1.2} and both perfect data $y=e^{-A}u$, derived
from the forward model with no noise, and noisy 
data $y'=e^{-A}u+\eta.$
Consider the case where $\eta=\eps \varphi_j$ with $\eps$ small
and $\varphi_j$ a normalized eigenfuction of $A$. Thus $\|\eta\|=\eps$.
Obviously application of the inverse of $e^{-A}$ to $y$ returns the
point $u$ which gave rise to the perfect data.
It is natural to apply the inverse of $e^{-A}$ to both $y$ and to $y'$
to understand the effect of the noise. Doing so yields the identity
\begin{equation*}
\|e^Ay-e^A y'\|  = \|e^A(y-y')\| =\|e^A\eta\|=\eps\|e^A\varphi_j\| =\eps e^{\alpha_j}\;.
\end{equation*}
Recall Assumption \ref{a:1.3} which gives $\alpha_j \asymp j^{2/d}$.
Now fix any $a>0$ and choose $j$ large enough to ensure that 
$\alpha_j=(a+1)\log(\eps^{-1}).$ It then follows that 
$\|y-y'\|={\cal O}(\eps)$ whilst $\|e^Ay-e^A y'\|={\cal O}(\eps^{-a}).$
This is a manifestation of ill-posedness. Furthermore, since $a>0$ is
arbitrary, the ill-posedness can be made arbitrarily bad by 
considering $a\to \infty.$
$\quad\Box$
\label{ex:heated}
\end{example}

Our aim in this section is to show that this ill-posedness
effect does not occur in the Bayesian \as{posterior}\index{posterior}
distribution: small  changes in the data $y$ lead to small changes in
the measure $\mu^y$.
Let $X, Y$ be separable Banach spaces, equipped with the Borel
$\sigma$-algebra, and $\mu_0$ a measure on $X$. We will work under
assumptions which enable us to make sense of the following
measure  $\mu^y\ll\mu_0$ defined, for some $\Phi: X \times Y \to \R$,
by
\begin{subequations}
\label{41}
\begin{eqnarray}
\frac{d\mu^y}{d\mu_0}(u)&=&\frac{1}{Z(y)}\exp\big(-\Phi(u;y)\big),
\label{41a}\\
Z(y)&=&\int_X\exp\big(-\Phi(u;y)\big)\mu_0(du). \label{41b}
\end{eqnarray}
\end{subequations}
We make the following assumptions concerning $\Phi$ :

\begin{assumptions}
\label{a:4.1}
Let $X'\subseteq X$ and assume that $\Phi\in
\CC(X'\times Y; \R)$. Assume
further that there are functions $M_i:\R^+ \times \R^+ \to \R^+$,
$i=1,2$, monotonic non-decreasing separately in each argument, and
with $M_2$ strictly positive, such that
for all $u\in X'$, $y, y_1, y_2\in B_Y(0,r)$,
$$\Phi(u;y)\geq -M_1(r,\|u\|_X),$$
$$|\Phi(u;y_1)-\Phi(u;y_2)|\leq M_2(r,\|u\|_X)\|y_1-y_2\|_Y.\quad\Box$$
\end{assumptions}

In order to measure the effect of changes in $y$ on the
measure $\mu^y$ we need a metric on measures. We use the
\as{Hellinger metric}\index{Hellinger distance}
defined in subsection \ref{ssec:metricpm}.

\begin{theorem}
\label{t:4.2}
Let Assumptions \ref{a:4.1} hold. 
Assume that $\mu_0(X')=1$ and that $\mu_0(X' \cap B)>0$ for some
bounded set $B$ in $X$. 
Assume additionally that, for every fixed $r>0$,
$$\exp\big(M_1(r,\|u\|_X)\big)\in L_{\mu_0}^1(X;\R).$$
Then, for every $y \in Y$,
$Z(y)$ given by (\ref{41b}) is positive and finite and the
probability measure $\mu^y$ given by (\ref{41}) is well-defined.
\end{theorem}

\noindent{\bf Proof.} The boundedness of $Z(y)$ follows directly
from the lower bound on $\Phi$ in Assumptions \ref{a:4.1}, together
with the assumed integrability condition in the theorem. 
Since $u\sim\mu_0$ satisfies $u\in X'$ a.s.,
we have
$$Z(y)=\int_{X'}\exp\big(-\Phi(u;y)\big)\mu_0(du).$$
Note that $B'=X' \cap B$ is bounded in $X$.
Define
$$R_1:=\sup_{u\in B'}\|u\|_X<\infty.$$
Since $\Phi: X'\times Y \rightarrow
\R$ is continuous it is finite at every point in
$B'\times\{y\}$. 
Thus, by the continuity of $\Phi(\cdot;\cdot)$ implied by
Assumptions \ref{a:4.1}, we see that
$$\sup_{(u,y)\in B'\times B_Y(0,r)} \Phi(u;y)=R_2<\infty.$$
Hence
\begin{equation}
\label{eq:bfb}
Z(y)\geq\int_{B'}\exp(-R_2)\mu_0(du)=\exp(-R_2)\mu_0(B')>0.
\end{equation}
Since $\mu_0(B')$ is assumed positive and $R_2$ is finite
we deduce that $Z(y)>0$.  $\Box$

\begin{remarks}

The following remarks apply to the preceding and following theorem.

\begin{itemize}

\item In the preceding theorem we are not explicitly working in
a Bayesian setting: we are showing that, under the stated conditions
on $\Phi$, the measure is well-defined and normalizable. In Theorem \ref{t:3.2} 
we did not need to check normalizability because $\mu^y$ was defined
as a regular conditional probability, via Theorem \ref{t:3.1}, and therefore
automatically normalizable. 

\item The lower bound \eqref{eq:bfb} is used repeatedly in what follows,
without comment.

\item Establishing the integrability conditions for both the preceding and 
following theorem is often achieved for Gaussian $\mu_0$ by
appealing to the Fernique theorem.

\end{itemize}

\end{remarks}

\begin{theorem}
\label{t:4.3}
Let Assumptions \ref{a:4.1} hold. 
Assume that $\mu_0(X')=1$ and that $\mu_0(X' \cap B)>0$ for some
bounded set $B$ in $X$. 
Assume additionally that, for every fixed $r>0$,
$$\exp\big(M_1(r,\|u\|_X)\big)\Bigl(1+M_2(r,\|u\|_X)^2\Bigr)\in L_{\mu_0}^1(X;\R).$$
Then there is $C=C(r)>0$ such that, for all $y, y^\prime\in B_Y(0,r)$
$$\dhh(\mu^y,\mu^{y^\prime})\leq C\|y-y^\prime\|_Y.$$
\end{theorem}

\noindent{\bf Proof.}
Throughout this proof we use $C$ to denote a constant independent of $u$,
but possibly depending on the fixed value of $r$; it may change
from occurence to occurence. We use the fact that, since $M_2(r,\cdot)$ is
monotonic non-decreasing and strictly positive on $[0,\infty)$, 
\begin{subequations}
\label{eq:4.2}
\begin{align}
\exp\big(M_1(r,\|u\|_X)\bigr)M_2(r,\|u\|_X) & \le \exp\big(M_1(r,\|u\|_X)\bigr)
\Bigl(1+M_2(r,\|u\|_X)^2\Bigr), \label{eq:4.2a}\\
\exp\big(M_1(r,\|u\|_X)\bigr)& \le \exp\big(M_1(r,\|u\|_X)\bigr)
\Bigl(1+M_2(r,\|u\|_X)^2\Bigr). \label{eq:4.2b}
\end{align}
\end{subequations}
Let $Z=Z(y)$ and $Z'=Z(y')$ denote the normalization constants for
$\mu^y$ and $\mu^{y'}$ so that, by Theorem \ref{t:4.2},
\begin{align*}
Z&=\int_{X'} \exp\Bigl(-\Phi(u;y)\Bigr)\mu_0(du)>0,\\
Z'&=\int_{X'} \exp\Bigl(-\Phi(u;y')\Bigr)\mu_0(du)>0.
\end{align*}
Then, using the local Lipschitz property of the
exponential and the assumed Lipschitz continuity of $\Phi(u;\cdot)$,
together with \eqref{eq:4.2a}, we have
\begin{eqnarray*}
|Z-Z'|&\leq&\int_{X'}|\exp\big(-\Phi(u;y)\big)-\exp\big(-\Phi(u;y^\prime)\big)|
\mu_0(du)\\
&\leq&\int_{X'}\exp\big(M_1(r,\|u\|_X)\big)|\Phi(u;y)-\Phi(u;y^\prime)|\mu_0(du)\\
&\leq&\Big(\int_{X'}\exp\big(M_1(r,\|u\|_X)\bigr)M_2(r,\|u\|_X)\mu_0(du)\Big)\|y-y^\prime\|_Y\\
&\leq&\Big(\int_{X'}\exp\big(M_1(r,\|u\|_X)\bigr)(1+M_2(r,\|u\|_X)^2)\mu_0(du)\Big)\|y-y^\prime\|_Y\\
&\leq&C\|y-y^\prime\|_Y.
\end{eqnarray*}
The last line follows because
the integrand is in $L_{\mu_0}^1$ by assumption.
From the definition of \as{Hellinger distance}\index{Hellinger distance}
we have
$$\Bigl(\dhh(\mu^y,\mu^{y^\prime})\Bigr)^2\leq I_1+I_2,$$
where
\begin{align*}
I_1&=\frac{1}{Z}\int_{X'}\Bigl(
\exp\bigl(-\frac{1}{2}\Phi(u;y)\bigr)-\exp(-\frac{1}{2}\Phi(u;y')\bigr)\Bigr)^2
\mu_0(du),
\\
I_2&=\bigl|Z^{-\frac12}-(Z')^{-\frac12}\bigr|^2\int_{X'}
\exp(-\Phi(u;y')\bigr) \mu_0(du).
\end{align*}
Note that, again using similar Lipschitz calculations to those above,
using the fact that $Z>0$ and Assumptions \ref{a:4.1},
\begin{eqnarray*}
I_1&\leq&\frac{1}{4Z}\int_{X'}\exp\big(M_1(r,\|u\|_X)\big)|\Phi(u;y)-\Phi(u;y^\prime)|^2\mu_0(du)\\
&\leq&\frac{1}{Z}\Big(\int_{X'}\exp\big(M_1(r,\|u\|_X)\big)M_2(r,\|u\|_X)^2\mu_0(du)\Big)\|y-y^\prime\|_Y^2\\
&\leq&C\|y-y^\prime\|_Y^2.
\end{eqnarray*}
Also, using Assumptions \ref{a:4.1}, together with \eqref{eq:4.2b},
\begin{align*}
\int_{X'}\exp\big(-\Phi(u;y^\prime)\big)\mu_0(du)
&\leq \int_{X'}\exp\big(M_1(r,\|u\|_X)\big)\mu_0(du)\\
&<\infty.
\end{align*}
Hence
\begin{eqnarray*}
I_2\leq C\big(Z^{-3}\vee
(Z^\prime)^{-3}\big)|Z-Z^\prime|^2\leq C\|y-y^\prime\|_Y^2.
\end{eqnarray*}
The result is complete. $\Box$

\begin{remark}
\label{r:4.7x}
The \as{Hellinger metric}\index{Hellinger distance} 
has the very desirable property that
it translates directly into bounds on expectations.
For functions $f$ which are in $L^2_{\mu^y}(X;\R)$ and 
$L^2_{\mu^{y'}}(X;\R)$ the closeness of the \as{Hellinger
metric}\index{Hellinger distance} 
implies closeness of expectations of $f$. To be precise,
for $y,y' \in B_{Y}(0,r)$ we have
$$|\EE^{\mu^y} f(u)-\EE^{\mu^{y'}} f(u)| \le C\dhh(\mu^y,\mu^{y'})$$
where constant $C$ depends on $r$ and on the expectations of $|f|^2$
under $\mu^y$ and $\mu^{y'}$. It follows that
$$|\EE^{\mu^y} f(u)-\EE^{\mu^{y'}} f(u)| \le C\|y-y'\|,$$
for a possibly different constant $C$ which also depends on $r$ 
and on the expectations of $|f|^2$ under $\mu^y$ and $\mu^{y'}$.
\end{remark}

\vspace{0.2in}
\subsection{Approximation}
\label{ssec:4.2}

In this section we concentrate on continuity properties of
the \as{posterior}\index{posterior} 
measure with respect to approximation of the
potential $\Phi$. The methods used are very similar to those
in the previous subsection, and we establish a continuity property
of the posterior distribution, in the 
\as{Hellinger metric}\index{Hellinger distance}, with
respect to small changes in the potential $\Phi$.

Because the data $y$ plays no explicit role in this discussion,
we drop explicit reference to it.
Let $X$ be a Banach space and $\mu_0$ a measure on $X$. Assume that
$\mu$ and $\mu^N$ are both absolutely continuous with respect to
$\mu_0$ and given by
\begin{subequations}\label{42}
\begin{eqnarray}
\frac{d\mu}{d\mu_0}(u)&=&\frac{1}{Z}\exp\big(-\Phi(u)\big),\label{42a}\\
Z&=&\int_{X}\exp\big(-\Phi(u)\big)\mu_0(du) \label{42b}
\end{eqnarray}
\end{subequations}
and
\begin{subequations}\label{43}
\begin{eqnarray}
\frac{d\mu^N}{d\mu_0}(u)&=&\frac{1}{Z^N}\exp\big(-\Phi^N(u)\big),\label{43a}\\
Z^N&=&\int_{X}\exp\big(-\Phi^N(u)\big)\mu_0(du) \label{43b}
\end{eqnarray}
\end{subequations}
respectively. The measure $\mu^N$ might arise, for
example, through an approximation of the forward map $G$
underlying an inverse problem of the form \eqref{32}.
It is natural to ask whether closeness of the
forward map and its approximation imply closeness
of the \as{posterior}\index{posterior} measures. We now address this question.

\begin{assumptions}
\label{a:4.5}
Let $X'\subseteq X$ and assume that $\Phi\in
\CC(X'; \R)$. Assume
further that there are functions $M_i:\R^+ \to \R^+$, $i=1,2$,
independent of $N$ and
monotonic non-decreasing separately in each argument, and
with $M_2$ strictly positive, such that
for all $u\in X'$,  

$$\Phi(u)\geq -M_1(\|u\|_X),$$
$$\Phi^N(u)\geq -M_1(\|u\|_X),$$
$$|\Phi(u)-\Phi^N(u)|\leq M_2(\|u\|_X)\psi(N),$$
where $\psi(N) \to 0$ as $N \to \infty$.
$\quad\Box$
\end{assumptions}

The following two theorems are very similar to
Theorems \ref{t:4.2}, \ref{t:4.3} and the proofs
are adapted to estimate changes in the \as{posterior}\index{posterior}
caused by changes in the potential $\Phi$, rather than the
data $y$.

\begin{theorem}
\label{t:4.6}
Let Assumptions \ref{a:4.5} hold. 
Assume that $\mu_0(X')=1$ and that $\mu_0(X' \cap B)>0$ for some
bounded set $B$ in $X$. 
Assume additionally that, for every fixed $r>0$,
$$\exp\big(M_1(r,\|u\|_X)\big)\in L_{\mu_0}^1(X;\R).$$
Then $Z, Z^N$ given by (\ref{41b}), (\ref{42b}) are positive and finite 
and the probability measures $\mu$ and $\mu^N$ given by (\ref{41}), (\ref{42}) 
are well-defined.  Furthermore, for sufficiently large $N$, 
$Z^N$ given by (\ref{43b}) is bounded below by a positive 
constant independent of $N$.
\end{theorem}

\noindent{\bf Proof.} Finiteness of the normalization constants $Z$ and 
$Z^N$ follows from the lower bounds on $\Phi$ and $\Phi^N$ given in
Assumptions \ref{a:4.5}, together with the integrability condition in the
theorem. Since $u\sim\mu_0$ satisfies $u\in X'$ a.s.,
we have
$$Z=\int_{X'}\exp\big(-\Phi(u)\big)\mu_0(du).$$
Note that $B'=X' \cap B$ is bounded in $X$.
Thus
$$R_1:=\sup_{u\in B'}\|u\|_X<\infty.$$
Since $\Phi: X' \rightarrow
\R$ is continuous it is finite at every point in
$B'$.  Thus, by the properties of $|\Phi(\cdot)-\Phi^N(\cdot)|$  
implied by Assumptions \ref{a:4.5}, we see that
$$\sup_{u\in B'} \Phi(u)=R_2<\infty.$$
Hence
$$Z\geq\int_{B'}\exp(-R_2)\mu_0(du)=\exp(-R_2)\mu_0(B').$$
Since $\mu_0(B')$ is assumed positive and $R_2$ is finite
we deduce that $Z>0$.  
By Assumptions \ref{a:4.5} we may choose $N$ large enough so that
$$\sup_{u\in B'}|\Phi(u)-\Phi^N(u)| \le R_2$$
so that
$$\sup_{u\in B'} \Phi^N(u) \le 2R_2<\infty.$$
Hence
$$Z^N\geq\int_{B'}\exp(-2R_2)\mu_0(du)=\exp(-2R_2)\mu_0(B').$$
Since $\mu_0(B')$ is assumed positive and $R_2$ is finite
we deduce that $Z^N>0$. Furthermore, the lower bound is
independent of $N$, as required. 
$\Box$

\begin{theorem}
\label{t:4.7}
Let Assumptions \ref{a:4.5} hold. 
Assume that $\mu_0(X')=1$ and that $\mu_0(X' \cap B)>0$ for some
bounded set $B$ in $X$. 
Assume additionally that
$$\exp\big(M_1(\|u\|_X)\big)\Bigl(1+M_2(\|u\|_X)^2\Bigr)\in L_{\mu_0}^1(X;\R).$$
Then there is $C>0$ such that, for all $N$ sufficiently large, 
$$\dhh(\mu,\mu^{N})\leq C\psi(N).$$
\end{theorem}

\noindent{\bf Proof.}
Throughout this proof we use $C$ to denote a constant independent of $u$,
and $N$; it may change from occurrence to occurrence. 
We use the fact that, since $M_2(\cdot)$ is
monotonic non-decreasing and since it is strictly positive on
$[0,\infty)$, 
\begin{subequations}
\label{eq:4.22}
\begin{align}
\exp\big(M_1(\|u\|_X)\bigr)M_2(\|u\|_X) & \le  \exp\big(M_1(\|u\|_X)\bigr)
\Bigl(1+M_2(\|u\|_X)^2\Bigr), \label{eq:4.22a}\\
\exp\big(M_1(\|u\|_X)\bigr)& \le \exp\big(M_1(\|u\|_X)\bigr)\Bigl(1+M_2(\|u\|_X)^2\Bigr). \label{eq:4.22b}
\end{align}
\end{subequations}

Let $Z$ and $Z^N$ denote the normalization constants for
$\mu$ and $\mu^{N}$ so that for all $N$ sufficiently large, by Theorem \ref{t:4.6},
\begin{align*}
Z&=\int_{X'} \exp\Bigl(-\Phi(u)\Bigr)\mu_0(du)>0,\\
Z^N&=\int_{X'} \exp\Bigl(-\Phi^N(u)\Bigr)\mu_0(du)>0,
\end{align*}
with positive lower bounds independent of $N$.
Then, using the local Lipschitz property of the
exponential and the approximation property of $\Phi^N(\cdot)$ from
Assumptions \ref{a:4.5},
together with \eqref{eq:4.22a}, we have
\begin{eqnarray*}
|Z-Z^N|&\leq&\int_{X'}|\exp\big(-\Phi(u)\big)-\exp\big(-\Phi^N(u)\big)|
\mu_0(du)\\
&\leq&\int_{X'}\exp\big(M_1(\|u\|_X)\big)|\Phi(u)-\Phi^N(u)|\mu_0(du)\\
&\leq&\Big(\int_{X'}\exp\big(M_1(\|u\|_X)\bigr)M_2(\|u\|_X)\mu_0(du)\big)\Big)\psi(N)\\
&\leq&\Big(\int_{X'}\exp\big(M_1(\|u\|_X)\bigr)(1+M_2(\|u\|_X)^2)\mu_0(du)\big)\Big)\psi(N)\\
&\leq&C\psi(N).
\end{eqnarray*}
The last line follows because
the integrand is in $L_{\mu_0}^1$ by assumption.
From the definition of \as{Hellinger distance}\index{Hellinger distance}
we have
$$\Bigl(\dhh(\mu^y,\mu^{y^\prime})\Bigr)^2\leq I_1+I_2,$$
where
\begin{align*}
I_1&=\frac{1}{Z}\int_{X'}\Bigl(
\exp\bigl(-\frac{1}{2}\Phi(u)\bigr)-\exp(-\frac{1}{2}\Phi^N(u)\bigr)\Bigr)^2
\mu_0(du),
\\
I_2&=\bigl|Z^{-\frac12}-(Z^N)^{-\frac12}\bigr|^2\int_{X'}
\exp(-\Phi^N(u)\bigr) \mu_0(du).
\end{align*}
Note that, again by means of similar Lipschitz calculations to those above,
using the fact that $Z,Z^N>0$ uniformly for $N$
sufficiently large by Theorem \ref{t:4.6}, and Assumptions \ref{a:4.5},
\begin{eqnarray*}
I_1&\leq&\frac{1}{4Z}\int_{X'}\exp\big(M_1(\|u\|_X\big)|\Phi(u)-\Phi^N(u)|^2\mu_0(du)\\
&\leq&\frac{1}{Z}\Big(\int_{X'}\exp\big(M_1(\|u\|_X)\big)M_2(\|u\|_X)^2\mu_0(du)\Big)\psi(N)^2\\
&\leq&C\psi(N)^2.
\end{eqnarray*}
Also, using Assumptions \ref{a:4.5}, together with \eqref{eq:4.22b},
\begin{align*}
\int_{X'}\exp\big(-\Phi^N(u)\big)\mu_0(du)
&\leq \int_{X'}\exp\big(M_1(\|u\|_X)\big)\mu_0(du)\\
&<\infty,
\end{align*}
and the upper bound is independent of $N$. Hence
\begin{eqnarray*}
I_2\leq C\big(Z^{-3}\vee
(Z^N)^{-3}\big)|Z-Z^N|^2\leq C\psi(N)^2.
\end{eqnarray*}
The result is complete. $\Box$

\begin{remarks}
\label{r:4.7}
The following two remarks are relevant to establishing the
conditions of the preceding two theorems, and to applying them.

\begin{itemize}

\item As mentioned in the previous susbection concerning well-posedness,
the Fernique Theorem can frequently be used to establish integrability
conditions, such as those in the two preceding theorems when $\mu_0$
is Gaussian.

\item Using the ideas underlying Remark \ref{r:4.7x},
the preceding theorem enables us to translate errors arising
from approximation of the forward problem into
errors in the Bayesian solution of the
inverse problem. Furthermore, the errors in the forward
and inverse problems scale the
same way with respect to $N$.
For functions $f$ which are in $L^2_{\mu}$ and $L^2_{\mu^N}$,
uniformly with respect to $N$, the closeness of the 
\as{Hellinger metric}\index{Hellinger distance}
implies closeness of expectations of $f$:
$$|\EE^{\mu} f(u)-\EE^{\mu^N} f(u)| \le C\psi(N).$$

\end{itemize}

\end{remarks}

\vspace{0.2in}
\subsection{MAP Estimators and Tikhonov Regularization}
\label{ssec:map}

The aim of this section is to connect the probabilistic approach
to inverse problems with the classical method of Tikhonov regularization.
We consider the setting in which the \as{prior}\index{prior!Gaussian} 
measure is a Gaussian.
We then show that MAP estimators, points of maximal probability,
coincide with minimizers of a Tikhonov-Phillips regularized
least-squares function, with regularization being with respect to
the Cameron-Martin norm of the Gaussian prior. 
The data $y$ plays no explicit role
in our developments here and so we work in the setting of equation 
\eqref{42}. Recall, however, that in the context of inverse problems,
a classical methodology is to simply try and minimize (subject
to some regularization) $\Phi(u)$. Indeed for finite data and
Gaussian observational noise with Gaussian distribution $N(0,\Gamma)$
we have 
$$\Phi(u)=\frac{1}{2}\bigl|
\Gamma^{-\frac{1}{2}}\bigl(y-G(u)\bigr)\bigr|^2.$$
Thus $\Phi$ is simply a covariance weighted model-data
misfit least squares function. 

In this section we show that maximizing probability
under $\mu$ (in a sense that we will make precise in what follows)
is equivalent to minimizing
\begin{equation}\label{eq:I}
  I(u) =
  \begin{cases}
    \Phi(u) + \frac12\|u\|_E^2 & \mbox{if $u \in E$, and} \\
    +\infty & \mbox{else.}
  \end{cases}
\end{equation}
Here $(E,\|\cdot\|_E)$ denotes the Cameron-Martin space associated
to $\mu$.
We view $\mu$ as a Gaussian probability measure on a separable Banach space
$(X,\|\cdot\|_X)$ so that $\mu_0(X)=1$. We make the following assumptions
about the function $\Phi:$

\begin{assumption} \label{a:asp1}
  The function $\Phi\colon X\to\R$ satisfies the following conditions:
  \begin{itemize}
  \item[(i)] For every $\eps>0$ there is an $M=M(\epsilon)\in \R$, such that
    for all $u\in X$,
    \begin{equation*}
      \Phi(u) \geq M -\eps\|u\|_X^2.
    \end{equation*}
  \item[(ii)] $\Phi$ is locally bounded from above, \textit{i.e.}\ for
    every $r>0$ there exists $K=K(r)>0$ such that, for all $u\in X$
    with $\|u\|_X<r$ we have
    \begin{equation*}
      \Phi(u) \leq K.
    \end{equation*}
  \item[(iii)] $\Phi$ is locally Lipschitz continuous, \textit{i.e.}\
    for every $r>0$ there exists $L=L(r)>0$ such that for all
    $u_1,u_2\in X$ with $\|u_1\|_X,\|u_2\|_X < r$ we have
    \begin{equation*}
      |\Phi(u_1)-\Phi(u_2)| \leq L\|u_1-u_2\|_X.
    \end{equation*}
  \end{itemize}
\end{assumption}

In finite dimensions, for measures which have a continuous density with respect to
Lebesgue measure, there is an obvious notion of most likely point(s):
simply the point(s) at which the Lebesgue density is maximized.
This way of thinking does not translate into the infinite dimensional
context, but there is a way of restating it which does. Fix a small
radius $\delta>0$ and identify centres of balls of radius $\delta$ 
which have maximal probability. Letting $\delta \to 0$ then recovers
the preceding definition, when there is a continuous Lebesgue
density. We adopt this small ball approach in the infinite dimensional setting.

For $z \in E$, let $B^\delta(z)\subset X$ be the open ball centred at
$z\in X$ with radius~$\delta$ in~$X$.  Let
\begin{equation*}
  J^{\delta}(z)
  = \mu\bigl( B^\delta(z) \bigr)
\end{equation*}
be the mass of the ball~$B^\delta(z)$ under the measure $\mu$.
Similarly we define
\begin{equation*}
  J_0^{\delta}(z)
  = \mu_0\bigl( B^\delta(z) \bigr)
\end{equation*}
the mass of the ball~$B^\delta(z)$ under the Gaussian \as{prior}\index{prior!Gaussian}. Recall that all balls in a separable Banach space 
have positive Gaussian measure, by Theorem \ref{theo:balls};
it thus follows that $J_0^\delta(z)$ is finite and positive for any $z \in E.$
By Assumptions~\ref{a:asp1}(i),(ii) together with the Fernique Theorem 
\ref{t:2.5} the same is true for $J^\delta(z).$
Our first theorem encapsulates the idea that probability is maximized
where $I$ is minimized. To see this, fix any point $z_2$ in the
Cameron-Martin space $E$ and notice that the probability of the small
ball at $z_1$ is maximized, asymptotically as the radius of the ball
tends to zero, at minimizers of $I$.

\begin{theorem}\label{t:OM}
  Let Assumptions~\ref{a:asp1} hold and assume that $\mu_0(X)=1$.  Then the function $I$ defined
  by~\eqref{eq:I} satisfies, for any $z_1,z_2\in E$, 
  \begin{equation*}
    \lim_{\delta\to 0}\frac{J^\delta (z_1)}{J^\delta(z_2)}
    =\exp\bigl(I(z_2)-I(z_1)\bigr).
  \end{equation*}
\end{theorem}

\begin{proof}
  Since $J^\delta(z)$ is finite and positive for any $z \in E$ 
the ratio of interest is finite and positive. The key estimate in the proof
  is given in Theorem \ref{t:smallball}: 
  \begin{equation}\label{eq:need}
    \lim_{\delta\to 0}
\frac{J_0^\delta (z_1)}{J_0^\delta(z_2)}
    = \exp\left( \frac12\|z_2\|_E^2-\frac12\|z_1\|_E^2 \right).
  \end{equation}
  This estimate transfers questions
  about probability, naturally asked on the space $X$ of full measure
  under $\mu_0$, into statements concerning the Cameron-Martin norm of
  $\mu_0$; note that under this norm a random variable distributed
as $\mu_0$ is almost surely infinite so the result is non-trivial.

 We have
  \begin{align*}
    \frac{J^\delta(z_1)}{J^\delta(z_2)}
    &=\frac{ \int_{B^\delta(z_1)}\exp (-\Phi(u))\,\mu_0(\ud u) }
    { \int_{B^\delta(z_2)}\exp (-\Phi(v))\,\mu_0(\ud v) } \\
    &=\frac{ \int_{B^\delta(z_1)}\exp (-\Phi(u)+\Phi(z_1))\exp(-\Phi(z_1))\,\mu_0(\ud u) }
    { \int_{B^\delta(z_2)}\exp (-\Phi(v)+\Phi(z_2))\exp (-\Phi(z_2))\,\mu_0(\ud v) }.
  \end{align*}
  By Assumption~\ref{a:asp1} (iii) there is $L=L(r)$ such that,
for all $u,v\in X$ with $\max\{\|u\|_X,\|v\|_X\}<r$,
  \begin{equation*}
    -L\,\|u-v\|_X\,\le\,\Phi(u)-\Phi(v)\,\le\,L\,\|u-v\|_X.
  \end{equation*}
  If we define $L_1=L(\|z_1\|_X+\delta)$ and $L_2=L(\|z_2\|_X+\delta)$ then
we have
  \begin{align*}
    \frac{J^\delta(z_1)}{J^\delta(z_2)}
    &\le \e^{\delta(L_1+L_2)}\frac{ \int_{B^\delta(z_1)}\exp (-\Phi(z_1))\,\mu_0(\ud u) }
    { \int_{B^\delta(z_2)}\exp (-\Phi(z_2))\,\mu_0(\ud v) } \\
    &=\e^{\delta(L_1+L_2)}\e^{-\Phi(z_1)+\Phi(z_2)}
\frac{  \int_{B^\delta(z_1)}\,\mu_0(\ud u) }
    { \int_{B^\delta(z_2)}\,\mu_0(\ud v) }.
  \end{align*}
  Now, by \eqref{eq:need}, we have
  \begin{equation*}
    \frac{J^\delta(z_1)}{J^\delta(z_2)}
    \le r_1(\delta)\,\e^{\delta (L_2+L_1)}\e^{-I(z_1)+I(z_2)}
  \end{equation*}
  with $r_1(\delta)\to 1$ as $\delta\to 0$. Thus
  \begin{equation}\label{e:Jlims}
    \limsup_{\delta\to 0}\frac{J^\delta(z_1)}{J^\delta(z_2)}
    \,\le\, \e^{-I(z_1)+I(z_2)}.
  \end{equation}
  Similarly we obtain
  \begin{equation*}
    \frac{J^\delta(z_1)}{J^\delta(z_2)}
    \ge \frac{1}{r_2(\delta)}\,\e^{-\delta (L_2+L_1)}\e^{-I(z_1)+I(z_2)}
  \end{equation*}
   with $r_2(\delta)\to 1$ as $\delta\to 0$ and deduce that
  \begin {equation}\label{e:Jlimi}
    \liminf_{\delta\to 0}\frac{J^\delta(z_1)}{J^\delta(z_2)}
    \,\ge\, \e^{-I(z_1)+I(z_2)}
\end{equation}
  Inequalities \eqref{e:Jlims} and \eqref{e:Jlimi} give the desired
  result.
\end{proof}

We have thus linked the Bayesian approach to inverse problems with
a classical regularization technique. We conclude the subsection
by showing that, under the prevailing Assumptions \ref{a:asp1}, 
the minimization problem for $I$ is well-defined. We first recall a
basic definition and lemma from the calculus of variations.

\begin{definition} The function $I:E \to \R$ is {\em weakly
lower semicontinuous} if
$$\liminf_{n \to \infty} I(u_n) \ge I(u)$$
whenever $u_n \weakc u$ in $E$.
The function $I:E \to \R$ is {\em weakly
continuous} if
$$\lim_{n \to \infty} I(u_n) = I(u)$$
whenever $u_n \weakc u$ in $E$.
$\quad\Box$
\end{definition}

Clearly weak continuity implies weak lower
semicontinuity.

\begin{lemma} \label{lem:wlsc}
If $(E,\langle\cdot,\cdot\rangle_{E})$ is a Hilbert
space with induced norm $\|\cdot\|_{E}$
then the quadratic form $J(u):=\frac12\|u\|_{E}^2$ is
weakly lower semicontinuous.
\end{lemma}

\begin{proof} The result follows from the fact that
\begin{align*}
J(u_n)-J(u)&=\frac12\|u_n\|_E^2-\frac12\|u\|_E^2\\
&=\frac12\langle u_n-u,u_n+u\rangle_E\\
&=\frac12 \langle u_n-u,2u\rangle_E+\frac12\|u_n-u\|_E^2\\
&\ge \frac12 \langle u_n-u,2u\rangle_E.
\end{align*}
But the right hand side tends to zero since $u_n \weakc u$ in $E$.
Hence the result follows. \end{proof}

\begin{theorem} \label{t:comeon}
Suppose that Assumptions \ref{a:asp1} hold
and let $E$ be a Hilbert space compactly embedded
in $X$.
Then there exists $\bu \in E$ such that
$$I(\bu)=\bI:=\inf\{I(u): u\in E\}.$$
Furthermore, if $\{u_n\}$ is a minimizing sequence
satisfying $I(u_n) \to I(\bu)$ then there is a subsequence
$\{u_{n'}\}$ that converges strongly to $\bu$ in $E$.
\end{theorem}

\begin{proof}

Compactness of $E$ in $X$
implies that, for some universal constant $C$,
$$\|u\nx^2 \le C\|u\|_{E}^2.$$
Hence, by Assumption \ref{a:asp1}(i),
it follows that, for any $\eps>0$,
there is $M(\eps) \in \R$ such that
$$\Bigl(\frac12-C\eps\Bigr)\|u\|_{E}^2+
M(\eps)\le I(u).$$
By choosing $\eps$ sufficiently small,
we deduce that there is $M \in \R$ such that,
for all $u \in E$,
\begin{equation}
\label{eq:coerce}
\frac14\|u\|_{E}^2+M \le I(u).
\end{equation}

Let $u_n$ be an infimizing sequence satisfying $I(u_n) \to \bI$
as $n \to \infty$.
For any $\delta>0$ there is $N=N_1(\delta)$:
\begin{equation}
\label{eq:refer}
\bI \le I(u_n) \le \bI+\delta, \quad \forall n \ge N_1.
\end{equation}
Using \eqref{eq:coerce} we deduce that
the sequence $\{u_n\}$ is bounded in $E$ and,
since $E$ is a Hilbert space,
there exists $\bu \in E$ such that
$u_n \weakc \bu$ in $E$. By the compact embedding of $E$ in $X$
we deduce that $u_n \to \bu$,
strongly in $X$. By the Lipschitz continuity of $\Phi$
in $X$ (Assumption \ref{a:asp1}(iii))
we deduce that $\Phi(u_n) \to \Phi(\bu)$.
Thus $\Phi$ is weakly continuous on $E$.
The functional
$J(u):=\frac12\|u\|_{E}^2$ is weakly lower
semicontinuous on $E$ by Lemma \ref{lem:wlsc}. 
Hence $I(u)=J(u)+\Phi(u)$ is weakly lower
semicontinuous on $E$.
Using this fact in \eqref{eq:refer}
it follows that, for any $\delta>0$,
$$\bI \le I(\bu) \le \bI+\delta.$$
Since $\delta$ is arbitrary the first result follows.

By passing to a further subsequence,
and for $n,\ell \ge N_2(\delta)$,
\begin{align*}
\frac14\|u_n-u_{\ell}\|_{E}^2&=\frac12\|u_n\|_{E}^2+\frac12\|u_{\ell}\|_{E}^2-
\frac14\|u_n+u_{\ell}\|_{E}^2\\
&=I(u_n)+I(u_{\ell})-2I\Bigl(\frac12(u_n+u_{\ell})\Bigr)
-\Phi(u_n)-\Phi(u_{\ell})+2\Phi\Bigl(\frac12(u_n+u_{\ell})\Bigr)\\
&\le 2(\bI+\delta)-2\bI-\Phi(u_n)-\Phi(u_{\ell})+
2\Phi\Bigl(\frac12(u_n+u_{\ell})\Bigr)\\
&\le 2\delta-\Phi(u_n)-\Phi(u_{\ell})+2\Phi\Bigl(\frac12(u_n+u_{\ell})\Bigr).
\end{align*}
But $u_n, u_{\ell}$ and $\frac12(u_n+u_{\ell})$ all converge strongly to $\bu$ in $X$.
Thus, by continuity of $\Phi$, we deduce that
for all $n,\ell \ge N_3(\delta)$,
$$\frac14\|u_n-u_{\ell}\|_{E}^2 \le 3\delta.$$
Hence the sequence is Cauchy in $E$ and converges
strongly and the proof is complete.
\end{proof}

\begin{corollary}\label{c:comeon}
Suppose that Assumptions \ref{a:asp1} hold
and the Gaussian measure $\mu_0$ with Cameron-Martin space space $E$ 
satisfies $\mu_0(X)=1$.
Then there exists $\bu \in E$ such that
$$I(\bu)=\bI:=\inf\{I(u): u\in E\}.$$
Furthermore, if $\{u_n\}$ is a minimizing sequence
satisfying $I(u_n) \to I(\bu)$ then there is a subsequence
$\{u_{n'}\}$ that converges strongly to $\bu$ in $E$.

\end{corollary}

\begin{proof}
By Theorem \ref{t:CMcompact}, $E$ is compactly embedded in $X$.
Hence the result follows by Theorem \ref{t:comeon}.

\end{proof}

\vspace{0.2in}

\subsection{Bibliographic Notes}

\begin{itemize}

\item Subsection \ref{ssec:4.1}. The well-posedness theory described
here was introduced in the papers \cite{CDRS08} and \cite{article:Stuart2010}.
Relationships between the \as{Hellinger dostance}\index{Hellinger distance} on probability measures,
and the Total Variation
distance and Kullback-Leibler divergence may be found in
~\cite{GS02},~\cite{Poll}, as well as in \cite{article:Stuart2010}.

\item Subsection \ref{ssec:4.2}.
Generalization of the well-posedness theory to study the effect of
numerical approximation of the forward model on the inverse problem
may be found in \cite{CDS10}.
The relationship between expectations and \as{Hellinger distance}\index{Hellinger distance},
as used in Remark \ref{r:4.7}, is demonstrated
in \cite{article:Stuart2010}.

\item Subsection \ref{ssec:map}. The connection between Tikhonov-Phillips
regularization and MAP estimators is widely appreciated in 
computational Bayesian inverse problems; see \cite{KS05}.
Making the connection rigorous in the separable Banach
space setting is the subject of the paper \cite{DLSV13}; 
further references to the historical development of the subject
may be found therein. 
Related to Lemma \ref{lem:wlsc} see also
\cite[Chapter 3]{Daco09}.

\end{itemize}

\vspace{0.3in}
\section{Measure Preserving Dynamics}
\label{sec:mpd}

The aim of this section 
is to study Markov processes, in continuous time, and Markov chains,
in discrete time, which preserve the measure $\mu$ given by
\eqref{42}. The overall setting is described in subsection \ref{ssec:set},
and introduces the role of detailed balance and reversibility in
constructing measure-preserving Markov chains/processes.  
Subsection \ref{ssec:4.4} 
concerns Markov chain-Monte Carlo (MCMC) methods; these are
Markov chains which are invariant with respect to $\mu$. Metropolis-Hastings
methods are introduced and the role of detailed balance in their
construction is explained. The benefits of conceiving MCMC methods
which are defined on the infinite dimensional space is emphasized. 
In particular, the idea of using proposals which 
preserve the \as{prior}\index{prior}, more specificallty which are prior reversible,
is introduced as an example. 
In subsection \ref{ssec:pf} we show how sequential Monte Carlo
(SMC) methods can be used to construct approximate samples from
the measure $\mu$ given by \eqref{42}. Again our perspective
is to construct algorithms which are provably well-defined
on the infinite dimensional space and in fact we find an upper
bound for the  approximation error of the SMC method 
which proves its convergence on an infinite dimensional space. 
The MCMC methods from the previous
section play an important role in the construction of these
SMC methods. 
Subsections \ref{ssec:start}--\ref{ssec:IDL} concern continuous time
$\mu$-reversible processes. In particular they 
concern derivation and study of
a  Langevin equation which is invariant with respect to the
measure $\mu$.
(Note that this is called the overdamped Langevin equation for a physicist,
the plain Langevin equation for a statistician.)
In continuous time we work entirely in the case
of Gaussian \as{prior}\index{prior!Gaussian} 
measure $\mu_0$ on Hilbert space $\mathcal {H}$ 
with inner-product and norm denoted by
$\langle \cdot, \cdot \rangle$ and $\|\cdot\|$ respectively;
however in discrete time our analysis
is more general, applying on a separable Banach space $(X,\|\cdot\|)$ 
and for quite general \as{prior}\index{prior} measure.

\vspace{0.2in}
\subsection{General Setting}
\label{ssec:set}

This section is devoted to Banach space 
valued Markov chains or processes which are invariant with respect
to the \as{posterior}\index{posterior} measure $\mu^y$ constructed in subsection 
\ref{ssec:3.2}. Within this section, the data $y$ arising in
the inverse problems plays no explicit role; 
indeed the theory applies to a wide range of measures
$\mu$ on separable Banach space $X$.
Thus the discussion in this chapter includes, but is not limited to, 
Bayesian inverse problems.  All of the Markov chains we construct 
will exploit structure in a reference measure $\mu_0$ 
with respect to which the measure
$\mu$ is absolutely continuous; thus $\mu$ has a density 
with respect to $\mu_0$.
In continuous time we will explicitly use the Gaussianity
of $\mu_0$, but in discrete time we will be more general.

Let $\mu_0$ be a reference measure on the 
separable Banach space $X$ 
equipped with the Borel $\sigma$-algebra $\cBc(X).$
We assume that $\mu\ll\mu_0$ is given by
\begin{subequations}\label{045}
\begin{eqnarray}
\frac{d\mu}{d\mu_0}(u)&=&\frac{1}{Z}\exp\big(-\Phi(u)\big),\label{045a}\\
Z&=&\int_{X}\exp\big(-\Phi(u)\big)\mu_0(du), \label{045b}
\end{eqnarray}
\end{subequations}
where $Z \in (0,\infty)$.
In the following we let $P(u,dv)$ denote a Markov transition kernel so
that $P(u,\cdot)$ is a probability measure on $\bigl(X,\cBc(X)\bigr)$ for
each $u \in X$. Our interest is in probability kernels which preserve $\mu$.

\begin{definition} 
\label{d:db}
The Markov chain with transition kernel $P$
is {\em invariant} with respect to $\mu$ if
$$\int_{X}\mu(du)P(u,\cdot)=\mu(\cdot)$$
as measures on $\bigl(X,\cBc(X)\bigr)$.
The Markov kernel is said to satisfy {\em detailed
balance} with respect to $\mu$ if
$$\mu(du)P(u,dv)=\mu(dv)P(v,du)$$
as measures on $\bigl(X \times X,\cBc(X) \otimes \cBc(X)\bigr)$.
The resulting Markov chain is then said to be {\em reversible} with respect
to $\mu$.
$\quad\Box$
\end{definition}

It is straightforward to see, by integrating the detailed balance condition
with respect to $u$ and using  the fact that $P(v,du)$ is a Markov kernel,
the following:

\begin{lemma} A Markov chain which is reversible with respect to $\mu$
is also invariant with respect to $\mu$.
\end{lemma}

Reversible Markov chains and processes
 arise naturally in many physical systems which
are in statistical equilibrium. They are also important, however, as a means
of {\em constructing} Markov chains which are invariant with respect to
a given probability measure.
We demonstrate this in subsection \ref{ssec:4.4} where we consider
the Metropolis-Hastings variant of MCMC methods.
Then, in subsections \ref{ssec:start}, \ref{ssec:FDL} and \ref{ssec:IDL},
we move to continuous time Markov processes. In particular we
show that the equation
\begin{equation}\label{44}
\frac{du}{dt}=-u-\cC D\Phi(u)+\sqrt{2}\frac{dW}{dt}, \;\,\, u(0)=u_0,
\end{equation}
preserves the measure $\mu$, where $W$ is a $\cC$-\as{Wiener process}\index{Wiener!process},
defined below in subsection \ref{sec:Ito}. 
Precisely we show that if $u_0\sim\mu$, independently of
the driving Wiener process, then
$\bbE\varphi\big(u(t)\big)=\bbE\varphi(u_0)$ for all $t>0$
for continuous bounded $\varphi$ defined on an appropriately
chosen subspaces, under boundedness conditions
on $\Phi$ and its derivatives.

\begin{example}
\label{ex:rev1}
Consider the (measurable) Hilbert space $\bigl(\cH,\cBc(\cH)\bigr)$
equipped, as usual, with the Borel $\sigma$-algebra.
Let $\mu$ denote the Gaussian measure $N(0,C)$ on $\cH$ and, for fixed $u$,
let $P(u,dv)$ denote the Gaussian measure 
$N\bigl((1-\beta^2)^{\frac12}u,\beta^2C\bigr)$, also viewed 
as a probability measure on $\cH.$ Thus $v \sim P(u,dv)$ 
can be expressed as $v=(1-\beta^2)^{\frac12}u+
\beta \xi$ where $\xi \sim N(0,C)$ is independent of $u$. We show that
$P$ is reversible, and hence invariant, with respect to $\mu$. 
To see this we note that $\mu(du)P(u,dv)$ 
is a centred Gaussian measure on $\cH \times \cH$, equipped with the 
$\sigma$-algebra $\cBc(\cH)\otimes \cBc(\cH)$.
The covariance of the jointly varying random variable is
characterized by the identities
\begin{align}\label{e:mGcov}
\bbE u \otimes u=C,\; \bbE v \otimes v=C,\; \bbE u \otimes v=(1-\beta^2)^{\frac12}C.
\end{align}
Indeed, letting $\nu(du,dv):=\mu(du)P(u,dv)$, and with
$\langle\cdot,\cdot\rangle$ and $\|\cdot\|$ the inner product and norm 
on $\cH$ respectively, we can write, using (\ref{e:FTG}),
\begin{align*}
\hat\nu(d\xi,d\eta)
&=\int_{\cH \times \cH} \e^{i\langle u,\xi\rangle+i\langle v,\eta\rangle}\mu(du)P(u,dv)\\
&=\int_\cH\e^{i\langle u,\xi\rangle}\int_\cH \e^{i\langle v,\eta\rangle} \,P(u,dv)\,\mu(du)\\
&=\int_\cH  \e^{i\langle u,\xi\rangle} \e^{i\sqrt{1-\beta^2}\langle u,\eta\rangle-\frac{1}{2}\|\beta C^{\frac12}\eta\|^2}\mu(du)\\
&=\e^{-\frac{\beta^2}{2}\| C^{\frac12}\eta\|^2}\int_\cH  \e^{i\langle u,\sqrt{1-\beta^2}\,\eta+\xi\rangle}\mu(du)\\
&=\e^{-\frac{\beta^2}{2}\| C^{\frac12}\eta\|^2}\e^{-\frac{1}{2}\| C^{\frac12}(\sqrt{1-\beta^2}\,\eta+\xi)\|^2}\\
&=\exp\left({-\frac{1}{2}\| C^{\frac12}\eta\|^2-\frac{1}{2}\| C^{\frac12}\xi\|^2-({1-\beta^2})^{\frac12}\langle C^{\frac12}\xi,C^{\frac12}\eta\rangle}\right).
\end{align*}
Hence, by Lemma \ref{prop:FT} and equation (\ref{e:FTG}), $\mu(du)P(u,dv)$ 
is a centred Gaussian measure with the \as{covariance operator}\index{covariance operator} given by (\ref{e:mGcov}).
Since the expression in the last line of the above equation 
is symmetric in $\xi$ and $\eta$, 
$\mu(dv)P(v,du)$ is a centred Gaussian measure with
the same covariance as $\mu(du)P(u,dv)$ and so the reversibility is proved.
$\quad\Box$
\end{example}

\begin{example}
\label{ex:MHs}
Consider the equation
\begin{equation}\label{a44}
\frac{du}{dt}=-u+\sqrt{2}\frac{dW}{dt}, \,\, u(0)=u_0,
\end{equation}
where $W$ is a $\cC$-\as{Wiener process}\index{Wiener!process} (defined in
subsection \ref{sec:Ito} below).
Then
$$u(t)=e^{-t}u_0+\sqrt{2}\int_0^t e^{-(t-s)}dW(s).$$
Use of the It\^o isometry demonstrates that $u(t)$ 
is distributed according to the Gaussian $N\bigl(e^{-t}u_0,(1-e^{-2t}){\mathcal C}\bigr)$.
Setting $\beta^2=1-e^{-2t}$ and employing the previous example shows that
the Markov process is reversible since, for every $t>0$, the transition
kernel of the process is reversible.
$\quad\Box$
\end{example}

\vspace{0.2in}
\subsection{Metropolis-Hastings Methods}
\label{ssec:4.4}

In this section we study Metropolis-Hastings methods designed
to sample from the probability measure $\mu$ given by \eqref{045}.
The perspective that we have described on inverse problems,
specifically the formulation of Bayesian inversion on function
space, leads to new sampling methods which are specifically tailored
to the high dimensional problems which arise from
discretization of the infinite dimensional setting. 
In particular it leads naturally to the philosophy that it is
advantageous to design algorithms which, in principle, make
sense in infinite dimensions; it is these methods which will
perform well under refinement of finite dimensional approximations. 
Most Metropolis-Hastings methods which are defined in finite 
dimensions will not make sense in the infinite dimensional limit.
This is because the acceptance probability for Metropolis-Hastings methods
is defined as the Radon-Nikodym derivative between two measures
describing the behaviour of the Markov chain in stationarity.
Since measures in infinite dimensions have a tendency to be mutually 
singular, only carefully designed methods will have interpretations
in infinite dimensions. To simplify the presentation we work with
the following assumptions throughout:

\begin{assumptions}
\label{ass:bdd}
The function $\Phi:X\rightarrow\R$ 
is bounded on bounded subsets of $X$.
$\quad\Box$
\end{assumptions}

We now consider the following prototype
Metropolis-Hastings method which accept-rejects proposals
from a Markov kernel $Q$ to produce a Markov chain with kernel
$P$ which is reversible with respect to $\mu$.

\begin{algorithm}\label{alg:pmh}
Given $a:X \times X \to [0,1]$ generate $\{u^{(k)}\}_{k \ge 0}$ as follows:
\begin{itemize}

\item [1] Set $k=0$ and pick $u^{(0)} \in X$.

\item [2] Propose $v^{(k)} \sim Q(u^{(k)},dv)$.

\item [3] Set $u^{(k+1)}=v^{(k)}$ with probability $a(u^{(k)},v^{(k)})$, 
independently of $(u^{(k)},v^{(k)})$.

\item [4] Set $u^{(k+1)}=u^{(k)}$ otherwise.

\item [5] $k \to k+1$ and return to 2.

\end{itemize}
$\Box$
\end{algorithm}

Given a proposal kernel $Q$, a key question in the design of
MCMC methods is the question of how to 
choose $a(u,v)$ to ensure that $P(u,dv)$ satisfies
detailed balance with respect to $\mu$. If the resulting Markov chain
is ergodic this then yields an algorithm which, asymptotically,
samples from $\mu$, and can be used to estimate expectations against $\mu$.

To determine conditions on $a$ which
are necessary and sufficient for detailed
balance we first note that the Markov kernel
which arises from accepting/rejecting proposals from $Q$ is given by
\begin{equation}
\label{eq:sulp1}
P(u,dv)=Q(u,dv)a(u,v)+\delta_u(dv)\int_{X}\bigl(1-a(u,w)\bigr) Q(u,dw).
\end{equation}
Notice that 
$$\int_{X} P(u,dv)=1$$
as required.  Substituting the expression for $P$ into the 
detailed balance condition from Definition \ref{d:db} we obtain 
\begin{eqnarray*}
&\mu(du)Q(u,dv)a(u,v)+\mu(du)\delta_u(dv)\int_{X}\bigl(1-a(u,w)\bigr) Q(u,dw)\\
=&\\
&\mu(dv)Q(v,du)a(v,u)+\mu(dv)\delta_v(du)\int_{X}\bigl(1-a(v,w)\bigr) Q(v,dw).
\end{eqnarray*}
We now note that the measure $\mu(du)\delta_u(dv)$ is in fact symmetric
in the pair $(u,v)$ and that $u=v$ almost surely under it. As a consequence
the identity reduces to
\begin{equation}
\label{eq:orez1}
\mu(du)Q(u,dv)a(u,v)=\mu(dv)Q(v,du)a(v,u).
\end{equation}

Our aim now is to identify choices of $a$ which ensure that \eqref{eq:orez1}
is satisfied. This will then ensure that the prototype algorithm does
indeed lead to a Markov chain for which $\mu$ is invariant. To this end we  
define the measures
$$\nu(du,dv)=\mu(du)Q(u,dv)$$
and 
$$\ntu(du,dv)=\mu(dv)Q(v,du)$$ 
on
$\bigl(X \times X,\cBc(X) \otimes \cBc(X)\bigr)$.
The following theorem determines a necessary and sufficient
condition for the choice of $a$ to make the algorithm $\mu$
reversible, and identifies the canonical Metropolis-Hastings choice.

\begin{theorem}
\label{t:mh1}
Assume that $\nu$ and $\ntu$ are equivalent as measures on 
$X \times X$, equipped with the 
$\sigma$-algebra $\cBc(X)\otimes \cBc(X)$,
and that $\nu(du,dv)=r(u,v)\ntu(du,dv)$.
Then the probability kernel \eqref{eq:sulp1} satisfies
detailed balance if and only if
\begin{equation}
\label{eq:reggad1}
r(u,v)a(u,v)=a(v,u), \quad \nu\mbox{-a.s.}\,.
\end{equation}
In particular the choice $\alpha_{\rm mh}(u,v)=\min\{1,r(v,u)\}$
will imply detailed balance.
\end{theorem}

\begin{proof}
Since $\nu$ and $\ntu$ are equivalent \eqref{eq:orez1} holds if and only if
$$\frac{d\nu}{d\ntu}(u,v)a(u,v)=a(v,u).$$
This is precisely \eqref{eq:reggad1}. Now note that
$\nu(du,dv)=r(u,v)\ntu(du,dv)$ and $\ntu(du,dv)=r(v,u)\nu(du,dv)$ since
$\nu$ and $\ntu$ are equivalent. Thus $r(u,v)r(v,u)=1$. It follows that
\begin{align*}
r(u,v)\alpha_{\rm mh}(u,v) &=\min\{r(u,v),r(u,v)r(v,u)\}\\
&=\min\{r(u,v),1\}\\
&=\alpha_{\rm mh}(v,u)
\end{align*}
as required.
\end{proof}

A good example of the resulting methodology arises in the case where
$Q(u,dv)$ is reversible with respect to $\mu_0:$

\begin{theorem}
\label{t:pr} Let Assumption \ref{ass:bdd} hold.
Consider Algorithm \ref{alg:pmh} applied to
\eqref{045} in the case where the proposal kernel $Q$ is
reversible with respect to the \as{prior}\index{prior} 
$\mu_0$. Then the resulting
Markov kernel $P$ given by \eqref{eq:sulp1} is 
reversible with respect to $\mu$ if $a(u,v)=\min\{1,\exp\bigl(\Phi(u)-\Phi(v)\bigr)\}$.  
\end{theorem}

\begin{proof}
Prior reversibility implies that
$$\mu_0(du)Q(u,dv)=\mu_0(dv)Q(v,du).$$
Multiplying both sides by $\exp\bigl(-\Phi(u)\bigr)$
gives
$$\mu(du)Q(u,dv)=\exp\bigl(-\Phi(u)\bigr)\mu_0(dv)Q(v,du)$$
and then multiplication by $\exp\bigl(-\Phi(v)\bigr)$ gives
$$\exp\bigl(-\Phi(v)\bigr)\mu(du)Q(u,dv)=\exp\bigl(-\Phi(u)\bigr)\mu(dv)Q(v,du).$$
This is the statement that
$$\exp\bigl(-\Phi(v)\bigr)\nu(du,dv)=\exp\bigl(-\Phi(u)\bigr)\ntu(du,dv).$$
Since $\Phi$ is bounded on bounded sets by Assumption \ref{ass:bdd}
we deduce that
$$\frac{d\nu}{d\ntu}(u,v)=r(u,v)=\exp\bigl(\Phi(v)-\Phi(u)\bigr).$$
Theorem \ref{t:mh1} gives the desired result.
\end{proof}

We provide two examples of prior reversible proposals,
the first applying in the general Banach space setting,
and the second when the prior is a Gaussian measure.

\begin{algorithm} \label{alg:I}
{\bf Independence Sampler}
The independence sampler arises when $Q(u,dv)=\mu_0(dv)$ so that proposals
are independent draws from the \as{prior}\index{prior}. Clearly prior reversibility is
satisfied. The following algorithm results:
Define
$$a(u,v)=\min\{1,\exp\bigl(\Phi(u)-\Phi(v)\bigr)\}$$
and generate $\{u^{(0)}\}_{k \ge 0}$ as follows:

\begin{enumerate}

\item Set $k=0$ and pick $u^{(0)} \in X$.

\item Propose $v^{(k)}\sim \mu_0$ independently of $u^{(k)}$.

\item Set $u^{(k+1)}=v^{(k)}$ with probability $a(u^{(k)},v^{(k)})$, independently of $(u^{(k)},v^{(k)})$.

\item Set $u^{(k+1)}=u^{(k)}$ otherwise.

\item $k \to k+1$ and return to 2.

\end{enumerate}
$\quad\Box$
\end{algorithm}

The preceding algorithm works well when the \as{likelihood}\index{likelihood} 
is not {\em too} informative;
however when the information in the \as{likelihood}\index{likelihood} 
is substantial, and $\Phi(\cdot)$
varies significantly depending on where it is evaluated, the 
independence sampler will
not work well. In such a situation it is typically the case that {\em local}
proposals are needed, with a parameter controlling the degree of locality;
this parameter can then be optimized by choosing it as large as possible,
consistent with achieving a reasonable acceptance probability.
The following algorithm is an example of this concept, with parameter
$\beta$ playing the role of the locality parameter.
The algorithm may be viewed as the natural generalization of
the random Walk Metropolis method, for targets defined by
density with respect to Lebesgue measure, to the situation where
the targets are defined by density with respect to Gaussian measure.
The name pCN is used because of the original derivation of
the algorithm via a Crank-Nicolson discretization of the Hilbert
space valued SDE \eqref{a44}.

\begin{algorithm} \label{alg:pcn}
{\bf pCN Method}
Assume that $X$ is a Hilbert space $\bigl(\cH,\cBc(\cH)\bigr)$ 
and that $\mu_0=N(0,C)$ is a Gaussian \as{prior}\index{prior!Gaussian}
on $\cH$. Now define $Q(u,dv)$ to be the Gaussian 
measure $N\bigl((1-\beta^2)^{\frac12}u,\beta^2C\bigr)$,
also on $\cH$. 
Example \ref{ex:rev1} shows that $Q$ is $\mu_0$ reversible.
The following algorithm results:

Define
$$a(u,v)=\min\{1,\exp\bigl(\Phi(u)-\Phi(v)\bigr)\}$$
and generate $\{u^{(0)}\}_{k \ge 0}$ as follows:

\begin{enumerate}

\item Set $k=0$ and pick $u^{(0)} \in X$.

\item Propose $v^{(k)}=
\sqrt{(1-\beta^2)}u^{(k)}+\beta  \xi^{(k)},
\quad \xi^{(k)} \sim N(0,\cC)$.

\item Set $u^{(k+1)}=v^{(k)}$ with probability $a(u^{(k)},v^{(k)})$, independently of $(u^{(k)},\xi^{(k)})$.

\item Set $u^{(k+1)}=u^{(k)}$ otherwise.

\item $k \to k+1$ and return to 2.

\end{enumerate}
$\quad\Box$
\end{algorithm}

\begin{example}
\label{ex:pCN}
Example \ref{ex:MHs} shows that using the proposal from
Example \ref{ex:rev1} within a Metropolis-Hastings context 
may be viewed as using a proposal based on the 
$\mu$-measure-preserving equation (\ref{44}), but with the $D\Phi$ term dropped. 
The accept-reject mechanism of Algorithm 5.10, which is based on
differences of $\Phi$, then compensates for the missing $D\Phi$ term.

\end{example}

\vspace{0.2in}
\subsection{Sequential Monte Carlo Methods}
\label{ssec:pf}

In this section we introduce sequential Monte Carlo
methods and show how these may be viewed as a generic
tool for sampling the \as{posterior}\index{posterior} distribution arising
in Bayesian inverse problems. These methods have their
origin in filtering of dynamical systems but, as we
will demonstrate, have the potential as algorithms
for probing a very wide class of probability measures.
The key idea is to introduce
a sequence of measures which evolve the \as{prior}\index{prior}
distribution into the posterior distribution. Particle filtering
methods are then applied to this sequence of measures
in order to evolve a set of particles that are prior
distributed into a set of particles that are approximately
posterior distributed. From a practical perspective,
a key step in the construction of these methods is
the use of MCMC methods which preserve the measure of
interest, and other measures closely related to it;
furthermore, our interest is in designing SMC methods which, 
in principle, are well-defined on the infinite dimensional 
space; for these two reasons the MCMC methods from the previous
subsection play a central role in what follows. 

Given integer $J$, let $h=J^{-1}$ and for non-negative integer $j\le J$
define the sequence of measures $\mu_j\ll\mu_0$ by
\begin{subequations}\label{145}
\begin{eqnarray}
\frac{d\mu_j}{d\mu_0}(u)&=&\frac{1}{Z_j}\exp\big(-jh\Phi(u)\big),\label{145a}\\
Z_j&=&\int_{\mathcal{H}}\exp\big(-jh\Phi(u)\big)\mu_0(du) \label{145b}.
\end{eqnarray}
\end{subequations}
Then $\mu_J=\mu$ given by \eqref{045}; thus our interest
is in approximating $\mu_J$ and we will achieve this by approximating
the sequence of measures $\{\mu_j\}_{j=0}^J$, using information
about $\mu_j$ to inform approximation of $\mu_{j+1}.$ To simplify
the analysis we assume that
$\Phi$ is bounded above and below on $X$ so that there is
$\phi^{\pm} \in \RR$ such that
\begin{equation}
\label{pb}
\phi^- \le \Phi(u) \le \phi^+\quad \forall u \in X.
\end{equation}
Without loss of generality we assume that $\phi^-\le 0$ and that $\phi^+\ge 0,$
which may be achieved by normalization.
Note that then the family of measures $\{\mu_j\}_{j=0}^{J}$
are mutually absolutely continuous and, furthermore,
\begin{equation}
\label{245}
\frac{d\mu_{j+1}}{d\mu_{j}}(u)=\frac{Z_j}{Z_{j+1}}\exp\big(-h\Phi(u)\big).
\end{equation}
An important idea here is that, whilst $\mu_0$ and $\mu$ may be
quite far apart as measures, the pair of measures 
$\mu_j, \mu_{j+1}$ can be quite close, for sufficiently
small $h$. This fact can be used to incrementally evolve
samples from $\mu_0$ into approximate samples of $\mu_J$.

Let $\LC$ denote the operator on probability measures which 
corresponds to application of Bayes' theorem with 
\as{likelihood}\index{likelihood}
proportional to $\exp\big(-h\Phi(u)\big)$ and let $P_j$
denote any Markov kernel which preserves the measure $\mu_j$;
such kernels arise, for example, from the MCMC methods of
the previous subsection.  These considerations imply that
\begin{equation}
\label{eq:prop}
\mu_{j+1}=\LC P_j\mu_j.
\end{equation}
Sequential Monte Carlo methods proceed by approximating
the sequence $\{\mu_j\}$ by a set of Dirac measures, as we now 
describe. It is useful to break up the iteration \eqref{eq:prop}
and write it as 
\begin{subequations}
\label{eq:2prop}
\begin{align}
\label{eq:2propa}
\hmu_{j+1}&=P_j\mu_j,\\
\mu_{j+1}&=\LC \hmu_{j+1}.
\label{eq:2propb}
\end{align}
\end{subequations}
We approximate each of the two steps in \eqref{eq:2prop}
separately. To this end it helps to note that, since
$P_j$ preserves $\mu_j$,
\begin{equation}
\label{345}
\frac{d\mu_{j+1}}{d\hmu_{j+1}}(u)=\frac{Z_j}{Z_{j+1}}\exp\big(-h\Phi(u)\big).
\end{equation}

To define the method, we write
an $N$-particle Dirac measure approximation of the form 
\begin{equation}
\mu_j \approx 
\mu_j^{N}:= 
\sum_{n=1}^{N} w_j^{(n)} \delta(v_j - v_j^{(n)}).
\label{eq:emp_filt}
\end{equation}
The approximate distribution is completely
defined by particle positions $v_j^{(n)}$ and weights $w_j^{(n)}$
respectively. Thus the objective of the method is to
find an update rule for $\{ v_j^{(n)}, w_j^{(n)}\}_{n=1}^N \mapsto
\{ v_{j+1}^{(n)}, w_{j+1}^{(n)}\}_{n=1}^{N}$.
The weights must sum to one. To do this we proceed as follows.
First each particle $v_j^{(n)}$ is updated 
by proposing a new candidate particle $\hv_{j+1}^{(n)}$ according
to the Markov kernel $P_j$; this corresponds to \eqref{eq:2propa}
and creates an approximation to
$\hmu_{j+1}.$ (See the last two parts of Remark \ref{rem:fil} for 
a discussion on the role of $P_j$ in the algorithm.) 
We can think of this approximation
as a \as{prior}\index{prior} distribution for
application of Bayes' rule in the form \eqref{eq:2propb},
or equivalently \eqref{345}.  Secondly, each new 
particle is re-weighted according to the desired
distribution $\mu_{j+1}$ given by \eqref{345}. 
The required calculations are very straightforward 
because of the assumed form of the measures as 
sums of Dirac's, as we now explain.

The first step of the algorithm has made the approximation
\begin{equation}
\hmu_{j+1} \approx \hmu_{j+1}^N=\sum_{n=1}^{N} w_j^{(n)} \delta(v_{j+1} - \hv_{j+1}^{(n)}).
\label{eq:emp_filt_ratz}
\end{equation}
We now apply Bayes' formula in
the form \eqref{345}. 
Using an approximation proportional to
(\ref{eq:emp_filt_ratz}) for $\hmu_{j+1}$
we obtain 
\begin{equation}
\mu_{j+1} \approx \mu_{j+1}^N:=\sum_{n=1}^{N} w_{j+1}^{(n)} \delta(v_{j+1} - \hv_{j+1}^{(n)}).
\label{eq:emp_filt2}
\end{equation}
where
\begin{equation}
\label{eq:waits}
\hw_{j+1}^{(n)} = \exp\bigl(-h\Phi(\hv_{j+1}^{(n)})\bigr)w_{j}^{(n)}
\end{equation}
and normalization requires
\begin{equation}
\label{eq:rewait}
w_{j+1}^{(n)}=\hw_{j+1}^{(n)}/\bigl(\sum_{n=1}^{N} \hw_{j+1}^{(n)}\bigr).
\end{equation}

Practical experience shows that some weights become very small 
and for this reason it is desirable to add a {\em resampling step}
to determine the $\{v_{j+1}^{(n)}\}$ by drawing from
(\ref{eq:emp_filt2}); this has the effect of removing particles
with very low weights and replacing them with multiple
copies of the particles with higher weights.
Because the initial measure $\rp(v_0)$ is not
in Dirac form it is convenient to place this resampling step at the
{\em start} of each iteration, rather than at the end as we have
presented here, as this naturally introduces a particle approximation
of the initial measure. This reordering makes no difference to the iteration 
we have described and results in the following algorithm.

\begin{algorithm}
\label{alg:smc}
\begin{enumerate}
\item[1.] Let $\mu_0^N=\mu_0$ and set $j=0$.
\item[2.] Draw $v_{j}^{(n)} \sim \mu_j^N$, $n=1,\dots,N$. 
\item[3.] Set $w_{j}^{(n)}=1/N$, $n=1,\dots,N$ and define
$\mu_j^N$ by \eqref{eq:emp_filt}.
\item[4.] Draw $\hv_{j+1}^{(n)} \sim P_j(v_{j}^{(n)},\cdot)$.
\item[5.] Define $w_{j+1}^{(n)}$ by \eqref{eq:waits}, \eqref{eq:rewait} and $\mu_{j+1}^N$ by \eqref{eq:emp_filt2}. 
\item[6.] $j+1\to j$ and return to 2.
\end{enumerate}
$\quad\Box$
\end{algorithm}

We define $S^N$ to be the mapping between probability
measures defined by sampling $N$ i.i.d.\ points
 from a measure and approximating that measure by an
equally weighted sum of Dirac's at the sample points. Then
the preceding algorithm may be written as
\begin{equation}
\label{eq:pfn}
\mu_{j+1}^N=\LC S^N P_j \mu_j^N.
\end{equation}
Although we have written the sampling step
$S^N$ {\em after} application of $P_j$, some reflection shows that
this is well-justified: applying $P_j$ followed by $S^N$ can be shown,
by first conditioning on the initial point and sampling with respect
to $P_j$, and then sampling over the distribution of the initial point,
to be the algorithm as defined.
The sequence of distributions that we wish to approximate 
simply satisfies the iteration
\eqref{eq:prop}.
Thus analyzing the particle filter requires estimation of the
error induced by application of $S^N$ (the {\em resampling error})
together with estimation of the rate of accumulation of this error
in time.

The operators $\LC, P_j$ and $S^N$ map the space $\cP(X)$ 
of probability measures on $X$ into itself according to the
following: 
\begin{subequations}
\begin{align*}
(\LC \mu) (dv) &= \frac{\exp\bigl(-h\Phi(v)\bigr) \mu(dv)}{\int_{X} \exp\bigl(-h\Phi(v)\bigr) \mu(dv)}, \\
(P_j \mu) (dv) &= \int_{X} P_j(v',dv) \mu(dv'), \\
(S^N \mu)(dv) &= \frac{1}{N} \sum_{n=1}^N \delta(v-v^{(n)}) dv, \quad v^{(n)} \sim \mu ~~~ {\rm i.i.d.}. 
\end{align*}
\end{subequations}
where $P_j$ is the kernel associated with the $\mu_j$-invariant
Markov chain.

Let $\mu=\mu^{(\omega)}$ denote, for each $\omega$, an element of $\cP(X)$.
If we assume that $\omega$ is a random variable, and let $\bbE^{\omega}$
denote expectation over $\omega$, then we may define a distance $d(\cdot,\cdot)$
between two random probability measures $\mu^{(\omega)}, 
\nu^{(\omega)}$, as follows:
$$
d(\mu,\nu) = {\rm sup}_{|f|_\infty \leq 1} \sqrt{ \bbE^{\omega} |\mu(f)-\nu(f)|^2 },
$$
with $|f|_\infty:=\sup_{v\in X}|f(v)|$, and 
where we have used the convention that $\mu(f) = \int_{X} f(v)\mu(dv)$ 
for measurable $f:X \to \R$, and similar for $\nu$. 
This distance does indeed generate a metric and, in particular,
satisfies the triangle inequality.
In fact it is simply the \as{total variation distance}\index{total variation distance} in the case of 
measures which are not random.

With respect to this distance between random probability measures
we may prove that the SMC method
generates a good approximation of the true measure $\mu$, in the
limit $N \to \infty$.
We use the fact that, under \eqref{pb}, we have
$$
\exp\bigl(-h\phi^+\bigr)< \exp\bigl(-h\Phi(v)\bigr) < \exp\bigl(-h\phi^-\bigr). 
$$ 
Since $\phi^- \le 0$ and $\phi^+ \ge 0$ we deduce that
there exists $\kappa \in (0,1)$ such that for all $v\in X$ 
$$
\kappa< \exp\bigl(-h\Phi(v)\bigr) < \kappa^{-1}.
$$ 
This constant $\kappa$ appears in the following.

\begin{theorem}\label{th39}
We assume in the following that \eqref{pb} holds.
Then
$$d(\mu^N_J,\mu_J) \le \sum_{j=1}^J (2\kappa^{-2})^j \frac{1}{\sqrt N}.$$
\end{theorem}

\begin{proof}
The desired result is a consequence of the following three 
facts, whose proof we postpone to three lemmas at the end of the subsection:
\begin{subequations}
\begin{align*}
 \sup_{\mu \in \cP(X)}d(S^N \mu,\mu) &\leq \frac{1}{\sqrt{N}}, \\
 d(P_j \nu , P_j \mu) &\leq   d(\nu,\mu), \\
 d(\LC \nu , \LC \mu) &\leq 2 \kappa^{-2}  d( \nu , \mu ). 
\end{align*}
\end{subequations}
By the triangle inequality we have, for $\nu_j^N=P\mu_j^N$,
\begin{align*}
d(\mu^N_{j+1},\mu_{j+1}) &= d(\LC S^N P_j \mu^N_j,\LC P_j \mu_j)\\
& \le d(\LC P_j \mu^N_j,\LC P_j \mu_j)+d(\LC S^N P_j \mu^N_j,\LC P_j \mu^N_j)\\
& \le 2\kappa^{-2}\Bigl(d(\mu^N_j,\mu_j)+d(S^N \nu_j^N,\nu^N_j)\Bigr)\\
& \le 2\kappa^{-2}\Bigl(d(\mu^N_j,\mu_j)+\frac{1}{\sqrt N}\Bigr).
\end{align*}
Iterating, after noting that $\mu_0^N=\mu_0$, gives the desired result.
\end{proof}

\begin{remarks}
\label{rem:fil}
This theorem shows that the sequential particle filter actually
reproduces the true \as{posterior}\index{posterior} distribution $\mu=\mu_J$, 
in the limit $N \to \infty$. 
We make some comments about this.

\begin{itemize}

\item
The measure $\mu=\mu_J$ is well-approximated by 
$\mu_j^N$ in the sense that, as the number of particles
$N \to \infty$, the approximating measure converges to the
true measure. The result holds in the infinite dimensional
setting. As a consequence the algorithm as stated is
robust to finite dimensional approximation.

\item Note that $\kappa=\kappa(J)$ and that $\kappa \to 1$ as
$J \to \infty.$ Using this fact shows that the error constant
in Theorem \ref{th39} behaves as 
$\sum_{j=1}^J (2\kappa^{-2})^j \asymp J\,2^{J}.$
Optimizing this upper bound does not give a useful rule-of-thumb
for choosing $J$, and in fact suggests choosing $J=1$.
In any case in applications $\Phi$ is not bounded
from above, or even  below in general, and a more refined analysis
is then required.

\item In principle the theory applies even if the Markov
kernel $P_j$ is simply the identity mapping on probability
measures. However, moving the particles according to a 
non-trivial $\mu_j$-invariant measure is absolutely essential for the methodology
to work in practice. This can be seen by noting that if $P_j$
is indeed taken to be the identity map on measures then the
particle positions will be unchanged as $j$ changes, meaning
that the measure $\mu=\mu_J$ is approximated by weighted
samples from the \as{prior}\index{prior}, clearly undesirable in general. 

\item In fact, if the Markov kernel $P_j$ is ergodic 
then it is sometimes possible to obtain bounds which are {\em uniform} in $J$.

\end{itemize}

\end{remarks}

We now prove the three lemmas which underly the convergence proof.

\begin{lemma}
The sampling operator satisfies
$$\sup_{\mu \in \cP(X)}d(S^N \mu,\mu) \leq \frac{1}{\sqrt{N}}.$$
\end{lemma}

\begin{proof} 
Let $\nu$ be an element of $\cP(X)$ and $\{v^{(k)}\}_{k=1}^N$
a set of  i.i.d.\ samples with $v^{(1)} \sim \nu$; the randomness
entering the probability measures is through these samples, expectation
with respect to which we denote by $\bbE^{\omega}$ in what follows. Then
$$S^N \nu(f)=\frac{1}{N}\sum_{k=1}^N f(v^{(k)})$$
and, defining $\bF=f-\nu(f)$, we deduce that
$$S^N\nu(f)-\nu(f)=\frac{1}{N}\sum_{k=1}^N \bF(v^{(k)}).$$
It is straightforward to see that
$$\bbE^{\omega} \bF(v^{(k)})\bF(v^{(l)})=\delta_{kl}\bbE^{\omega} |\bF(v^{(k)})|^2.$$
Furthermore, for $|f|_\infty \le 1$,
$$\bbE^{\omega} |\bF(v^{(1)})|^2=\bbE^{\omega} |f(v^{(1)})|^2-|\bbE^{\omega} f(v^{(1)})|^2 \le 1.$$
It follows that, for $|f|_\infty \le 1$, 
$$\bbE^{\omega}|\nu(f)- S^N\nu(f)|^2=\frac{1}{N^2}\sum_{k=1}^N \bbE^{\omega} |\bF(v^{(k)})|^2 \le \frac{1}{N}.$$
Since the result is independent of $\nu$ we may take the supremum over all
probability measures and obtain the desired result.
\end{proof}

\begin{lemma}
Since $P_j$ is a Markov kernel we have
$$d(P_j \nu , P_j \nu') \leq   d(\nu,\nu').$$
\end{lemma}

\begin{proof}
The result is generic for any Markov kernel $P$, so we drop the
index $j$ on $P_j$ for the duration of the proof.
Define
$$q(v')=\int_{X} P(v',dv)f(v),$$
that is the expected value of $f$ under one-step of the Markov 
chain started from $v'$. Clearly,
since
$$|q(v')| \le \Bigl(\int_X P(v',dv)\Bigr)\sup_{v}|f(v)|=\sup_{v}|f(v)|$$
it follows that
$$\sup_{v} |q(v)| \le \sup_{v}|f(v)|.$$
Also, since
$$P\nu(f)=\int_X f(v) \Bigl(\int_X P(v',dv)\nu(dv')\Bigr),$$
exchanging the order of integration shows that
$$|P\nu(f)-P\nu'(f)|=|\nu(q)-\nu'(q)|.$$
Thus 
\begin{align*}
d(P\nu,P\nu')&=\sup_{|f|_\infty\le 1}\Bigl(\bbE^{\omega}|P\nu(f)-P\nu'(f)|^2\Bigr)^{\frac12}\\
& \le \sup_{|q|_\infty\le 1}\Bigl(\bbE^{\omega}|\nu(q)-\nu'(q)|^2\Bigr)^{\frac12}\\
& = d(\nu,\nu')
\end{align*}
as required.
\end{proof}

\begin{lemma}
Under the Assumptions of Theorem \ref{th39} we have
$$d(\LC \nu , \LC \mu) \leq 2 \kappa^{-2}  d( \nu , \mu ).$$ 
\end{lemma}

\begin{proof}
Define $g(v)=\exp\bigl(-h\Phi(v)\bigr)$.
Notice that for $|f|_\infty <\infty$  
we can rewrite 
\begin{subequations}
\begin{align*}
(\LC \nu)(f) - (\LC \mu)(f)  = & \frac{\nu(f g)}{\nu(g)} - \frac{\mu(f g)}{\mu(g)} \\
 = & \frac{\nu(f g)}{\nu(g)}  - \frac{\mu(f g)}{\nu(g)} + \frac{\mu(f g)}{\nu(g)} - \frac{\mu(f g)}{\mu(g)} \\
 = & \frac{\kappa^{-1}}{\nu(g)}[\nu(\kappa f g) -\mu(\kappa f g)] + 
 \frac{\mu(f g)}{\mu(g)} \frac{\kappa^{-1}}{\nu(g)} [\mu(\kappa g) - \nu(\kappa g)].
\end{align*}
\end{subequations}
Now notice that 
$\nu(g)^{-1} \leq \kappa^{-1}$ and that, for $|f|_{\infty} \le 1$,
$\mu(fg)/\mu(g) \le 1$ since the expression corresponds to
an expectation with respect to measure found from $\mu$ by reweighting
with \as{likelihood}\index{likelihood} proportional to $g$. Thus 
$$|(\LC \nu)(f) - (\LC \mu)(f)| \le \kappa^{-2}|\nu(\kappa f g) -\mu(\kappa f g)|+\kappa^{-2}|\nu(\kappa g) - \mu(\kappa g)|.$$ 
Since $|\kappa g| \le 1$ it follows that
$$\bbE^{\omega}|(\LC \nu)(f) - (\LC \mu)(f)|^2 \le 4\kappa^{-4}
\sup_{|f|_\infty \le 1} \bbE^{\omega}|\nu(f) - \mu(f)|^2$$
and the desired result follows.
\end{proof}

\vspace{0.2in}
\subsection{Continuous Time Markov Processes}
\label{ssec:start}

In the remainder of this section we shift our attention to 
continuous time processes which
preserve $\mu$; these are important in the construction of
proposals for MCMC methods, and also as diffusion limits for
MCMC. Our main goal is to show that the equation \eqref{44} preserves $\mu$.
Our setting is to work in the separable
Hilbert space $\mathcal {H}$ with 
inner-product and norm denoted by
$\langle \cdot, \cdot \rangle$ and $\|\cdot\|$ respectively.
We assume that the \as{prior}\index{prior!Gaussian}
$\mu_0$ is a Gaussian on $\mathcal {H}$ and,
furthermore, we specify the space $X \subset \mathcal {H}$ 
that will play a central role in this continuous time setting. 
This choice of space $X$ will link the properties
of the reference measure $\mu_0$ and the potential $\Phi$.
We assume that $\mathcal{C}$ has eigendecomposition
\begin{equation}\label{047}
\mathcal{C}\phi_j=\gamma_j^2\phi_j
\end{equation}
where $\{\phi_j\}_{j=1}^\infty$ forms an orthonormal basis for
$\mathcal{H}$, and where $\gamma_j\asymp j^{-s}$. Necessarily
$s>\frac{1}{2}$ since $\mathcal{C}$ must be trace-class 
to be a covariance on $\mathcal{H}$.
We define the following scale of Hilbert subspaces, defined 
for $r>0$, by
$$\mathcal{X}^r=\Bigl\{u\in\mathcal{H}\big|\sum_{j=1}^{\infty} j^{2r}|\langle u,\phi_j\rangle|^2<\infty\Bigr\}$$
and then extend to superspaces $r<0$ by duality. We use 
$\|\cdot\|_r$ to denote the norm induced by the inner-product 
$$\langle u,v \rangle_r=\sum_{j=1}^{\infty} j^{2r} u_j v_j$$
for $u_j=\langle u,\phi_j \rangle$ and $v_j=\langle v,\phi_j \rangle$.
Application of Theorem \ref{t:2.2} with $d=1$ and $q=2$ shows that
$\mu_0(\mathcal{X}^r)=1$ for all $r \in [0,s-\frac12)$.
In what follows we will take $X=\mathcal{X}^t$ for some fixed
$t \in [0,s-\frac12)$.

Notice that we have not assumed that the underlying Hilbert
space is comprised of $L^2$ functions mapping $D \subset
\R^d$ into $\R$, and hence we have not introduced the dimension $d$ of
an underlying physical space $\R^d$ into either the decay
assumptions on the $\gamma_j$ or the spaces $\mathcal{X}^r$.
However, note that the spaces $\mathcal{H}^t$ introduced in
subsection \ref{ssec:2.4} are, in the case where $\mathcal{H}=L^2(D;\R)$,
the same as the spaces $\mathcal{X}^{t/d}$.

We now break our developments into introductory discussion of 
the finite dimensional setting, in subsection \ref{ssec:FDL},
and into the Hilbert space setting in subsection \ref{ssec:IDL}. 
In subsection \ref{sssec:1} we introduce a family of Langevin
equations which are invariant with respect to a given measure
with smooth Lebesgue density. Using this, in subsection \ref{sssec:2},
we motivate equation \eqref{44} showing that, in finite dimensions,
it corresponds to a particular choice of Langevin equation. 
In subsection \ref{sssec:3}, for the infinite-dimensional setting, we describe the precise assumptions
under which we will prove invariance of measure $\mu$ under the
dynamics \eqref{44}. Subsection \ref{sssec:4} describes the
elements of the finite dimensional approximation of \eqref{44}
which will underly our proof of invariance. Finally, subsection
\ref{sssec:5} contains statement of the measure invariance
result as Theorem \ref{t:4.14}, together with its proof; this
is preceded by Theorem \ref{t:4.13} which establishes existence
and uniqueness of a solution to \eqref{44}, as well as
continuous dependence of the solution on the initial condition
and Brownian forcing. Theorems \ref{t:4.10} and \ref{t:4.9} are
the finite dimensional analogues of Theorems \ref{t:4.14} and \ref{t:4.13}
respectively and play a useful role in motivating the
infinite dimensional theory.

\vspace{0.2in}
\subsection{Finite Dimensional Langevin Equation} 
\label{ssec:FDL}
 
\subsubsection{Background Theory}
\label{sssec:1}

Before setting up the (rather involved) technical assumptions
required for our proof of measure invariance, we give some
finite-dimensional intuition.
Recall that $|\cdot|$ denotes the Euclidean
norm on $\R^n$ and we also use this notation for
the induced matrix norm on
$\R^n$. We assume that
$$I\in \CC^2(\R^n,\R^+), \,\,\,\,
\int_{\R^n} e^{-I(u)}du=1.$$ Thus $\rho(u)=e^{-I(u)}$ is the
Lebesgue density corresponding to a
random variable on $\R^n$. Let $\mu$ be 
the corresponding measure.

Let $\mathbb{W}$ denote standard \as{Wiener measure}\index{Wiener!measure} 
on $\R^n$.
Thus $B\sim \mathbb{W}$ is a standard Brownian motion in $\CC([0,\infty);
\R^n)$.
Let $u\in \CC([0,\infty); \R^n)$ satisfy the SDE
\begin{equation}\label{048}
\frac{du}{dt}=-A\,DI(u)+\sqrt{2A}\frac{dB}{dt}, \,\, u(0)=u_0
\end{equation}
where $A\in\R^{n\times n}$ is symmetric and strictly positive definite
and $DI \in \CC^1(\R^n,\R^n)$ is the gradient of $I$.
Assume that $\exists M>0: \forall u\in\R^n$, the Hessian of $I$
satisfies
$$|D^2I(u)|\leq M.$$
We refer to equations of the form \eqref{048} as {\em Langevin
equations} (as mentioned earlier they correspond to
overdamped Langevin equations in the physics literature, and to Langevin
equations in the statistics literature), 
and the matrix $A$ as a {\em preconditioner}.

\begin{theorem}\label{t:4.9}
For every $u_0 \in \R^n$ and $\mathbb{W}$-a.s., equation
(\ref{048}) has a unique global in time solution $u\in \CC([0,\infty);
\R^n)$.
\end{theorem}

\begin{proof} A solution of the SDE
is a solution of the integral equation
\begin{equation}\label{049}
u(t)=u_0-\int_0^tA\,DI\big(u(s)\big)ds+\sqrt{2A}B(t).
\end{equation}
Define $X=\CC([0,T];\R^n)$ and $\mathcal{F}:X\rightarrow X$ by
\begin{equation}\label{0410}
(\mathcal{F}v)(t)=u_0-\int_0^t A\,DI\big(v(s)\big)ds+\sqrt{2A}B(t).
\end{equation}
Thus $u\in X$ solving (\ref{049}) is a fixed point of $\mathcal{F}$.
We show that $\mathcal{F}$ has a unique fixed point, for $T$
sufficiently small. To this end we study a contraction property
of $\mathcal{F}$:
\begin{eqnarray*}
\|(\mathcal{F}v_1)-(\mathcal{F}v_2)\|_X&=&\sup_{0\leq t\leq
T}\Big|\int_0^t\Big(A\,DI\big(v_1(s)\big)-A\,DI\big(v_2(s)\big)\Big)ds\Big|\\
&\leq&\int_0^T\Big|A\,DI\big(v_1(s)\big)-A\,DI\big(v_2(s)\big)\Big|ds\\
&\leq&\int_0^T|A|M|v_1(s)-v_2(s)|ds\\
&\leq&T|A|M\|v_1-v_2\|_X.
\end{eqnarray*}
Choosing $T: T|A|M<1$ shows that $\mathcal{F}$ is a contraction on
$X$. This argument may be repeated on successive intervals $[T,2T],
[2T,3T],\ldots$ to obtain a unique global solution in
$\CC([0,\infty);\R^n)$. 
\end{proof}

\begin{remark}
Note that, since $A$ is positive-definite symmetric, its eigenvectors $e_j$
form an orthonormal basis for $\R^n$. We write $Ae_j=\alpha_j^2 e_j$. Thus
$$B(t)=\sum_{j=1}^n \beta_j(t) e_j$$
where the $\{\beta_j\}_{j=1}^n$ are an i.i.d.\ collection of standard unit
Brownian motions on $\R$. Thus we obtain
$$\sqrt{A} B(t)=\sum_{j=1}^n \alpha_j \beta_j e_j=:W(t).$$
We refer to $W$ as an $A$-\as{Wiener process}\index{Wiener!process}. 
Such a process is Gaussian with
mean zero and covariance structure
$$\EE W(t) \otimes W(s)= A (t \wedge s).$$
The equation \eqref{048} may be written as 
\begin{equation}\label{148}
\frac{du}{dt}=-ADI(u)+\sqrt{2}\frac{dW}{dt}, \,\, u(0)=u_0.
\end{equation}
\label{rem:cwp}
\end{remark}

\begin{theorem}\label{t:4.10}
Let $u(t)$ solve \eqref{048}.
If $u_0\sim\mu$ then $u(t)\sim\mu$ for all $t>0$. More precisely,
for all $\varphi:\R^n\rightarrow\R^+$ bounded and
continuous, $u_0\sim\mu$ implies
$$\bbE\varphi\big(u(t)\big)=\bbE\varphi(u_0),\,\,\forall t>0.$$
\end{theorem}

\begin{proof}
Consider the additive noise SDE, for
additive noise with strictly positive-definite diffusion
matrix $\Sigma$, 
$$\frac{du}{dt}=f(u)+\sqrt{2\Sigma}\frac{dB}{dt},\,\,u(0)=u_0\sim\nu_0.$$
If $\nu_0$ has pdf $\rho_0$, then the Fokker-Planck equation for
this SDE is
\begin{eqnarray*}
\frac{\partial\brho}{\partial
t}&=&\nabla\cdot(-f\brho+\Sigma\nabla\brho), \,\, (u,t)\in
\R^n\times\R^+,\\
\brho|_{t=0}&=&\rho_0.
\end{eqnarray*}
At time $t>0$ the solution of the 
SDE is distributed according to measure $\nu(t)$
with density $\brho(u,t)$ solving the Fokker-Planck equation.
Thus the initial measure $\nu_0$ is preserved if
$$\nabla\cdot(-f\rho_0+\Sigma\nabla\rho_0)=0$$ and then $\brho(\cdot,t)=\rho_0,\,\, \forall t\geq
0$.

We apply this Fokker-Planck equation to show that $\mu$ is invariant
for equation (\ref{049}). We need to show that
$$\nabla \cdot \big(ADI(u)\rho+A\,\nabla\rho\big)=0$$ if $\rho=e^{-I(u)}$.
With this choice of $\rho$ we have
$$\nabla\rho=-DI(u)e^{-I(u)}=-DI(u)\rho.$$ Thus
$$A\,DI(u)\rho+A\,\nabla\rho=A\,DI(u)\rho-A\,DI(u)\rho=0,$$
so that
$$\nabla\cdot\big(A\,DI(u)\rho+A\,\nabla\rho\big)=\nabla\cdot(0)=0.$$
Hence the proof is complete. \end{proof}

\subsubsection{Motivation for Equation (\ref{44})}
\label{sssec:2}

Using the preceding finite dimensional development, we now motivate
the form of equation (\ref{44}). 
For (\ref{045}) we have, if $\mathcal{H}$ is
$\R^n$,
\begin{equation*}
\mu(du) = \exp\big(-I(u)\big)\,du\;,\qquad 
I(u) = \frac{1}{2}|\mathcal{C}^{-\frac{1}{2}}u|^2+\Phi(u)+\ln Z\;.
\end{equation*}
Thus $$DI(u)=\mathcal{C}^{-1}u+D\Phi(u)$$ and equation (\ref{048}),
which preserves $\mu$, is
$$\frac{du}{dt}=-A\big(\mathcal{C}^{-1}u+D\Phi(u)\big)+\sqrt{2A}\frac{dB}{dt}.$$
Choosing the preconditioner $A=\mathcal{C}$ gives
$$\frac{du}{dt}=-u-\mathcal{C}D\Phi(u)+\sqrt{2\mathcal{C}}\frac{dB}{dt}.$$
This is exactly (\ref{44}) provided $W=\sqrt{\mathcal{C}}B$, where
$B$ is a Brownian motion with covariance $\mathcal{I}$. Then $W$ is
a Brownian motion with covariance $\mathcal{C}$.
This is the finite dimensional analogue of the construction of a $\cC$-Wiener
process in the Appendix. 
We are now in a position to prove Theorems \ref{t:4.13} and \ref{t:4.14}
which are the infinite dimensional analogues of Theorems \ref{t:4.9}
and \ref{t:4.10}.

\vspace{0.2in}
\subsection{Infinite Dimensional Langevin Equation}
\label{ssec:IDL}

\subsubsection{Assumptions on Change of Measure}
\label{sssec:3}

Recall that $\mu_0(\mathcal{X}^r)=1$ for all $r \in [0,s-\frac12)$.
The functional $\Phi(\cdot)$ is assumed to be
defined on $\mathcal{X}^t$ for some $t \in [0,s-\frac12)$, and
indeed we will assume appropriate bounds on the first
and second derivatives, building on this assumption.
(Thus, in this subsection
\ref{sssec:3}, $t$ does not denote time; instead we use $\tau$ to denote the generic time-argument.)
These regularity assumptions on
$\Phi(\cdot)$ ensure that the probability distribution $\mu$ is
not too different from $\mu_0$, when projected into directions
associated with $\phi_j$ for $j$ large. 

For each $u \in \mathcal{X}^t$ the derivative $\nnabla \Phi(u)$ is an element of the
dual $(\mathcal{X}^t)^*$ of $\mathcal{X}^t$ comprising continuous linear
functionals on $\mathcal{X}^t$. However, we may identify
$(\mathcal{X}^t)^*$ with $\mathcal{X}^{-t}$ and view $\nnabla
\Phi(u)$ as an element of $\mathcal{X}^{-t}$ for each $u \in
\mathcal{X}^t$. With this identification, the following identity
holds
\begin{equation*}
\| \nnabla \Phi(u)\|_{\mathcal{L}(\mathcal{X}^t,\R)} = \|
\nnabla \Phi(u) \|_{-t}
\end{equation*}
and the second derivative $D^2 \Phi(u)$ can be identified as
an element of $\mathcal{L}(\mathcal{X}^t, \mathcal{X}^{-t})$. To
avoid technicalities we assume that $\Phi(\cdot)$ is quadratically
bounded, with first derivative linearly bounded and second
derivative globally bounded. Weaker assumptions could be dealt with
by use of stopping time arguments.

\begin{assumptions}
\label{a:4.8}
There exist constants $M_i \in \R^+, i \leq 4$ and $t \in [0,
s - 1/2)$ such that, for all $u \in \mathcal{X}^t$, the functional
$\Phi:\mathcal{X}^t \rightarrow \R$ satisfies
\begin{eqnarray*}
-M_1 \leq \Phi(u) &\leq&  M_2 \, \Big(1 +  \|u\|_t^2\Big);     \\
\| \nnabla \Phi(u)\|_{-t} &\leq& M_3 \, \Big(1 + \|u\|_t\Big); \\
\|D^2 \Phi(u)\|_{\mathcal{L}(\mathcal{X}^t,\mathcal{X}^{-t})}
&\leq& M_4.
\end{eqnarray*}
$\quad\Box$
\end{assumptions}

\begin{example}
The functional $\Phi(u)  =
\frac{1}{2}\|u\|_t^2$ satisfies Assumptions \ref{a:4.8}. To see this note
that we may write $\Phi(u)=\frac12 \langle u, \cK u \rangle$ where
$$\cK=\frac12 \sum_{j=1}^{\infty} j^{2t} \phi_j \phi_j^*.$$
The functional $\Phi: \mathcal{X}^t \to \R^+$ is clearly
well-defined by definition. Its derivative at $u \in \mathcal{X}^t$ is given
by $\cK u=\nnabla \Phi(u) = \sum_{{j \geq 1}} j^{2t} u_j \phi_j$,
where $u_j=\langle \phi_j, u \rangle.$ Furthermore $\nnabla \Phi(u) \in
\mathcal{X}^{-t}$ with $\|\nnabla \Phi(u)\|_{-t} = \|u\|_t$. The
second derivative $D^2 \Phi(u) \in \mathcal{L}(\mathcal{X}^t,
\mathcal{X}^{-t})$ is the linear operator $\cK$, that is the operator
that maps $u \in
\mathcal{X}^t$ to $\sum_{j \geq 1} j^{2t} \langle u,\phi_j\rangle
\phi_j \in \mathcal{X}^t$: its norm satisfies $\|
D^2 \Phi(u) \|_{\mathcal{L}(\mathcal{X}^t, \mathcal{X}^{-t})} = 1$ for any $u
\in \mathcal{X}^t$. \qed
\end{example}

Since the eigenvalues $\gamma_j^2$ of $\mathcal{C}$ decrease as
$\gamma_j\asymp j^{-s}$, the operator $\mathcal{C}$ has a
smoothing effect: $\mathcal{C}^{\alpha} h$ gains $2 \alpha s$ orders
of regularity in the sense that the $\mathcal{X}^{\beta}$-norm of
$\mathcal{C}^{\alpha}h$ is controlled by the $\mathcal{X}^{\beta-2
\alpha s}$-norm of $h \in \mathcal{H}$. Indeed
it is straightforward to show the following:

\begin{lemma}
\label{l:4.10} Under Assumptions \ref{a:4.8}, the following
estimates hold:
\begin{enumerate}
\item The operator $\mathcal{C}$ satisfies
\begin{equation*}
\|\mathcal{C}^{\alpha}h \|_{\beta} \asymp \|h\|_{\beta - 2 \alpha s}.
\end{equation*}
\item The function $\mathcal{C} \nnabla \Phi: \h^t \to \h^t$ is globally Lipschitz on $\h^t$:
there exists a constant $M_5 > 0$ such that
\begin{equation}
\| \mathcal{C}\nnabla \Phi(u) - \mathcal{C}\nnabla \Phi(v) \|_t \leq M_5 \, \| u-v \|_t
\qquad \qquad \forall u,v \in \h^t.
\end{equation}
\item The function $F: \h^t \to \h^t$ defined by
\begin{equation}
\label{eqn:mu}
F(u)=-u-\cC D\Phi(u)
\end{equation}
is globally Lipschitz on $\h^t$.
\item
The functional $\Phi(\cdot): \h^t \to \RR$ satisfies a second order
Taylor formula (for which we extend $\langle \cdot,\cdot \rangle$ from an
inner-product on $\h$ to the dual pairing between
$\h^{-t}$ and $\h^t$.) There exists a constant $M_6 > 0$ such that
\begin{equation} \label{eqn:2nd-Taylor}
\Phi(v) - \Big( \Phi(u) + \langle \nnabla \Phi(u), v-u \rangle\Big)
\leq M_6 \, \|u - v \|_t^2
\qquad \forall u,v \in \h^t.
\end{equation}
\end{enumerate}
\end{lemma}

\subsubsection{Finite Dimensional Approximation}
\label{sssec:4}

Our analysis now proceeds as follows. First we introduce an
approximation of the measure $\mu$, denoted by $\mu^N$.
To this end we let $P^N$ denote orthogonal projection in $\cH$
onto $X^N:={\rm span}\{\phi_1,\cdots,\phi_N\}$ and denote
by $Q^N$ orthogonal projection in $\cH$
onto $X^\perp:={\rm span}\{\phi_{N+1},\phi_{N+2},\cdots\}$.
Thus $Q^N=I-P^N$. Then define the measure $\mu^N$ by 
\begin{subequations}
\label{x43}
\begin{eqnarray}
\frac{d\mu^N}{d\mu_0}(u)&=&\frac{1}{Z^N}\exp\big(-\Phi(P^N u)\big),\label{x43a}\\
Z^N&=&\int_{X'}\exp\big(-\Phi(P^N u)\big)\mu_0(du) \label{x43b}.
\end{eqnarray}
\end{subequations}
This is a specific example of the approximating family
in \eqref{43} if we define
\begin{equation}
\label{eq:fn}
\Phi^N:=\Phi \circ P^N.
\end{equation}
Indeed if we take $X={\mathcal X}^\tau$ for any $\tau \in (t,s-\frac12)$
we see that $\|P^N\|_{{\mathcal L}(X,X)}=1$ and that, for any $u \in X$,
\begin{align*}
\|\Phi(u)-\Phi^N(u)\| & = \|\Phi(u)-\Phi(P^N u)\| \\
& \le M_3(1+\|u\|_t)\|(I-P^N)u\|_t\\
& \le CM_3(1+\|u\|_{\tau})\|u\|_{\tau}N^{-(\tau-t)}.
\end{align*}
Since $\Phi$, and hence $\Phi^N$, are bounded below by $-M_1$,
and since the function $1+\|u\|_{\tau}^2$ is integrable by the
Fernique Theorem \ref{t:2.5}, the approximation Theorem \ref{t:4.7} applies. 
We deduce that the \as{Hellinger distance}\index{Hellinger distance}
between $\mu$ and $\mu^N$ 
is bounded above by ${\cal O}(N^{-r})$ for any $r<s-\frac12-t$ since
$\tau-t \in (0,s-\frac12-t)$.

We will not use this explicit convergence rate in what follows,
but we will use the idea that $\mu^N$ converges to $\mu$ in order
to prove invariance of the measure $\mu$ under the SDE \eqref{44}. 
The measure $\mu^N$ has a product structure that we will exploit
in the following. We note that any element $u \in \cH$ is uniquely
decomposed as $u=p+q$ where $p \in X^N$ and $q \in X^{\perp}$.
Thus we will write $\mu^N(du)=\mu^N(dp,dq)$, and similar expressions
for $\mu_0$ and so forth, in what follows.

\begin{lemma} \label{l:prod}
Define $\cC^N=P^N\cC P^N$ and $\cC^{\perp}=Q^N\cC Q^N$. Then $\mu_0$
factors as the product of measures $\mu_{0,P}=N(0,\cC^N)$ and
$\mu_{0,Q}=N(0,\cC^{\perp})$ on $X^N$ and $X^\perp$ respectively.
Furthermore $\mu^N$ itself
also factors as a product measure on $X^N\oplus X^\perp$:
$\mu^N(dp,dq)=\mu_P(dp)\mu_Q(dq)$ with $\mu_Q=\mu_{0,Q}$
and 
$$\frac{d\mu_{P}}{d\mu_{0,P}}(u)\propto\exp\big(-\Phi(p)\big).$$
\end{lemma}

\begin{proof} 
Because $P^N$ and $Q^N$ commute with $\cC$, and because $P^N Q^N=Q^N P^N=0$,
the factorization of the reference measure $\mu_0$ follows automatically.
The factorization of the measure $\mu$ follows from  the fact
that $\Phi^N(u)=\Phi(p)$ and hence does not depend on $q$.
\end{proof}

To facilitate the proof of the desired measure
preservation property, we introduce the equation
\begin{equation}\label{44x}
\frac{du^N}{dt}=-u^N-\cC  P^N D \Phi^N(u^N)+\sqrt{2}\frac{dW}{dt}.
\end{equation}
By using well-known properties of finite dimensional SDEs,
we will show that, if $u^N(0) \sim \mu^N$, then $u^N(t)
\sim \mu^N$ for any $t>0$. By passing to the limit $N=\infty$
we will deduce that for \eqref{44}, if $u(0) \sim \mu$,
then $u(t) \sim \mu$ for any $t>0$.

The next lemma gathers various regularity estimates on the
functional $\Phi^N( \cdot)$ that are repeatedly
used in the sequel; they follow from the analogous properties
of $\Phi$ by using the structure $\Phi^N=\Phi\circ P^N$.

\begin{lemma}
\label{l:4.10b} Under Assumptions \ref{a:4.8}, the following
estimates hold with all constants uniform in $N$
\begin{enumerate}
\item
The estimates of Assumptions \ref{a:4.8} hold with $\Phi$ replaced
by $\Phi^N$.
\item The function $\mathcal{C} \nnabla \Phi^N: \h^t \to \h^t$ is globally Lipschitz on $\h^t$:
there exists a constant $M_5 > 0$ such that
\begin{equation*}
\| \mathcal{C}\nnabla \Phi^N(u) - \mathcal{C}\nnabla \Phi^N(v) \|_t \leq M_5 \, \| u-v \|_t
\qquad \qquad \forall u,v \in \h^t.
\end{equation*}
\item The function $F^N: \h^t \to \h^t$ defined by
\begin{equation}
\label{eqn:muN}
F^N(u)=-u-\cC P^N D\Phi^N(u)
\end{equation}
is globally Lipschitz on $\h^t$.
\item
The functional $\Phi^N(\cdot): \h^t \to \RR$ satisfies a second order
Taylor formula\ (for which we extend $\langle \cdot,\cdot \rangle$ from an
inner-product on $\h$ to the dual pairing between
$\h^{-t}$ and $\h^t$.) There exists a constant $M_6 > 0$ such that
\begin{equation} \label{eqn:2nd-TaylorN}
\Phi^N(v) - \Big( \Phi^N(u) + \langle \nnabla \Phi^N(u), v-u \rangle \Big)
\leq M_6 \, \|u - v \|_t^2
\qquad \forall u,v \in \h^t.
\end{equation}
\end{enumerate}
\end{lemma}

\subsubsection{Main Theorem and Proof}
\label{sssec:5}

Fix a function $W \in \CC([0,T];\h^t).$
Recalling $F$ defined by \eqref{eqn:mu}, we define a solution of 
\eqref{44} to be a function
$u \in \CC([0,T];\h^t)$ satisfying the integral equation
\begin{equation}
u(\tau) = u_0 + \int_0^\tau F\bigl(u(s)\bigr) \,ds + \sqrt{2}\,W(\tau) \label{eqn:IE}
\qquad \qquad \forall \tau \in [0,T].
\end{equation}
The solution is said to be {\em global} if $T>0$ is arbitrary.
For us, $W$ will be a $\mathcal{C}$-Wiener process and hence random; we
look for existence of a global solution, almost surely with respect to
the Wiener measure, Similarly a solution of \eqref{44x} is a function
$u^N \in \CC([0,T];\h^t)$ satisfying the integral equation
\begin{equation}
u^N(\tau) = u_0 + \int_0^\tau F^N\bigl(u^N(s)\bigr) \,ds + \sqrt{2}\,W(\tau) \label{eqn:IEN}
\qquad \qquad \forall t \in [0,T].
\end{equation}
Again, the solution is random because $W$ is a $\mathcal{C}$-Wiener process.
Note that the solution to this equation is not confined to $X^N$,
because both $u_0$ and $W$ have non-trivial components in $X^{\perp}.$
However within $X^{\perp}$ the behaviour is purely Gaussian and
within $X^N$ it is finite dimensional. We will exploit these two
facts.

The following establishes basic existence, uniqueness, continuity
and approximation properties of the solutions of \eqref{eqn:IE} and
\eqref{eqn:IEN}.

\begin{theorem}\label{t:4.13} For every $u_0 \in \h^t$ and for almost every
$\cC$-\as{Wiener\index{Wiener!process} process} $W$, equation \eqref{eqn:IE}
(respectively \eqref{eqn:IEN}) has a unique global solution. For any
pair $(u_0,W) \in \h^t \times \CC([0,T]; \h^t)$ we define the It{\^o}
map
\begin{equation*}
\Theta \colon \h^t \times \CC([0,T];\h^t) \rightarrow \CC([0,T];\h^t)
\end{equation*}
which maps $(u_0, W)$ to the unique solution $u$
(resp. $u^N$ for \eqref{eqn:IEN}) of the integral equation \eqref{eqn:IE} (resp. $\Theta^N$ for
\eqref{eqn:IEN}).  The map $\Theta$ (resp. $\Theta^N$)
is globally Lipschitz continuous.
Finally we have that $\Theta^N(u_0,W) \to \Theta(u_0,W)$ strongly in $\CC([0,T];\h^t)$ for
every pair $(u_0,W) \in \h^t \times \CC([0,T]; \h^t)$.
\end{theorem}

\noindent{\bf Proof.} The existence and uniqueness of local
solutions to the integral equation (\ref{eqn:IE}) is a simple
application of the contraction mapping principle, following
arguments similar to those employed in the proof of Theorem \ref{t:4.9}. 
Extension to a global solution may be achieved by repeating
the local argument on successive intervals.

Now let $u^{(i)}$ solve
$$u^{(i)}=u^{(i)}_0+\int_0^\tau F(u^{(i)})(s)ds +\sqrt{2}W^{(i)}(\tau), \quad\tau\in[0,T],$$
for $i=1,2$. Subtracting and using the Lipschitz property of $F$
shows that $e=u^{(1)}-u^{(2)}$ satisfies
\begin{eqnarray*}
\|e(\tau)\|_t\leq\|u^{(1)}_0-u^{(2)}_0\|_t+L\int_0^\tau\|e(s)\|_tds
+\sqrt{2}\|W^{(1)}(\tau)-W^{(2)}(\tau)\|_t\\
\leq \|u^{(1)}_0-u^{(2)}_0\|_t+L\int_0^\tau\|e(s)\|_tds
+\sqrt{2}\sup_{0\leq s \leq T}\|W^{(1)}(s)-W^{(2)}(s)\|_t.
\end{eqnarray*}
By application of the Gronwall inequality we find that
$$\sup_{0\leq\tau\leq T}\|e(\tau)\|_t\leq C(T)\big(\|u^{(1)}_0-u^{(2)}_0\|_t+
\sup_{0\leq s \leq T}\|W^{(1)}(s)-W^{(2)}(s)\|_t\big)$$ and the
desired continuity is established.

Now we prove pointwise convergence of $\Theta^N$ to $\Theta$.
Let $e=u-u^N$ where $u$ and $u^N$ solve (\ref{eqn:IE}), (\ref{eqn:IEN})
respectively.
The pointwise convergence of $\Theta^N$ to $\Theta$
is established by proving that $e \to 0$ in $\CC([0,T];\h^t)$.
Note that
$$F(u)-F^N(u^N)=\bigl(F^N(u)-F^N(u^N)\big)+\bigl(F(u)-F^N(u)\bigr).$$
Also, by Lemma \ref{l:4.10b}, $\|F^N(u)-F^N(u^N)\|_t \le L\|e\|_t$.
Thus we have
$$\|e\|_t\leq
L\int_0^\tau\|e(s)\|_tds+\int_0^\tau\|F\big(u(s)\big)-F^N\big(u(s)\big)\|_tds.$$
Thus, by Gronwall, it suffices to show that $$\delta^N:=\sup_{0\leq s\leq
T}\|F\big(u(s)\big)-F^N\big(u(s)\big)\|_t$$ tends to zero as
$N\rightarrow\infty$. Note that
\begin{eqnarray*}
F(u)-F^N(u)&=&\cC D\Phi(u)-\cC P^ND\Phi(P^Nu)\\
&=&(I-P^N)\cC D\Phi(u)+P^N\big(\cC D\Phi(u)-\cC D\Phi(P^Nu)\big).
\end{eqnarray*}
Thus, since $\C D\Phi$ is globally Lipschitz on $\mathcal{X}^t$, by Lemma
\ref{l:4.10}, and $P^N$
has norm one as a mapping from $\mathcal{X}^t$ into itself,
$$\|F(u)-F^N(u)\|_t  \le  \|(I-P^N)\cC D\Phi(u)\|_t+C\|(I-P^N)u\|_t.$$
By dominated convergence $\|(I-P_N)a\|_t \to 0$ for any fixed
element $a \in \mathcal{X}^t$. Thus,
because $\cC D\Phi$ is globally Lipschitz, by Lemma
\ref{l:4.10}, and as $u \in \CC([0,T];\mathcal{X}^t)$,
we deduce that it suffices to bound
$\sup_{0\leq s\leq T}\|u(s)\|_t$. 
But such a bound is a consequence of the existence theory outlined
at the start of the proof, based on the proof of Theorem \ref{t:4.9}.
$\quad\Box$

The following is a straightforward corollary of the preceding theorem:

\begin{corollary}\label{c:4.13} 
For any pair 
$(u_0,W) \in \h^t \times \CC([0,T]; \h^t)$ we define the 
point It{\^o} map
\begin{equation*}
\Theta_\tau \colon \h^t \times \CC([0,T];\h^t) \rightarrow \h^t
\end{equation*}
(resp. $\Theta^N_\tau$ for \eqref{eqn:IEN}) 
which maps $(u_0, W)$ to the unique solution $u(\tau)$
of the integral equation \eqref{eqn:IE} (resp. $u^N(\tau)$ for
\eqref{eqn:IEN}) at time $\tau$ .  The map $\Theta_\tau$ (resp. $\Theta^N_\tau$)
is globally Lipschitz continuous.
Finally we have that $\Theta^N_\tau(u_0,W) \to \Theta_\tau(u_0,W)$ for
every pair $(u_0,W) \in \h^t \times \CC([0,T]; \h^t)$.
\end{corollary}

\begin{theorem}\label{t:4.14} Let Assumptions \ref{a:4.8} hold. Then
the measure $\mu$ given by \eqref{42} is invariant for \eqref{44}:
for all continuous bounded functions $\varphi: \h^t \to \R$ it
follows that, if $\EE$ denotes expectation with respect to the
product measure found from initial condition $u_0 \sim \mu$ and $W
\sim \WW$, the $\cC$-\as{Wiener measure}\index{Wiener!measure} on $\h^t$, then  $\EE
\varphi\bigl(u(\tau)\bigr)= \EE \varphi(u_0)$.
\end{theorem}

\noindent{\bf Proof.} We have that
\begin{equation}\label{419}
\bbE\varphi\big(u(\tau)\big)=\int\varphi\big(\Theta_\tau(u_0,W)\big)\mu(du_0)\mathbb{W}(dW),
\end{equation}
\begin{equation}\label{420}
\bbE\varphi(u_0)=\int\varphi(u_0)\mu(du_0).
\end{equation}
If we solve equation (\ref{44x}) with $u_0\sim\mu^N$ then, using
$\bbE^N$ with the obvious notation,
\begin{equation}\label{421}
\bbE^N\varphi\big(u^N(\tau)\big)=\int\varphi\big(\Theta^N_\tau(u_0,W)\big)\mu^N(du_0)\mathbb{W}(dW),
\end{equation}
\begin{equation}\label{422}
\bbE^N\varphi(u_0)=\int\varphi(u_0)\mu^N(du_0).
\end{equation}
Lemma \ref{l:4.15} below shows that, in fact,
$$\bbE^N\varphi\big(u^N(\tau)\big)=\bbE^N\varphi(u_0).$$
Thus it suffices to show that
\begin{equation}\label{423}
\bbE^N\varphi\big(u^N(\tau)\big)\rightarrow\bbE\varphi\big(u(\tau)\big)
\end{equation}
and
\begin{equation}\label{424}
\bbE^N\varphi(u_0)\rightarrow\bbE\varphi(u_0).
\end{equation}
Both of these facts follow from the dominated convergence theorem as
we now show. First note that
$$\bbE^N\varphi(u_0)=\int\varphi(u_0)e^{-\Phi(P^Nu_0)}\mu_0(du_0).$$
Since $\varphi(\cdot)e^{-\Phi\circ P^N}$ is bounded independently of
$N$, by $(\sup\varphi)e^{M_1}$, and since $(\Phi\circ P^N)(u)$
converges pointwise to $\Phi(u)$ on $\mathcal{X}^t$, we deduce that
$$\bbE^N\varphi(u_0)\rightarrow
\int\varphi(u_0)e^{-\Phi(u_0)}\mu_0(du_0)=\bbE\varphi(u_0)$$ so that
(\ref{424}) holds. The convergence in (\ref{423}) holds by a similar
argument. From \eqref{421} we have
\begin{equation}\label{421b}
\bbE^N\varphi\big(u^N(\tau)\big)=\int\varphi\big(\Theta^N_\tau(u_0,W)\big)e^{-\Phi(P^Nu_0)}\mu_0(du_0)\mathbb{W}(dW).
\end{equation}
The integrand is again dominated by $(\sup\varphi)e^{M_1}$.
Using the pointwise convergence of $\Theta^N_\tau$ to $\Theta_\tau$
on $\mathcal{X}^t\times \CC([0,T];\mathcal{X}^t)$, as proved in Corollary \ref{c:4.13}, as well as the pointwise
convergence of $(\Phi\circ P^N)(u)$ to $\Phi(u)$, the
desired result follows from dominated convergence: we find
that
$$\bbE^N\varphi\big(u^N(\tau)\big) \to \int\varphi\big(\Theta_\tau(u_0,W)\big)e^{-\Phi(u_0)}\mu_0(du_0)\mathbb{W}(dW)=
\bbE\varphi\big(u(\tau)\big).$$
The desired result follows. $\Box$

\begin{lemma} \label{l:4.15}
Let Assumptions \ref{a:4.8} hold. Then the measure $\mu^N$ given by
\eqref{x43} is invariant for \eqref{44x}: for all continuous bounded
functions $\varphi: \h^t \to \R$ it follows that, if $\EE^N$ denotes
expectation with respect to the product measure found from initial
condition $u_0 \sim \mu^N$ and $W \sim \WW$, the $\cC$-\as{Wiener measure}\index{Wiener!measure} on
$\h^t$, then  $\EE^N \varphi\bigl(u^N(\tau)\bigr)= \EE^N \varphi(u_0)$.
\end{lemma}

\noindent{\bf Proof.} Recall from Lemma \ref{l:prod}
that measure $\mu^N$ given by (\ref{x43})
factors as the independent product of two measures on $\mu_P$ on
$X^N$ and $\mu_{Q}$ on $X^{\perp}$. On $X^{\perp}$ the
measure is simply the Gaussian $\mu_{Q}=\mathcal {N}(0,\cC^\perp)$,
whilst $X^N$ the measure $\mu_{P}$ is finite
dimensional with density proportional to
\begin{equation}\label{425}
\exp\Big(-\Phi(p)-\frac{1}{2}\|(\cC^N)^{-\frac{1}{2}}p\|^2\Big).
\end{equation}

The equation (\ref{44x}) also decouples on the spaces $X^N$ and
$X^\perp$. On $X^\perp$ it gives the integral equation 
\begin{equation}\label{426}
q(\tau)=-\int_0^\tau q(s)+\sqrt{2}Q^N W(\tau)
\end{equation}
whilst on $X^N$ it gives the integral equation 
\begin{equation}\label{427}
p(\tau)=-\int_0^\tau \Bigl(p(s)+\cC^N D\Phi\bigl(p(s)\bigr)\Bigr)ds+\sqrt{2}P^N W(\tau).
\end{equation}
Measure $\mu_Q$ is preserved by (\ref{426}), because (\ref{426})
simply gives an (integral equation formulation of) the
Ornstein-Uhlenbeck process with desired Gaussian
invariant measure. On the other hand, equation (\ref{427}) is simply
(an integral equation formulation of)the
Langevin equation for measure on $\R^N$ with density
(\ref{425}) and a calculation with the Fokker-Planck
equation, as in Theorem \ref{t:4.10},
demonstrates the required invariance of $\mu_P$. $\Box$

\vspace{0.2in}
\subsection{Bibliographic Notes}

\begin{itemize}

\item Subsection \ref{ssec:set} describes general
background on Markov processes and invariant
measures. The book \cite{Nor97} is a good starting
point in this area. The book \cite{MT93} provides a good overview
of this subject area, from an applied and computational
statistics perspective. For continuous time Markov chains
see \cite{QZ1}. 

\item Subsection \ref{ssec:4.4} concerns MCMC methods.
The standard RWM was introduced in \cite{MRRT53}
and led, via the paper \cite{Has}, to the development of
the more general class of Metropolis-Hastings methods.
The paper \cite{Tie} is a key
reference which provides a framework for
the study of Metropolis-Hastings methods on
general state spaces. 
The subject of MCMC methods which are invariant with respect to the
target measure $\mu$ on infinite dimensional spaces
is overviewed in the paper \cite{CRSW12}. The
specific idea behind the Algorithm \ref{alg:pcn}
is contained in \cite[equation (15)]{Neal98}, in the finite dimensional
setting.
It is possible to show that, in the limit $\beta\to 0$, suitably 
interpolated output of Algorithm \ref{alg:pcn}
converges to solution of
the equation (\ref{44}): see \cite{PST11}.
Furthermore it is also possible to compute a spectral gap for the
Algorithm \ref{alg:pcn} in the infinite dimensional setting \cite{Vol13}.
This implies the existence of a dimension independent spectral
gap when finite dimensional approximation is used; in contrast
standard Metropolis-Hastings methods, such as Random Walk
Metropolis, have a dimension-dependent spectral gap which shrinks
with increasing dimension \cite{Vol13}.

\item Subsection \ref{ssec:pf} concerns SMC methods
and the foundational work in this area is overviewed
in the book \cite{del2004feynman}. The application
of those ideas to the solution of PDE inverse problems
was first demonstrated in \cite{kantas2013sequential}, where
the inverse problem is to determine the initial condition
of the Navier-Stokes equations from observations. The method
is applied to the elliptic inverse problem, with uniform priors,
in \cite{BJMS15}. The proof of Theorem \ref{th39} follows the very clear
exposition given in \cite{rebeschini2013can} in the context of
filtering for hidden Markov models.

\item Subsections \ref{ssec:start}--\ref{ssec:IDL} 
concern measure preserving continuous time dynamics.
The finite dimensional aspects of this subsection,
which we introduce for motivation, are covered in the
texts \cite{Oks98} and \cite{Gar85}; the first of
these books is an excellent introduction to the
basic existence and uniqueness theory, outlined in
a simple case in Theorem \ref{t:4.9}, whilst the second provides
an in depth treatment of the subject from the viewpoint
of the Fokker-Planck equation, as used in Theorem \ref{t:4.10}.
This subject has a long history which is overviewed
in the paper \cite{HSV05} where the idea is applied to
finding SPDEs which are invariant with respect to the
measure generated by a conditioned diffusion process.
This idea is generalized to certain conditioned
hypoelliptic diffusions in \cite{HSV11}.
It is also possible to study deterministic Hamiltonian
dynamics which preserves the same measure. This idea
is described in \cite{BPSSS11} in the same set-up as
employed here; that paper also contains references to
the wider literature. Lemma \ref{l:4.10} is 
proved in \cite{mattingly2011spde} and Lemma \ref{l:4.10b} in
\cite{PST11}
Lemma \ref{l:4.15} requires knowledge of the invariance of
Ornstein-Uhlenbeck processes
together with invariance of
finite dimensional first order Langevin equations with the form
of gradient dynamics subject to additive noise.
The invariance of the Ornstein-Uhlenbeck process
is covered in \cite{da1996ergodicity}
and invariance of finite dimensional SDEs using
the Fokker-Planck equation is discussed in \cite{Gar85}.
The $\cC$-Wiener process, and its properties, are described
in \cite{DapZab92}.

\item 
The primary focus of this section has been
on the theory of measure-preserving
dynamics, and its relations to algorithms. The SPDEs are of interest in
their own right as a theoretical object, but have particular importance
in the construction of MCMC methods, and in understanding the limiting
behaviour of MCMC methods. It is also important to appreciate
that MCMC and SMC methods are by no means the only tools available to
study the Bayesian inverse problem. In this context we note that
computing the expectation with respect to the 
\as{posterior}\index{posterior} can be reformulated
as computing the ratio of two expectations
with respect to the \as{prior}\index{prior}, 
the denominator being the normalization constant.
effectively in some such high dimensional integration problems; 
\cite{ks05} and \cite{niederreiter} are general references on the QMC methodology.
The paper \cite{kss11} is a survey on the theory of QMC for bounded 
integration domains and is relevant for uniform \as{prior}s\index{prior}.  
The paper \cite{ksww10} 
contains theoretical results for unbounded integration domains 
and is relevant to, for example, Gaussian \as{prior}s\index{prior!Gaussian}.
The use of QMC in plain uncertainty quantification (calculating
the pushforward of a measure through a map) is studied for elliptic PDEs 
with random coefficients in \cite{kss12} (uniform) 
and \cite{gknsss13} (Gaussian). More sophisticated integration tools
can be employed, using polynomial chaos representations of the
\as{prior}\index{prior} measure, and computing \as{posterior}\index{posterior}
expectations in a manner
which exploits sparsity in the map from unknown random coefficients to 
measured data; see \cite{ScSt11,schillings2013sparse}.
Much of this work, viewing uncertainty quantification from
the point of high dimensional integration, has its roots in
early papers concerning plain uncertainty quantification
in elliptic PDEs with random coefficients; the paper
\cite{BTZ} was foundational in this area.

\end{itemize}

\section{Conclusions}
\label{sec:con}

We have highlighted a theoretical treatment for Bayesian inversion
over infinite dimensional spaces. The resulting framework is appropriate
for the mathematical analysis of inverse problems, as well as
the development of algorithms. For example, on the analysis side,
the idea of MAP estimators, which links the Bayesian approach with classical
regularization, developed for Gaussian priors in \cite{DLSV13},
has recently been extended to other prior models in \cite{Hel14};
the study of contraction of the posterior distribution to a Dirac
measure on the truth underlying the data is undertaken in
\cite{ALS12, ASZ12,Vol13}. On the algorithmic side algorithms for
Bayesian inversion in geophysical applications are formulated
in \cite{BT14a,BT14b}, and on the computational statistics side
methods for optimal experimental design are formulated in
\cite{APSG14, AGG14}. All of these cited papers build on the
framework developed in detail here, and first outlined in 
\cite{article:Stuart2010}. It is thus anticipated that the framework
herein will form the bedrock of other, related, developments
of both the theory and computational practice of Bayesian inverse
problems.

\section{Appendix}

\subsection{Function Spaces}
\label{ssec:ifs}

In this subsection we briefly define the Hilbert and
Banach spaces that will be important in our developments of
probability and integration in infinite dimensional spaces. As a consequence
we pay particular attention to the issue of \as{separability}\index{separable}
(the existence of a countable dense subset)
which we require in that context.  We primarily 
restrict our discussion to $\R$ or ${\mathbb C}$-valued 
functions, but the reader will easily be able to extend to 
$\R^n$-valued or $\R^{n \times n}$-valued situations, 
and we discuss Banach-space valued 
functions at the end of the subsection.

\subsubsection{$\ell^p$  and $L^{p}$ Spaces}

\label{sssec:elp}

Consider real-valued sequences $u=\{u_j\}_{j=1}^{\infty} \in \R^{\infty}.$
Let $w \in \R^{\infty}$ denote
a positive sequence so that $w_j>0$ for each $j \in \N$.
For every $p \in [1,\infty)$ we define
$$\ell_w^p=\ell_w^p(\N;\R)=\Bigl\{u \in \R\Big|
\sum_{j=1}^{\infty} w_j |u_j|^p<\infty\Bigr\}.$$
Then $\ell_w^p$ is a Banach space when equipped with the
norm
$$\|u\|_{\ell_w^p}=\Bigl(\sum_{j=1}^{\infty} w_j |u_j|^p\Bigr)^{\frac{1}{p}}.$$
In the case $p=2$ the resulting spaces are Hilbert spaces
when equipped with the inner-product
$$\langle u,v \rangle=\sum_{j=1}^\infty w_j u_j v_j.$$
These $\ell^p$ spaces, with $p\in [1,\infty)$, are separable.
Throughout we simply write $\ell^p$ for the spaces $\ell^p_{w}$
with $w_j \equiv 1$.
In the case $w_j \equiv 1$ we extend the definition of Banach spaces to 
the case $p=\infty$ by defining
$$\ell^{\infty}=\ell^{\infty}(\N;\R)=\Bigl\{u \in \R\Big|
{\rm sup}_{j\in \N} (|u_j|)<\infty\Bigr\}$$
and
$$\|u\|_{\ell^{\infty}}={\rm sup}_{j\in \N} (|u_j|).$$
The space $\ell^\infty$ of bounded 
sequences is \textit{not} separable.
Each element of the sequence $u_j$ is real-valued, but the definitions
may be readily extended to complex-valued, $\R^n$-valued 
and $\R^{n \times n}$-valued sequences, 
replacing $|\cdot|$ by the complex modulus, the vector $\ell^p$ norm 
and the operator $\ell^p$ norm on matrices respectively.

We now extend the idea of $p$-summability to functions, and to
$p$-integrability.
Let $D$ be a bounded open set in $\R^d$ with Lipschitz boundary and 
define the space $~L^p=L^p(D;\R)$ of Lebesgue measurable functions
$f:D\rightarrow \R$ with norm $\|\cdot\|_{L^p(D)}$ defined by
\begin{eqnarray*}
\| f \|_{L^p(D)} := \left\{ \begin{array}{cc}
          \left( \int_{D} |f|^p \, dx \right)^{\frac{1}{p}}
            & \mbox{for } \, 1 \leq p < \infty
             \medskip

               \\
             \mbox{ess} \sup_{D} |f| & \mbox{for } \, p = \infty.
            \end{array}   \right  .
\end{eqnarray*}
In the above definition we have used the notation
$$
\mbox{ess} \sup_{D} |f| = \inf \left\{ C: |f| \leq C \mbox{ a.e.\ on
} D\right\}.
$$
Here $a.e$.\ is with respect to Lebesgue measure and the integral
is, of course, the Lebesgue integral. Sometimes we drop explicit reference to
the set $D$ in the norm and simply write $\|\cdot\|_{L^p}$. For Lebesgue
measurable functions
$f:D\rightarrow \R^n$ the norm is readily extended replacing $|f|$ under the
integral by the vector $p$-norm on $\R^n$. Likewise we may consider 
Lebegue measurable
$f:D\to \R^{n \times n}$, using the operator $p$-norm on
$\R^{n \times n}$. In all these cases we
write $L^p(D)$ as shorthand for $L^p(D;X)$ where $X=\R,\, \R^n$ or $\R^{n \times n}$. 
Then $L^p(D)$ is the vector space of all (equivalence classes of)
measurable functions
$f: D\rightarrow \R$ for which $\|f \|_{L^p(D)} < \infty$.
The space $L^p(D)$ is \as{separable}\index{separable} for $p \in [1,  \infty)$
whilst $L^{\infty}(D)$ is not separable.
We define periodic versions of $L^p(D)$, denoted
by $L^p_{\rm{per}}(D)$, in the case where $D$ is
a unit cube; these spaces are defined as the completion 
of $\CC^{\infty}$
periodic functions on the unit cube, with respect to the $L^p$-norm.
If we define $\bbT^d$ to be the $d$-dimensional unit torus then
we write $L_{\rm per}^p([0,1]^d)=L^p(\bbT^d)$.
Again these spaces are separable for $1 \le p<\infty$, but not
for $p=\infty.$

\subsubsection{Continuous and H\"{o}lder Continuous Functions}
\label{sssec:ch}

Let $D$ be an open and bounded set in $\R^d$ with Lipschitz
boundary.  We will denote by $\CC(\overline{D},\R)$, or simply $\CC(\overline{D})$, 
the space of  continuous functions
$f:\overline{D} \rightarrow \R$.
When equipped with the supremum norm,
$$
\|f \|_{\CC(\overline{D})} = \sup_{x \in \overline{D}} |f(x)|,
$$
$\CC(\overline{D})$ is a Banach spaces.  
Building on this we define the space $\CC^{0,\gamma}(\overline{D})$
to be the space of functions in $\CC(\overline{D})$ which are 
\as{H\"{o}lder}\index{H\"{o}lder spaces}
with any exponent $\gamma \in (0,1]$ with norm 
\begin{equation}
\label{eq:hb}
\|f \|_{\CC^{0,\gamma}(\overline{D})} = \sup_{x \in \overline{D}} |f(x)|+
\sup_{x,y\in \overline{D}}\Bigl(\frac{|f(x)-f(y)|}{|x-y|^{\gamma}}\Bigr).
\end{equation}
The case $\gamma=1$ corresponds to Lipschitz functions.

We remark that $\CC(\overline{D})$ is separable since $\overline{D}\subset \R^d$ is compact here. 
The space of H\"{o}lder functions $\CC^{0,\gamma}(\overline{D};\R)$ 
is, however, {\em not} \as{separable}\index{separable}. 
Separability can be recovered by working in the subset of $\CC^{0,\gamma}(\overline{D};\R)$
where, 
in addition to \eqref{eq:hb} being finite,  
$$\lim_{y \to x} {|f(x) - f(y)| \over |x-y|^\gamma} = 0,$$ 
uniformly in $x;$ 
we denote the resulting separable space by $\CC^{0,\gamma}_0(\overline{D}, \R).$ 
This is analogous to the fact that the space of bounded
measurable functions is not separable, while the space of continuous functions on a compact domain is. 
Furthermore it may be shown that $\CC^{0,\gamma'} \subset \CC_0^{0,\gamma}$ 
for every $\gamma' > \gamma$.
All of the preceding spaces can be generalized to
functions $\CC^{0,\gamma}(\overline{D},\R^n)$ and 
$\CC^{0,\gamma}_0(\overline{D}, \R^n);$ 
they may also be extended to periodic functions
on the unit torus $\bbT^d$ found by identifying opposite faces
of the unit cube $[0,1]^d$. 
The same separability issues arise for these generalizations.

\subsubsection{\as{Sobolev Spaces}\index{Sobolev!space}}
\label{sssec:sob}

We define Sobolev spaces of functions with integer number
of derivatives, extend to fractional and negative derivatives,
and make the connection with \as{Hilbert scales}\index{Hilbert scale}.
Here $D$ is a bounded open set in $\R^d$ with Lipschitz
boundary.  In the context of a function $u \in L^2(D)$
we will use the notation $\pdxi{u}$ to denote the weak derivative with
respect to $x_i$ and the notation 
$\nabla u$ for the weak gradient.

The {\em Sobolev space $W^{r,p}(D)$} consists of all $L^p$-integrable functions
$u:D\to \R$ whose $\alpha^{th}$ order weak derivatives exist and 
are $L^p$-integrable for all $|\alpha|\le r$:
\begin{equation}
W^{r,p}(D) = \left\{u \Big| D^\alpha u \in L^p(D)\mbox{ for }|\alpha|\le r \right\}
\label{eq:matt_star}
\end{equation}
with norm
\begin{equation}\label{eq:circle21}
\|u \|_{W^{r,p}(D)} 
=\left\{ 
\begin{array}{ll}
\left(\sum_{|\alpha|\le r}\|D^{\alpha} u \|^p_{L^p(D)}\right)^{\frac{1}{p}}&\mbox{for }\;1\le p<\infty,\\
\sum_{|\alpha|\le r}\|D^{\alpha} u \|_{L^\infty(D)}&\mbox{for }\;p=\infty.
\end{array}
\right.
\end{equation}
We denote $W^{r,2}(D)$ by $H^r(D)$. 
We define periodic versions of $H^s(D)$, denoted
by $H^s_{\rm{per}}(D)$, in the case where $D$ is
a unit cube $[0,1]^d$; these spaces are defined as the completion 
of $\CC^{\infty}$
periodic functions on the unit cube, with respect to the $H^s$-norm.
If we define $\bbT^d$ to be $d$-dimensional unit torus,
we then write $H^s(\bbT^d)=H^s_{\rm per}([0,1]^d)$.

The spaces $H^s(D)$ with $D$ a bounded open set in $\R^d$,
and $H^s_{\rm per}([0,1]^d)$, are \as{separable}\index{separable}
Hilbert spaces.  In particular if we define the inner-product 
$(\cdot,\cdot)_{L^2(D)}$ on $L^2(D)$ by 
$$(u,v)_{L^2(D)}:=\int_D u(x)v(x)dx$$
and define the resulting norm  $\|\cdot\|_{L^2(D)}$ by
the identity
$$\|u\|_{L^2(D)}^2=(u,u)_{L^2(D)}$$
then the space $H^1(D)$ is 
a separable Hilbert space with inner product
\begin{equation*}
\langle u, v\rangle_{H^1(D)} = (u,v)_{L^2(D)} + (\nabla u, \nabla v)_{L^2(D)}
\end{equation*}
and norm \eqref{eq:circle21} with $p=2.$
Likewise the space $H^1_0(D)$ is a separable Hilbert space with
inner product
\begin{equation*}
\langle u, v\rangle_{H^1_0(D)} =  (\nabla u, \nabla v)_{L^2(D)}
\end{equation*}
and norm
\begin{equation}
\|u \|_{H^1_0(D)} = \|\nabla u \|_{L^2(D)}.
\label{eq:circle212}
\end{equation}

As defined above, Sobolev spaces concern integer numbers of derivatives.
However the concept can be extended to fractional derivatives and there
is then a natural connection to \as{Hilbert scales}\index{Hilbert scale}
of functions. To explain
this we start our development in the periodic setting. 
Recall that, given an element $u$ in 
$L^2(\T^d)$, we can decompose it as a Fourier series:
\begin{equation*}
u(x) = \sum_{k \in \Z^d} u_k e^{2\pi i\langle k,x \rangle}\;,
\end{equation*}
where the identity holds for (Lebesgue) almost every $x \in \T^d$. Furthermore, the $L^2$ norm of $u$ is given by Parseval's 
identity $\|u\|_{L^2}^2 = \sum |u_k|^2$. 
The fractional Sobolev space $H^s(\T^d)$ for $s \ge 0$  is given by the subspace of functions $u \in L^2(\T^d)$
such that
\begin{equation}
\label{e:defHs}
\|u\|_{H^s}^2 := \sum_{k \in \Z^d} (1+4\pi^2|k|^2)^s |u_k|^2 < \infty\;.
\end{equation}
Note that this is a separable Hilbert space by virtue of 
$\ell_w^2$ being separable. Note also that
$H^0(\T^d)=L^2(\T^d)$ and that, for positive integer $s$, the definition agrees
with the definition $H^s(\bbT^d)=W^{s,2}(\bbT^d)$ obtained from
\eqref{eq:matt_star} with the obvious generalization from $D$ to
$\bbT^d$. 
For $s < 0$, we define $H^s(\T^d)$ as the closure of $L^2$
under the norm \eqref{e:defHs}. The spaces $H^s(\T^d)$ 
for $s<0$ may also be defined via duality.
The resulting spaces $H^s$ are separable for all $s \in \R$.

We now link the spaces $H^s(\T^d)$ to a specific 
\as{Hilbert scale}\index{Hilbert scale} of spaces.
Hilbert scales are families of spaces defined by $\CD(A^{s/2})$ for
$A$ a positive, unbounded, self-adjoint operator on a Hilbert space.
To view the fractional Sobolev spaces from this perspective
let $A=I-\triangle$ with domain $H^2(\T^d)$, noting that the eigenvalues of
$A$ are simply $1+4\pi^2|k|^2$ for $k \in \Z^d$. We thus see that,
by the spectral decomposition theorem, $H^s=\CD(A^{s/2})$, 
and we have $\|u\|_{H^s} = \|A^{s/2}u\|_{L^2}$.
Note that we may work in the space of
real-valued functions where the eigenfunctions of $A$,
$\{\varphi_j\}_{j=1}^{\infty}$,
comprise sine and cosine functions; the eigenvalues of $A$,
when ordered on a one-dimensional lattice, then satisfy
$\alpha_j \asymp j^{2/d}$. This is relevant to the more general
perpsective of \as{Hilbert scales}\index{Hilbert scale} that we now introduce.

We can now generalize the previous construction of
fractional Sobolev spaces to more general domains than the torus. 
The resulting spaces do not, in general, coincide with Sobolev
spaces, because of the effect of the boundary conditions 
of the operator $A$ used in the construction.
On an arbitrary bounded open
set $D \subset \R^d$ with Lipschitz boundary
we consider a positive self-adjoint operator
$A$ satisfying Assumption \ref{a:1.3} so that its eigenvalues 
satisfy $\alpha_j \asymp j^{2/d}$; then we define
the spaces ${\mathcal H}^s=\CD(A^{s/2})$ for $s>0.$
Given a Hilbert space $(H,\langle\cdot,\cdot\rangle, \|\cdot\|)$ 
of real-valued functions on a bounded open set $D$ in $\R^d$,
we recall from Assumption \ref{a:1.3} the orthonormal basis for $H$
denoted by $\{\varphi_j\}_{j=1}^{\infty}$. 
Any $u\in H$ can be written as
$$u=\sum_{j=1}^{\infty} \langle u, \varphi_j \rangle \varphi_j.$$
Thus
\begin{equation}
\label{eq:norm2}
\mathcal {H}^s=\Big\{
u: D \to \R \Big|\|w\|_{\mathcal {H}^s}^2<\infty\Big\}
\end{equation}
where, for $u_j=\langle u,\varphi_j \rangle$,
$$\|u\|_{\mathcal {H}^s}^2=\sum_{j=1}^\infty j^{\frac{2s}{d}}|u_j|^2.$$
In fact $\mathcal {H}^s$ is a Hilbert space: for 
$v_j=\langle v,\varphi_j \rangle$ we may define the inner-product
$$\langle u,v \rangle_{\mathcal {H}^s}=\sum_{j=1}^{\infty} j^{\frac{2s}{d}}u_j v_j.$$
For any $s>0$, the Hilbert space
$(\mathcal {H}^s,\langle\cdot,\cdot\rangle_{\mathcal {H}^t}, \|\cdot\|_{\mathcal {H}^t})$ 
is a subset of the original Hilbert space $H$; for $s<0$ the spaces are
defined by duality and are supersets of $H$. Note also that we have
Parseval-like identities showing that the $\mathcal {H}^s$ norm on a function
$u$ is equivalent to the $\ell^2_w$ norm on the sequence $\{u_j\}_{j=1}^{\infty}$ with the choice $w_j=j^{2s/d}$. 
The spaces $\mathcal {H}^s$ are separable Hilbert spaces for any $s \in \R.$

\subsubsection{Other Useful Function Spaces}
\label{sssec:other}

As mentioned in passing,
all of the preceding function spaces can be extended to functions
taking values in $\R^n, \R^{n\times n}$; thus we may then 
write $\CC(D;\R^n), L^p(D;\R^n), H^s(D;\R^n)$, for example. More
generally we may wish to consider functions taking values 
in a separable Banach space $E$.
For example when we are interested in solutions of time-dependent
PDEs then these may be formulated as ordinary differential
equations taking values in a separable Banach space $E$, with norm
$\|\cdot\|_E$. It is then natural to consider  
Banach spaces such as $L^2((0,T);E)$ and $\CC([0,T];E)$ with norms
$$\|u\|_{L^2((0,T);E)}=\sqrt{\Bigl(\int_0^T \|u(\cdot,t)\|_E^2 dt\Bigr)},\quad
\|u\|_{\CC([0,T];E)}=\sup_{t \in [0,T]}\|u(\cdot,t)\|_E.$$
These norms can be generalized in a variety of ways,
by generalizing the norm on the time variable.

The preceding idea of defining Banach space-valued $L^p$ spaces defined
on an interval $(0,T)$ can be taken further to define Banach space-valued
$L^p$ spaces defined on a measure space. Let $(\CM,\nu)$ any countably 
generated measure space, like for example any Polish space
(a separable completely metrizable topological space)
equipped with a positive Radon measure $\nu$. Again let $E$ denote
a separable Banach space. Then
$L^p_{\nu}(\CM;E)$ 
is the space of functions $u: \CM \to E$ with norm (in this
defintion of norm we use Bochner integration, defined in the next
subsection)
$$\|u\|_{L^p_{\nu}(\CM;E)}=\Bigl(\int_{\CM} \|u(x)\|_E^p \nu(dx)\Bigr)^{\frac{1}{p}}.$$
For $p \in (1,\infty)$ these spaces are separable. 
However, separability  fails to hold for $p=\infty.$ 
We will use these Banach spaces in the case where $\nu$ is a probability
measure $\P$, with corresponding expectation $\bbE$, and we then have
$$\|u\|_{L^p_{\P}(\CM;E)}=\Bigl(\bbE \bigl(\|u\|_E^p\bigr) \Bigr)^{\frac{1}{p}}.$$

\subsubsection{Interpolation Inequalities and \as{Sobolev\index{Sobolev!embedding} Embeddings}}
\label{ssec:iise}

Here we state some useful interpolation inequalities, and 
use them to prove a Sobolev embedding result, all in the context of
fractional Sobolev spaces, in the generalized sense defined through
a \as{Hilbert scale}\index{Hilbert scale} of functions.

Let $p,q \in [1,\infty]$ be a pair of conjugate exponents so
that $p^{-1}+q^{-1}=1$. Then for any positive real $a,b$ we have
the Young inequality 
\begin{equation*}
ab \le {a^p \over p} + {b^q \over q}\;.
\end{equation*}
As a corollary of this elementary bound, 
we obtain the following H\"older inequality
Let $(\CM, \mu)$ be a measure space and denote the 
norm $\|\cdot\|_{L^p_{\nu}(\CM;\R)}$ by $\|\cdot\|_{p}.$
For $p,q \in [1,\infty]$ as above and 
$u,v \colon \CM \to \R$ a pair of measurable functions we have
\begin{equation}
\label{hol}
\int_{\CM} |u(x) v(x)|\,\mu(dx) \le \|u\|_p \, \|v\|_q.
\end{equation}
From this H\"older-like inequality the following interpolation bound
results: let $\alpha \in [0,1]$ and let $L$ denote a (possibly unbounded) self-adjoint  operator on the Hilbert space 
$(H,\langle\cdot,\cdot\rangle, \|\cdot\|)$.
Then, the bound
\begin{equation}
\label{eq:interp}
\|L^\alpha u\| \le \|Lu\|^\alpha \|u\|^{1-\alpha}
\end{equation}
holds for every $u \in \CD(L) \subset H.$

Now assume that $A$ is a self-adjoint unbounded operator
on $L^2(D)$ with $D \subset \R^d$ a bounded open set with Lipschitz
boundary. Assume further that $A$ has eigenvalues
$\alpha_j\asymp j^{\frac{2}{d}}$ and define the \as{Hilbert scale}\index{Hilbert
scale} of spaces $\cH^t=\CD(A^{\frac{t}{2}})$. 
An immediate corollary of the bound \eqref{eq:interp}, obtained
by choosing $H = \cH^s$, $L = A^{t-s\over 2}
$,
and $\alpha = (r-s)/(t-s)$, is:
\begin{lemma}
\label{l:intA}
Let Assumption \ref{a:1.3} hold.
Then for any $t>s$, any $r \in [s,t]$ and any $u \in {\mathcal H}^t$
it follows that
\begin{equation*}
\|u\|_{\cH^r}^{t-s} \le \|u\|_{\cH^t}^{r-s} \|u\|_{\cH^s}^{t-r}.
\end{equation*}
\end{lemma}

It is of interest to bound the $L^p$ norm of a function in terms of one of 
the \as{fractional Sobolev norms}\index{Sobolev!space (fractional)}, or more
generally in terms of norms from a \as{Hilbert scale}\index{Hilbert scale}.
To do this we need to not only make assumptions on the eigenvalues of the operator
$A$ which defines the \as{Hilbert scale}\index{Hilbert scale}, 
but also on the behaviour of the corresponding orthonormal
basis of eigenfunctions in $L^{\infty}$. To this end we let Assumption~\ref{def:triangle} hold.
It then turns out that bounding the $L^\infty$ norm is rather straightforward and
we start with this case.

\begin{lemma}\label{lem:Linfty}
Let Assumption \ref{def:triangle} hold and define the resulting Hilbert
scale of spaces $\cH^s$ by \eqref{eq:norm2}. 
Then for every $s > {d \over 2}$, the space $\cH^s$ 
is contained in the space $L^{\infty}(D)$ 
and there exists a constant $K_1$ such that $\|u\|_{L^\infty} \le K_1 \|u\|_{\cH^s}$.
\end{lemma}

\begin{proof}
It follows from Cauchy-Schwarz that
\begin{equation*}
\frac{1}{C}\|u\|_{L^{\infty}} \le \sum_{k\in \Z^d} |u_k| \le \Bigl(\sum_{k\in \Z^d} (1+|k|^2)^s |u_k|^2\Bigr)^{1/2} \Bigl(\sum_{k\in \Z^d} (1+|k|^2)^{-s}\Bigr)^{1/2}\;.
\end{equation*}
Since the sum in the second factor converges if and only if $s > {d\over 2}$, the claim follows.
\end{proof}

As a consequence of Lemma~\ref{lem:Linfty}, we are able to obtain a more general \index{Sobolev!embedding}\as{Sobolev embedding}  for all $L^p$ spaces:

\begin{theorem}[{\bf Sobolev Embeddings}]
\label{t:sob}
Let Assumption \ref{def:triangle} hold, define the resulting Hilbert
scale of spaces $\cH^s$ by \eqref{eq:norm2} 
and assume that $p \in [2,\infty]$.
Then, for every $s > {d \over 2} - {d\over p}$, the space $\cH^s$ is contained in 
the space $L^p(D)$
and there exists a constant $K_2$ such that $\|u\|_{L^p} \le K_2 \|u\|_{\cH^s}$.
\end{theorem}

\begin{proof}
The case $p=2$ is obvious and the case $p=\infty$ has already been shown, so it remains to show the claim for $p \in (2,\infty)$.
The idea is to divide the space of eigenfunctions
into ``blocks'' and to estimate separately the
$L^p$ norm of every block. More precisely, we define a sequence of functions $u^{(n)}$ by
\begin{equation*}
u^{(-1)} = u_0\,\varphi_0 \;,\quad u^{(n)} = \sum_{2^n \le j < 2^{n+1}} u_j\, \varphi_j\;,
\end{equation*}
where the $\varphi_j$ are an orthonormal basis of eigenfunctions for $A$,
so that $u = \sum_{n \ge -1} u^{(n)}$. For $n \ge 0$ the H\"older inequality gives
\begin{equation}\label{e:boundLp}
\|u^{(n)}\|_{L^p}^p \le \|u^{(n)}\|_{L^2}^2 \|u^{(n)}\|_{L^\infty}^{p-2}\;.
\end{equation}
Now set $s' = {d\over 2} + \eps$ for some $\eps>0$ and note that the construction of $u^{(n)}$, together with Lemma~\ref{lem:Linfty}, gives the bounds
\begin{equation}
\|u^{(n)}\|_{L^2} \le K2^{-ns/d} \|u^{(n)}\|_{\cH^s}\;,\quad \|u^{(n)}\|_{L^\infty} \le K_1\|u^{(n)}\|_{\cH^{s'}} \le K 2^{n(s'-s)/d} \|u^{(n)}\|_{\cH^s}\;.
\end{equation}
Inserting this into \eqref{e:boundLp}, we obtain (possibly for an enlarged $K$)
\begin{align*}
\|u^{(n)}\|_{L^p} &\le K\|u^{(n)}\|_{\cH^s} 2^{n \bigl((s'-s){p-2\over p} - {2s\over p}\bigr)/d} = K\|u^{(n)}\|_{\cH^s} 2^{n \bigl(\eps{p-2\over p}+ {d\over 2} - {d\over p} -s\bigr)/d}\\
&\le K\|u\|_{\cH^s} 2^{n \bigl(\eps + {d\over 2} - {d\over p}-s\bigr)/d}\;.
\end{align*}
It follows that $\|u\|_{L^p} \le |u_0| + \sum_{n \ge 0}\|u^{(n)}\|_{L^p} \le K_2\|u\|_{\cH^s}$, provided that the exponent appearing in this
expression is negative which, since $\eps$ can be chosen arbitrarily small, is precisely the case whenever $s > {d \over 2} - {d\over p}$.
\end{proof}

\subsection{Probability and Integration In Infinite Dimensions}
\label{ssec:pmii}

\subsubsection{Product Measure for \as{i.i.d.\ Sequences}\index{i.i.d.\ sequence}}
\label{ssec:iid}

Perhaps the most straightforward setting in which probability measures
in infinite dimensions are encountered is when studying i.i.d.\ sequences
of real-valued random variables. Furthermore, this is our basic building block
for the construction of random functions -- see subsection \ref{ssec:2.1} --
so we briefly overview the subject.
Let $\P_0$ be a probability measure on $\R$ so that $(\R,\cBc(\R),\P_0)$ is a 
probability space and consider
the i.i.d.\ sequence $\xi:=\{\xi_j\}_{j=1}^{\infty}$ with $\xi_1 \sim \P_0$.

The construction of such a sequence can be formalised as follows. We consider $\xi$ as a 
random variable taking values in the space $\R^\infty$ endowed with the product topology, i.e.\
the smallest topology for which the projection maps $\ell_n\colon \xi \mapsto \xi_n$
 are continuous
for every $n$.
This is a complete metric space; an example of a distance generating
the product topology is given by
$$d(x,y)=\sum_{n=1}^\infty 2^{-n}\frac{|x_n-y_n|}{1+|x_n-y_n|}\;.$$ 
Since we are considering a \textit{countable} product, the resulting 
$\sigma$-algebra $\cBc(\R^{\infty})$ coincides with the product $\sigma$-algebra, which is
the smallest $\sigma$-algebra for which all $\ell_n$'s are measurable.

In what follows we need the notion of
the \textit{pushforward} of a probability measure under
a measurable map. If $f:\CB_1 \to \CB_2$ is a measurable map between 
two measurable spaces $\bigl(\CB_i,\cBc(\CB_i)\bigr)$ $i=1,2$
and $\mu_1$ is a probability measure on $\CB_1$ then
$\mu_2=f^\sharp\mu_1$ denotes the pushforward probability measure on $\CB_2$
defined by $\mu_2(A)=\mu_1\bigl(f^{-1}(A)\bigr)$ for all $A \in \cBc(\CB_2)$.
(The notation $f^*\mu$ is sometimes used in place of $f^\sharp\mu$, 
but we reserve this notation for adjoints.)
Recall that in section \ref{sec:prior} we construct random functions via
the random series \eqref{21} whose coefficients are constructed
from an i.i.d sequence. Our interest is in studying the pushforward
measure ${\mathcal F}^\sharp \P_0$ 
where ${\mathcal F}: \R^{\infty} \to X'$ is defined by 
\begin{equation}\label{21F}
{\mathcal F}\xi =m_0+\sum_{j=1}^\infty \gamma_j\xi_j \phi_j.
\end{equation}
In particular section \ref{sec:prior} is devoted to determing
suitable separable Banach spaces $X'$ on which to define the
pushforward measure.

With the pushforward notation at hand, 
we may also describe \as{Kolmogorov's extension theorem} 
\index{Kolmogorov!extension theorem}
which can be stated as follows.

\begin{theorem}{\bf (\as{Kolmogorov Extension})}
\label{theo:extension}\index{Kolmogorov!extension theorem}
Let $\CX$ be a Polish space
and let $\CI$ be an arbitrary set. Assume that, for any finite subset
$A \subset \CI$, we are given a probability measure $\P_A$ on the finite
product space $\CX^A$. Assume furthermore that the family of measures $\{\P_A\}$ is consistent
in the sense that if $B \subset A$ and $\Pi_{A,B}\colon \CX^A \to \CX^B$ denotes the 
natural projection map, then $\Pi_{A,B}^\sharp \P_A = \P_B$. 
Then, there exists a unique probability measure $\P$ on $\CX^\CI$ endowed with the product 
$\sigma$-algebra with the property that $\Pi_{\CI,A}^\sharp \P = \P_A$ for every finite
subset $A \subset \CI$.
\end{theorem}

Loosely speaking, one can interpret this theorem as stating that if one knows the law of any \textit{finite}
number of components of a random vector  or function
then this determines the law of the \textit{whole} 
random vector or function; in particular in the case of the random
function  this comprises {\em uncountably} many components.
This statement is thus highly non-trivial as
soon as the set $\CI$ is infinite since we have \textit{a priori} defined $\P_A$ only
for finite subsets $A\subset \CI$ and the theorem allows us to extend this uniquely also to infinite 
subsets.

As a simple application, we can use this theorem to define 
the infinite product measure $\P=\bigotimes_{k=1}^{\infty} \P_0$
as the measure given by \as{Kolmogorov's Extension Theorem} 
\ref{theo:extension}\index{Kolmogorov!extension theorem} 
if we take as our family of specifications
$\P_A=\bigotimes_{k \in A} \P_0$.
Our i.i.d.\ sequence $\xi$ is then naturally defined as a random sample taken from the probability 
space $\bigl(\R^{\infty},\cBc(\R^{\infty}),\P\bigr)$.
A more complicated example follows from making sense of the
random field perspective on random functions as explained in
subsection \ref{ssec:grf}.

\subsubsection{Probability and Integration on Separable Banach Spaces}
\label{ssec:measures}

We now study probability and integration on separable 
Banach spaces $\CB$; we let $\CB^*$ denote the dual space of bounded
linear functionals on $\CB$. The assumption of separability 
rules out some important function spaces like $L^\infty(D;\R)$, 
but is required in order for the
basic results of integration theory to hold. 
This is because, when considering
a non-separable Banach space $\CB$, it is not clear what the ``natural''
$\sigma$-algebra on $\CB$ is. One natural candidate is
the Borel $\sigma$-algebra,  denoted $\cBc(\CB)$,
namely the smallest $\sigma$-algebra 
containing all open sets; another is the cylindrical $\sigma$-algebra, 
namely the smallest $\sigma$-algebra for which all bounded linear 
functionals on $\CB$ are measurable.
For i.i.d.\ sequences, the analogues of these two
$\sigma$-algebras can be identified whereas, in the general setting,
the cylindrical $\sigma$-algebra can be strictly smaller than the
Borel $\sigma$-algebra. In the case of separable Banach spaces however,
both $\sigma$-algebras agree:

\begin{lemma}\label{prop:determinefinitedim}
Let $\CB$ be a separable Banach space and let $\mu$ and $\nu$ be two 
Borel probability measures on $\CB$. If 
$\ell^\sharp\mu = \ell^\sharp\nu$ for every $\ell \in \CB^*$, then $\mu = \nu$.
\end{lemma}

Thus, as for i.i.d.\ sequences, there is therefore a canonical notion of
measurability. Whenever we refer to (probability) measures on a separable Banach space
$\CB$ in the sequel, we really mean (probability) measures on $\bigl(\CB,\cBc(\CB)\bigr)$.

We now turn to the definition of integration with respect to
probability measures on $\CB$.
Given a (Borel) measurable function $f \colon \Omega \to \CB$
where $(\Omega, \CF, \P)$ is a standard probability space, we say that 
$f$ is integrable with respect to $\P$ 
if the map $\omega \mapsto \|f(\omega)\|$ belongs to 
$L^1_{\P}(\Omega;\R)$. (Note that this map is certainly Borel measurable since
the norm $\|\cdot\|\colon \CB \to \R$ is a continuous, and therefore
also Borel measurable, function.) Given such an integrable function $f$, 
we \textit{define} its Bochner integral by 
\begin{equation*}
\int f(\omega)\,\P(d\omega) = \lim_{n \to \infty}\int f_n(\omega)\,\P(d\omega)\;,
\end{equation*}
where $f_n$ is a sequence of simple functions, for which the
integral on the right-hand side may be defined in the usual way,
chosen such that 
$${\lim_{n\to\infty}}\int\|f_n(\omega) - f(\omega)\|\,\P(d\omega) = 0.$$ 
With this definition the value of the integral does 
not depend on the approximating sequence, it is
linear in $f$, and 
\begin{equation}\label{e:linearBochner}
\int \ell(f(\omega))\,\P(d\omega) = \ell\Bigl( \int f(\omega)\,\P(d\omega)\Bigr)\;,
\end{equation}
for every element $\ell$ in the dual space $\CB^*$.

Given a probability measure $\mu$ on a separable Banach space $\CB$, we now say that
$\mu$ has \textit{finite expectation} if the identity function $x\mapsto x$ is
integrable with respect to $\mu$. If this is the case, we define the
expectation of $\mu$ as
\begin{equation*}
\int_\CB x\,\mu(dx)\;,
\end{equation*}
where the integral is interpreted as a Bochner integral. 

Similarly, it is natural to say that $\mu$ has \textit{\as{finite variance}}\index{variance!finite} if the map
$x \mapsto \|x\|^2$ is integrable with respect to $\mu$. Regarding the 
covariance $C_\mu$ of $\mu$ itself, it is natural to define it as a bounded linear \as{operator}\index{covariance operator}
$C_\mu \colon \CB^* \to \CB$ with the property that
\begin{equation}\label{e:propCmu}
C_\mu \ell = \int_\CB x \ell(x)\,\mu(dx)\;,
\end{equation}
for every $\ell \in \CB^*$. At this stage however, it is not clear whether such
an operator $C_\mu$ always exists solely under the assumption that $\mu$ has
\as{finite variance}\index{variance!finite}. 
For any $x \in \CB$, we define the projection operator
$P_x\colon \CB^* \to \CB$ by
\begin{equation}
\label{eq:px}
P_x \ell = x\,\ell(x)\;,
\end{equation}
suggesting that we define 
\begin{equation}\label{e:defVar}
C_\mu := \int_\CB P_x\,\mu(dx)\;.
\end{equation}
The problem with this definition is that if we view the map $x \mapsto P_x$ as
a map taking values in the space $\CL(\CB^*,\CB)$ 
of bounded linear operators from $\CB^* \to \CB$
then, since this space is not separable in general, it is not clear a priori
whether \eqref{e:defVar} makes sense as a Bochner integral.
This suggests to define the subspace $B_\star(\CB)\subset \CL(\CB^*,\CB)$
given by the closure (in the usual operator norm) of the linear span of 
operators of the type $P_x$ given in \eqref{eq:px} for $x \in \CB$. We 
then have:

\begin{lemma}
\label{lem:BC}
If $\CB$ is separable, then $B_\star(\CB)$ is also separable.
Furthermore, $B_\star(\CB)$ consists of compact operators. 
\end{lemma}

This leads to the following corollary:

\begin{corollary}
\label{cor:var}
Assume that $\mu$ has finite variance so that 
the map $x \mapsto \|x\|^2$ is integrable with respect to $\mu$.
Then the \as{covariance operator}\index{covariance operator} $C_\mu$ defiend by \eqref{e:defVar}
exists as a Bochner integral in $B_\star(\CB)$.
\end{corollary}

\begin{remark}
Once the covariance is defined, the fact that
\eqref{e:propCmu} holds is then an immediate 
consequence of \eqref{e:linearBochner}.
In general, not every element $C \in B_\star(\CB)$ can be realised as the 
covariance of some probability measure. This is the case 
even if we impose the positivity 
condition $\ell(C\ell) \ge 0$, which  by \eqref{e:propCmu} 
is a condition satisfied by every \as{covariance operator.}\index{covariance operator}
For further insight into this issue, see Lemma \ref{prop:trace}
which characteritzes precisely the \as{covariance operators}\index{covariance operator} of
a Gaussian measure in separable Hilbert space.
$\quad\Box$
\end{remark}

Given any probability measure $\mu$ on $\CB$, we can define its 
{\em \as{Fourier transform}}\index{measure! Fourier transform}
$\hat \mu \colon \CB^* \to \mathbb{C}$ by
\begin{equation}\label{e:FT}
\hat \mu(\ell) := \int_\CB e^{i\ell(x)}\, \mu(dx)\;.
\end{equation}
For a Gaussian measure $\mu_0$ on $B$ with mean $a$ and 
\as{covariance operator}\index{covariance operator} $C$, 
it may be shown show that, for any $\ell\in\CB^*$,
the characteristic function is given by
\begin{align}\label{e:FTG}
\hat\mu_0(\ell)=\e^{i\ell(a)-\frac{1}{2}\ell(C\ell)}.
\end{align}
As a consequence of Lemma~\ref{prop:determinefinitedim}, 
it is almost immediate that a measure
is uniquely determined by its Fourier transform, and this is the content of the
following result.

\begin{lemma}\label{prop:FT}
Let $\mu$ and $\nu$ be any two probability measures on a separable Banach space $\CB$. If $\hat \mu(\ell) = \hat \nu(\ell)$ for
every $\ell \in \CB^*$, then $\mu = \nu$.
\end{lemma}

\subsubsection{Probability and Integration on Separable Hilbert Spaces}
\label{ssec:hilbert}

We will frequently be interested in the case where $\CB=\cH$
for $\bigl(\cH, \langle \cdot,\cdot \rangle,\|\cdot\|\bigr)$ 
some separable Hilbert space. Bochner integration
can then, of course, be defined as a special case of the preceding
development on separable Banach spaces. 
We make use of the Riesz representation theorem to identify 
$\cH$ with its dual and
$\cH\otimes \cH$ with a subspace of the space of linear operators on $\cH$.
The \as{covariance operator}\index{covariance operator}
of a measure $\mu$ on $\cH$ may then be viewed as a bounded linear
operator from $\cH$ into itself. The definition \eqref{e:propCmu} of
$C_\mu$ becomes
\begin{equation}
\label{e:propCmu22}
C_\mu \ell =\int_{\cH}\langle \ell,x \rangle x \,\mu(dx)\;,
\end{equation} 
for all $\ell \in \cH$
and \eqref{e:defVar} becomes
\begin{equation}
\label{e:defVar22}
C_\mu=\int_{\cH} x \otimes x\,\mu(dx)\;.
\end{equation}
Corollary \ref{cor:var} shows that we can indeed make sense of the
second formulation as a Bochner integral, provided that $\mu$ has
\as{finite variance}\index{variance!finite} in $\cH$. 

\subsubsection{Metrics on Probability Measures}
\label{ssec:metricpm}

When discussing well-posedness and approximation theory
for the \as{posterior}\index{posterior}
distribution, it is of interest to estimate
the distance between two probability measures and thus we
will be interested in metrics between probability measures.
In this subsection we introduce two useful metrics 
on measures:
the {\em \as{total variation distance}}\index{total
variation distance} and
the {\em \as{Hellinger distance}}\index{Hellinger distance}.
We discuss the relationships between the metrics
and indicate how they may be used to estimate differences
between expectations of random variables under two different
measures. We also discuss the 
{\em \as{Kullback-Leibler divergence}}\index{Kullback-Leibler divergence},
a useful distance measure which does not satisfy the axioms
of a metric, but which may be used to bound both the \as{Hellinger}
\index{Hellinger distance}
and \as{total variation distances}\index{total variation distance}, 
and which is also useful
in defining algorithms for finding the best approximation to
a given measure from within some restricted class of measures,
such as Gaussians. 

Assume that we have two probability measures $\mu$ and $\mu'$
on a separable Banach space denoted by $B$ (actually the considerations
here apply on a Polish space but we do not need this level of generality).
Assume that $\mu$ and $\mu'$ are both absolutely continuous with respect to
a common reference measure $\nu$, also defined on the same 
measure space. Such a measure always exists -- take $\nu=\frac12(\mu+\mu')$ 
for example. 
In the following, all integrals of real-valued functions
over $B$ are simply denoted by $\int$. The following
define two concepts of distance between $\mu$ and $\mu'$.
The resulting metrics that we define are independent
of the choice of this common reference measure.

\begin{definition} \label{def:tvd}
The {\em \as{total variation distance}}\index{total variation
distance} between
$\mu$ and $\mu'$ is
$$\dtv(\mu,\mu')=\frac12 \int \Bigl|
\frac{d\mu}{d\nu}-\frac{d\mu'}{d\nu}\Bigr|
d\nu.  \quad\Box$$
\end{definition}

In particular, if $\mu'$ is absolutely continuous with
respect to $\mu$ then
\begin{equation}
\label{eq:tvduse}
\dtv(\mu,\mu')=\frac12 \int \Bigl|1-\frac{d\mu'}{d\mu}\Bigr|
d\mu.
\end{equation}

\begin{definition} \label{def:hell}
The {\em \as{Hellinger distance}}\index{Hellinger distance} between
$\mu$ and $\mu'$ is
$$\dhh(\mu,\mu')=\sqrt{\frac12 \int \Bigl(
\sqrt\frac{d\mu}{d\nu}-\sqrt\frac{d\mu'}{d\nu}\Bigr)^2
d\nu}.\quad\Box$$
\end{definition}

In particular, if $\mu'$ is absolutely continuous with
respect to $\mu$ then
\begin{equation}
\label{eq:heluse}
\dhh(\mu,\mu')=\sqrt{\frac12 \int \Bigl(1-\sqrt\frac{d\mu'}{d\mu}\Bigr)^2
d\mu}.
\end{equation}

Note that the numerical constant ${1\over 2}$ appearing in both definitions is chosen in such a way as to ensure the bounds
$$0 \le \dtv(\mu,\mu') \le 1\;,\qquad 0 \le \dhh(\mu,\mu') \le 1\;.$$
In the case of the total variation inequality this is an immediate consequence of the triangle inequality,
combined with the fact that both $\mu$ and $\mu'$ are probability measures, so that
$\int{d\mu\over d\nu}\,d\nu = 1$ and similarly for $\mu'$. In the case of the 
\as{Hellinger distance}\index{Hellinger distance}, it 
follows by expanding the square and applying similar considerations.

The Hellinger and total variation distances are related
as follows, which shows in particular that they both generate the same
topology:

\begin{lemma} \label{l:tvhell}
The \as{total variation}\index{total variation diatance}
and \as{Hellinger} \index{Hellinger distance} metrics are related by the inequalities
$$\frac{1}{\sqrt{2}}\dtv(\mu,\mu') \le \dhh(\mu,\mu') \le
\dtv(\mu,\mu')^{\frac12}.$$ 
\end{lemma}

\begin{proof} We have
\begin{align*}
\dtv(\mu,\mu') & =\frac12 \int \Bigl|\sqrt{\frac{d\mu}{d\nu}}-
\sqrt{\frac{d\mu'}{d\nu}}\Bigr|
\Bigl|\sqrt{\frac{d\mu}{d\nu}}+
\sqrt{\frac{d\mu'}{d\nu}}\Bigr| d\nu\\
& \le \sqrt{\Bigl(\frac12 \int \Bigl(
\sqrt\frac{d\mu}{d\nu}-\sqrt\frac{d\mu'}{d\nu}\Bigr)^2
d\nu\Bigr)}
\sqrt{\Bigl(\frac12 \int \Bigl(
\sqrt\frac{d\mu}{d\nu}+\sqrt\frac{d\mu'}{d\nu}\Bigr)^2
d\nu\Bigr)}\\
& \le \sqrt{\Bigl(\frac12 \int \Bigl(
\sqrt\frac{d\mu}{d\nu}-\sqrt\frac{d\mu'}{d\nu}\Bigr)^2
d\nu\Bigr)}
\sqrt{\Bigl( \int \Bigl( 
\frac{d\mu}{d\nu}+\frac{d\mu'}{d\nu}\Bigr)
d\nu\Bigr)}\\
&= \sqrt{2} \dhh(\mu,\mu')
\end{align*}
as required for the first bound.

For the second bound note that, for any positive $a$ and $b$, one has the bound
$|\sqrt a-\sqrt b| \le \sqrt a+\sqrt b$. As a consequence, we have the bound
\begin{align*}
\dhh(\mu,\mu')^2 & \le \frac12 \int \Bigl|\sqrt{\frac{d\mu}{d\nu}}-
\sqrt{\frac{d\mu'}{d\nu}}\Bigr|
\Bigl|\sqrt{\frac{d\mu}{d\nu}}+
\sqrt{\frac{d\mu'}{d\nu}}\Bigr| d\nu\\
& =\frac12 \int \Bigl|\frac{d\mu}{d\nu}-\frac{d\mu'}{d\nu}\Bigr|
d\nu\\
&=\dtv(\mu,\mu')\;,
\end{align*} 
as required.
\end{proof}

\begin{example}
\label{ex:hell}
Consider two Gaussian densities on $\R$: $\cN(m_1,\sigma_1^2)$
and $\cN(m_2,\sigma_2^2)$.
The Hellinger distance\index{Hellinger distance} between them is given by
$$\dhh(\mu,\mu')^2=1-\sqrt{\exp\Bigl(-\frac{(m_1-m_2)^2}{2(\sigma_1^2
+\sigma_2^2)}\Bigr)\frac{2\sigma_1\sigma_2}{(\sigma_1^2+\sigma_2^2)}}.$$
To see this note that
$$\dhh(\mu,\mu')^2=1-{\frac{1}{(2\pi\sigma_1\sigma_2)^{\frac12}}}\int_{\R}\exp(-Q)dx$$
where
$$Q=\frac{1}{4\sigma_1^2}(x-m_1)^2+\frac{1}{4\sigma_2^2}(x-m_2)^2.$$
Define $\sigma^2$ by
$$\frac{1}{2\sigma^2}=\frac{1}{4\sigma_1^2}+\frac{1}{4\sigma_2^2}.$$
We change variable under the integral to $y$ given by
$$y=x-\frac{m_1+m_2}{2}$$
and note that then, by completing the square,
$$Q=\frac{1}{2\sigma^2}(y-m)^2+\frac{1}{4(\sigma_1^2+\sigma_2^2)}(m_2-m_1)^2$$
where $m$ does not appear in what follows and so we do not detail it.
Noting that the integral  is then a multiple of
a standard Gaussian $N(m,\sigma^2)$ gives the desired result.
In particular this calculation shows that
the Hellinger distance\index{Hellinger distance}
between two Gaussians on $\R$ tends to
zero if and only if the means and variances of the two Gaussians
approach one another. Furthermore, by the previous lemma, the
same is true for the 
\as{total variation distance}\index{total variation distance}.
$\quad\qed$
\end{example}

The preceding example generalizes to higher dimension and shows that,
for example, the \as{total variation}\index{total variation distance} and 
\as{Hellinger metrics}\index{Hellinger distance} cannot
metrize weak convergence of probability measures 
(as one can also show that 
convergence in \as{total variation metric}\index{total variation distance} 
implies strong convergence).  They are nonetheless
useful distance measures, for example between families of measures
which are mutually absolutely continuous. Furthermore, the \as{Hellinger 
distance}\index{Hellinger distance} is particularly useful
for estimating the difference between expectation
values of functions of random variables under
different measures. This is encapsulated in the following lemma:

\begin{lemma} \label{l:tvhell2}
Let  $\mu$ and $\mu'$ be two probability measures on a separable Banach 
space $X$. Assume also that $f:X \to E$, where 
$(E,\|\cdot\|)$ is a separable Banach space, 
is measurable and has second moments with respect to
both $\mu$ and $\mu'$. Then
$$\|\bbE^{\mu}f-\bbE^{\mu'}f\| \le 2\Bigl(\bbE^{\mu}\|f\|^2+
\bbE^{\mu'}\|f\|^2\Bigr)^{\frac12}\dhh(\mu,\mu').$$
Furthermore, if $E$ is
a separable Hilbert space and $f:X \to E$ as before has fourth moments, then
$$\|\bbE^{\mu}(f\otimes f)-\bbE^{\mu'}(f\otimes f)\| \le 2\Bigl(\bbE^{\mu}\|f\|^4+
\bbE^{\mu'}\|f\|^4\Bigr)^{\frac12}\dhh(\mu,\mu').$$
\end{lemma}

\begin{proof} 
Let $\nu$ be a reference probability measure as above. We then have the bound
\begin{align*}
\|\bbE^{\mu}f-\bbE^{\mu'}f\|
&\le \int \|f\| \Bigl|\frac{d\mu}{d\nu}-\frac{d\mu'}{d\nu}\Bigr|d\nu\\
& =  \int \Bigl(\frac{1}{\sqrt 2}\Bigl|\sqrt{\frac{d\mu}{d\nu}}-
\sqrt{\frac{d\mu'}{d\nu}}\Bigr|\Bigr)\Bigl(
\sqrt{2}\|f\|\Bigl|\sqrt{\frac{d\mu}{d\nu}}+
\sqrt{\frac{d\mu'}{d\nu}}\Bigr|\Bigr) d\nu\\
& \le \sqrt{\Bigl(\frac12 \int \Bigl(
\sqrt\frac{d\mu}{d\nu}-\sqrt\frac{d\mu'}{d\nu}\Bigr)^2
d\nu\Bigr)}
\sqrt{\Bigl(2 \int \|f\|^2\Bigl(
\sqrt\frac{d\mu}{d\nu}+\sqrt\frac{d\mu'}{d\nu}\Bigr)^2
d\nu\Bigr)}\\
& \le \sqrt{\Bigl(\frac12 \int \Bigl(
\sqrt\frac{d\mu}{d\nu}-\sqrt\frac{d\mu'}{d\nu}\Bigr)^2
d\nu\Bigr)}
\sqrt{\Bigl(4 \int \|f\|^2\Bigl( 
\frac{d\mu}{d\nu}+\frac{d\mu'}{d\nu}\Bigr)
d\nu\Bigr)}\\
&= 2\Bigl(\bbE^{\mu}\|f\|^2+
\bbE^{\mu'}\|f\|^2\Bigr)^{\frac12}\dhh(\mu,\mu')
\end{align*}
as required. 

The proof for $f\otimes f$ follows from the bound
\begin{align*}
\|\bbE^{\mu}(f\otimes f)-\bbE^{\mu'}(f\otimes f)\|
&=\sup_{\|h\|=1}
\|\bbE^{\mu}\langle f,h\rangle f-\bbE^{\mu'}\langle f,h\rangle f\|\\
&\le \int \|f\|^2 \Bigl|\frac{d\mu}{d\nu}-\frac{d\mu'}{d\nu}\Bigr|d\nu\;,
\end{align*}
and then arguing similarly to the first case
but with $\|f\|$ replaced by $\|f\|^2$.
\end{proof}

\begin{remark}
Note, in particular, that choosing $X=E$, and with $f$ chosen
to be the identity mapping, we deduce that the differences between the
mean (resp. \as{covariance operator}\index{covariance operator}) 
of two measures are bounded
above by their \as{Hellinger distance}\index{Hellinger distance}, 
provided that one has some
\textit{a priori} control on the second (resp. fourth) moments.
$\quad\Box$
\end{remark}

We now define a third widely used distance
concept for comparing two 
probability measures. Note, however, that it does
not give rise to a metric in the strict sense,
because it violates both
symmetry and the triangle inequality. 

\begin{definition} \label{defn:KL}
The \as{Kullback-Leibler divergence}\index{Kullback-Leibler
divergence} between
two measures $\mu'$ and $\mu$, with $\mu'$ absolutely
continuous with respect to $\mu$, is
$$\dkl(\mu'||\mu)=\int \frac{d\mu'}{d\mu}\log\Bigl( \frac{d\mu'}{d\mu}\Bigr)
d\mu.\quad\Box$$
\end{definition}

If $\mu$ is also absolutely continuous with respect to $\mu'$, so
that the two measures are equivalent, then
$$\dkl(\mu'||\mu)= -\int \log\Bigl( \frac{d\mu}{d\mu'}\Bigr)
d\mu'$$
and the two definitions coincide.

\begin{example}
\label{ex:kl}
Consider two Gaussian densities on $\R$: $\cN(m_1,\sigma_1^2)$
and $\cN(m_2,\sigma_2^2)$.
The Kullback-Leibler divergence between them is given by
$$\dkl(\mu_1||\mu_2)=\ln\Bigl(\frac{\sigma_2}{\sigma_1}\Bigr)+
\frac12\Bigl(\frac{\sigma_1^2}{\sigma_2^2}-1\Bigr)+\frac{(m_2-m_1)^2}{2\sigma_2^2}.$$
To see this note that
\begin{align*}
\dkl(\mu_1||\mu_2)&=\bbE^{\mu_1}\Bigl(\ln\sqrt\frac{\sigma_2^2}{\sigma_1^2}
+\frac{1}{2\sigma_2^2}|x-m_2|^2-\frac{1}{2\sigma_1^2}|x-m_1|^2\Bigr)\\
&=\ln\frac{\sigma_2}{\sigma_1}+\bbE^{\mu_1}\Bigl(\bigl(\frac{1}{2\sigma_2^2}-
\frac{1}{2\sigma_1^2}\bigr)|x-m_1|^2\Bigr)+\bbE^{\mu_1}
\frac{1}{2\sigma_2^2}\Bigl(|x-m_2|^2-|x-m_1|^2\Bigr)\\
&=\ln\frac{\sigma_2}{\sigma_1}+
\frac12\Bigl(\frac{\sigma_1^2}{\sigma_2^2}- 1\Bigr)+
\frac{1}{2\sigma_2^2}\bbE^{\mu_1}\Bigl(m_2^2-m_1^2+2x(m_1-m_2)\Bigr)\\
&=\ln\frac{\sigma_2}{\sigma_1}+
\frac12\Bigl(\frac{\sigma_1^2}{\sigma_2^2}- 1\Bigr)+
\frac{1}{2\sigma_2^2}(m_2-m_1)^2
\end{align*}
as required. \qed
\end{example}

As for \as{Hellinger distance}\index{Hellinger distance}, 
this example shows that two Gaussians on $\R$
approach one another in the Kullback-Leibler divergence if and only if
their means and variances approach one another. This generalizes
to higher dimensions.
The Kullback-Leibler divergence provides an upper bound
for the square of the \as{Hellinger distance}\index{Hellinger distance}
and for the square of the \as{total variation distance.}\index{total variation distance}

\begin{lemma} \label{lem:klbnd}
Assume that two measures $\mu$ and $\mu'$
are equivalent. Then the bounds
$$\dhh(\mu,\mu')^2 \le \frac12 \dkl(\mu||\mu')\;,\qquad 
\dtv(\mu,\mu')^2 \le \dkl(\mu||\mu')\;,$$ 
hold.
\end{lemma}

\begin{proof}
The second bound follows from the first by using
Lemma~\ref{l:tvhell}, thus it suffices to proof the
first.
In the following we use the fact that
$$x-1 \ge  \log(x)\qquad \forall x \ge 0\;,$$
so that
$$\sqrt{x}-1 \ge \frac12 \log(x)\qquad \forall x \ge 0\;.$$
This yields the bound
\begin{align*}
\dhh(\mu,\mu')^2 &=\frac12 \int \Bigl( \sqrt\frac{d\mu'}{d\mu}-1\Bigr)^2 d\mu
=\frac12\int \Bigl(\frac{d\mu'}{d\mu}+1-2\sqrt{\frac{d\mu'}{d\mu}}\Bigr)d\mu\\
&=\int \Bigl(1-\sqrt{\frac{d\mu'}{d\mu}}\Bigr)d\mu
 \le \frac12\int \Bigl(-\log\frac{d\mu'}{d\mu}\Bigr)d\mu\\
&=\frac12\dkl(\mu||\mu')\;,
\end{align*}
as required.
\end{proof}

\subsubsection{Kolmogorov Continuity Test}

The setting of \as{Kolmogorov's continuity test}\index{Kolmogorov!continuity criterion} is the following. We assume that we
are given a compact domain $D \subset \R^d$,
a complete separable metric space $\CX$,
as well as a collection 
of $\CX$-valued random variables $u: x \in D \mapsto \CX$. At this stage we assume
no regularity whatsoever on the parameter $x$: the distribution
of this collection of random variables is a measure $\mu_0$ on the space $\CX^D$ of
all functions from $D$ to $\CX$ endowed with the product $\sigma$-algebra.
Any consistent family of marginal distributions does yield such a measure by
\as{Kolmogorov's extension Theorem}~\ref{theo:extension} \index{Kolmogorov!extension theorem}. 
With these notations at hand, \as{Kolmogorov's continuity test}\index{Kolmogorov!continuity criterion} can be formulated 
as follows, and enables the extraction of regularity with respect to variation 
of $u(x)$ with respect to $x$.

\begin{theorem}[{\bf \as{Kolmogorov Continuity Test}}]\label{theo:Kolmogorov}
\index{Kolmogorov!continuity criterion}
Let $D$ and $u$ be as above and assume 
that there exist $p > 1$, $\alpha > 0$ and $K>0$ such that 
\begin{equation}
\bbE {\mathsf d}\bigl(u(x), u(y)\bigr)^p \le K|x-y|^{p \alpha + d}\;,\qquad \forall x,y \in D\;,
\end{equation}
where ${\mathsf d}$ denotes the distance function on $\CX$, and
$d$ the dimension of the compact domain $D$. Then, for every $\beta < \alpha$, 
there exists a unique measure $\mu$ on $\CC^{0,\beta}(D, \CX)$ 
such that the canonical process under $\mu$ has the same law as $u$.
\end{theorem}

We have here generalized the notion of H\"older spaces from 
subsection \ref{sssec:ch} to functions taking
values in a Polish space; such generalizations are discussed in
subsection \ref{sssec:other}. 
The notion of {\em canonical process} is defined in subsection \ref{sec:Ito}.

We will frequently use \as{Kolmogorov's continuity test}\index{Kolmogorov!continuity criterion} in the following
setting: we again assume that we
are given a compact domain $D \subset \R^d$,
and now a collection 
$u(x)$ of $\R^n$-valued random variables indexed by $x \in D$. 
We have the following:

\begin{corollary}\label{cor:Kolmogorov}
Assume that there exist $p > 1$, $\alpha > 0$ and $K>0$ 
such that 
\begin{equation*}
\bbE |u(x)-u(y)|^p \le K|x-y|^{p \alpha + d}\;,\qquad \forall x,y \in D\;.
\end{equation*}
Then, for every $\beta < \alpha$, 
there exists a unique measure $\mu$ on $\CC^{0,\beta}(D)$ 
such that the canonical process under $\mu$ has the same law as $u$.
\end{corollary}

\begin{remark}
Recall that $\CC^{0,\gamma'}(D) \subset \CC_0^{0,\gamma}(D)$ for all
$\gamma'>\gamma$ so that, since the interval $\beta < \alpha$ for this
theorem is open, we may interpret the result as  giving an equivalent
measure defined on a separable Banach space. 
\end{remark}

A very useful consequence of \as{Kolmogorov's continuity criterion}\index{Kolmogorov!continuity criterion} is the following
result.  The setting is to consider a \as{random function}\index{random function}
$f$ given by the random series 
\begin{equation}
\label{eq:rf}
u = \sum_{k\ge 0}  \xi_k \psi_k
\end{equation} 
where $\{\xi_k\}_{k \ge 0}$ is an \as{i.i.d.\ sequence}\index{i.i.d.\ sequence}
and the $\psi_k$ are real- or complex-valued
H\"older functions on bounded open  $D \subset \R^d$ satisfying, for some $\alpha \in (0,1]$,
\begin{equation}
\label{eq:holder}
|\psi_k(x)-\psi_k(y)| \le h(\alpha,\psi_k)|x-y|^{\alpha} \quad x,y \in D;
\end{equation}
of course if $\alpha=1$ the functions are Lipschitz.

\begin{corollary}\label{cor:Kol}
Let $\{\xi_k\}_{k \ge 0}$ be countably many centred i.i.d.\
random variables (real or complex) with bounded moments of all orders.
Moreover let $\{\psi_k\}_{k \ge 0}$ satisfy \eqref{eq:holder}. 
Suppose there is some $\delta\in(0,2)$ such that
\begin{equation}\label{e:critKol}
S_1:=\sum_{k\ge 0} \|\psi_k\|^2_{L^\infty}<\infty
\quad\mathrm{and}\quad
S_2:=\sum_{k\ge 0} \|\psi_k\|^{2-\delta}_{L^\infty}h(\alpha,\psi_k)^\delta<\infty\;.
\end{equation}
Then $u$ defined by \eqref{eq:rf} is almost surely finite
for every $x \in D$, and $u$ is H\"older continuous for 
every H\"older exponent smaller than $\alpha \delta / 2$.
\end{corollary}
\begin{proof} 
Let us denote by $\kappa_n(X)$ the $n$th cumulant of a random variable $X$.
The odd cumulants of centred random variables are zero. Furthermore,
using the fact that the cumulants of independent random variables simply add up and 
that the cumulants of $\xi_k$ are all finite by assumption, we obtain for $p \ge 1$ 
the bound
\begin{align*}
\bigl|\kappa_{2p} \bigl(u(x)-u(y)\bigr)\bigr|
&= \Bigl|\sum_{k\ge 0}\kappa_{2p}(\xi_k)\,\bigl(\psi_k(x)-\psi_k(y)\bigr)^{2p}\Bigr|\\
&\lesssim C_p \sum_{k\ge 0}\min\{2^{2p}\|\psi_k\|_{L^\infty}^{2p}, h(\alpha,\psi_k)^{2p}|x-y|^{2p\alpha}\}\\
&\lesssim C_p\sum_{k\ge 0} \|\psi_k\|_{L^\infty}^{(1-\frac{\delta}{2})2p}
 h(\alpha,\psi_k)^{2p.\frac{\delta}{2}}  |x-y|^{2p\alpha. \frac{\delta }{2}}\\ 
&\lesssim  C_p |x-y|^{p\alpha\delta}\;,
\end{align*}
with $C_p$ denoting positive constants depending on $p$ which can change from occurrence to occurrence, and where we used that $\min\{a,bx^2\} \le a^{1-\delta/2} b^{\delta/2} |x|^\delta$
for any $a,b\ge0$ and the finiteness of $S_2$. In a similar way, we obtain $\bigl|\kappa_{2p} u(x)\bigr| < \infty$ for every $p \ge 1$.
Since the random variables $u(x)$ are centred, all moments of even order $2p$, $p \ge 1$, 
can be expressed in
terms of homogeneous polynomials of the even cumulants of order upto $2p$, so that 
\begin{equation*}
\bbE |u(x)-u(y)|^{2p} \lesssim C_p|x-y|^{p\alpha\delta} \;,\qquad \bbE |u(x)|^{2p} < \infty\;,
\end{equation*}
uniformly over $x,y \in D$.
The almost sure boundedness on $L^{\infty}$ follows from the second bound.
The H\"older continuity claim follows from \as{Kolmogorov's continuity test} 
\index{Kolmogorov!continuity criterion}
in the form of Corollary \ref{cor:Kolmogorov}, after noting that $p\alpha\delta=2p\bigl(\frac12\alpha\delta-
\frac{d}{2p}\bigr)+d$ and choosing $p$ arbitrarily large.
\end{proof}

\begin{remark}
Note that \eqref{eq:rf} is simply a rewrite of \eqref{21},
with $\psi_0=m_0$, $\xi_0=1$ and $\psi_k=\gamma_k\phi_k$.
In the case where the $\xi_k$ are standard normal then
the $\psi_k$'s in Corollary~\ref{cor:Kol} form an orthonormal basis
of the Cameron-Martin space (see Definition \ref{def:CM})
of a Gaussian measure.  The criterion \eqref{e:critKol} then provides
an effective way of showing that the measure in question 
can be realised on a space of H\"older continuous functions.$\quad\Box$ 
\end{remark}

\subsection{Gaussian Measures}
\label{sec:GM}

\subsubsection{Separable Banach Space Setting}
\label{sec:GMB}

We start with the definition of a Gaussian measure on a separable
Banach space $B$. There is no equivalent to Lebesgue measure
in infinite dimensions (as it could not be $\sigma$-additive), and so
we cannot define a Gaussian measure by prescribing the form of its density. 
However, note that Gaussian measures on $\R^n$ can be characterised by
prescribing that the projections of the measure onto any one-dimensional 
subspace of $\R^n$ are all Gaussian. This is a property that can readily 
be generalised to infinite-dimensional spaces:

\begin{definition}
A \index{Gaussian measure}\textit{\as{Gaussian probability measure}} $\mu$ on a 
separable Banach space $\CB$ is a Borel measure such that
$\ell^\sharp\mu$ is a Gaussian probability measure on $\R$ for every continuous linear functional $\ell \colon \CB \to \R$. (Here, Dirac measures 
are considered to be Gaussian measures with zero variance.)
The measure is said to be \textit{centred} if $\ell^\sharp\mu$ has mean
zero for every $\ell$. $\quad\Box$
\end{definition}

This is a reasonable definition since, provided
that $\CB$ is separable, the one-dimensional projections 
of any probability measure carry sufficient information 
to characterise it -- see Lemma \ref{prop:determinefinitedim}.  
We now state an important result which controls
the tails of Gaussian distributions: 

\index{Fernique theorem}
\begin{theorem}[{\bf \as{Fernique}}]\label{theo:Fernique}
Let $\mu$ be a Gaussian probability measure on a separable Banach space $\CB$. Then, there exists $\alpha > 0$ such that $\int_\CB \exp(\alpha \|x\|^2)\,\mu(dx) < \infty$.
\end{theorem}

As a consequence of Fernique's theorem and the Corollary \ref{cor:var},
every Gaussian measure $\mu$ admits a compact \as{covariance operator}\index{covariance operator} $C_\mu$
given by \eqref{e:defVar}, because the second moment is bounded. In
fact the techniques used to prove the Fernique theorem show that, if 
$M = \int_\CB \|x\|\,\mu(dx)$, then there is a global constant $K>0$ such that 
\begin{equation}
\int_\CB \|x\|^{2n}\,\mu(dx) \le n! K \alpha^{-n} M^{2n}.
\label{eq:Gmoment}
\end{equation}

Since the \as{covariance operator}\index{covariance operator}, and hence the mean, exist for a Gaussian measure,
and since they may be shown to characterize the measure completely, we write
$N(m,C)$ for a Gaussian with mean $m$ and covariance operator $C$.

Measures in infinite dimensional spaces are typically mutually
singular. Furthermore, two Gaussian measures are either mutually singular
or equivalent (mutually absolutely continuous). The \as{Cameron-Martin 
space}\index{Cameron-Martin!space} plays a key role in characterizing
whether or not two Gaussians are equivalent. 

\begin{definition}\label{def:CM}
The \textit{\as{Cameron-Martin space}}\index{Cameron-Martin!space} 
${\mathcal H}_{\mu}$ of measure $\mu$ on a separable Banach
space $\CB$ is the completion of the linear subspace $\HH_\mu \subset \CB$ 
defined by
\begin{equation}
\HH_\mu = \{ h \in \CB \,:\, \exists\, h^* \in \CB^*\;\text{with}\; h = C_\mu h^*\}\;,
\end{equation}
under the \index{Cameron-Martin!norm}norm $\|h\|_\mu^2 = \langle h,h \rangle_\mu = h^*(C_\mu h^*)$.
It is a Hilbert space when endowed with the scalar product $\langle h,k \rangle_\mu = h^*(C_\mu k^*) = h^*(k) = k^*(h)$.
\end{definition}

The Cameron-Martin space is actually independent of the space $\CB$ in the
sense that, although we may view the measure as living on a range of
separable Hilbert or Banach spaces, the Cameron-Martin space will be the
same in all cases. The space characterizes exactly the directions in
which a centred Gaussian measure may be shifted to obtain an equivalent
Gaussian measure:

\index{Cameron-Martin!theorem}
\begin{theorem}[{\bf \as{Cameron-Martin}}]\label{theo:CM}
For $h \in \CB$, define the map $T_h \colon \CB \to \CB$ by $T_h(x) = x+h$. Then, the measure $T_h^\sharp\mu$ is absolutely continuous
with respect to $\mu$ if and only if $h \in \CH_\mu$. Furthermore, in the latter case,  its Radon-Nikodym derivative
is given by
\begin{equation*}
{d T_h^\sharp\mu \over d\mu}(u) =  \exp \bigl(h^*(u) - \hf  \|h\|_\mu^2\bigr)
\end{equation*}
where $h = C_\mu h^*$.
\end{theorem}

Thus this theorem characterizes the Radon-Nikodym derivative of the
measure $N(h,C_{\mu})$ with respect to the measure $N(0,C_{\mu})$.
Below, in the Hilbert space setting, we also consider changes in the
\as{covariance operator}\index{covariance operator} 
which lead to equivalent Gaussian measures.
However, before moving to the Hilbert space setting, we conclude this
subsection with several useful observations concerning Gaussians
on separable Banach spaces. The {\em topological support}
of measure $\mu$ on the separable Banach space $B$
is the set of all $u\in B$ such that any neighborhood 
of $u$ has a positive measure.

\begin{theorem}
\label{theo:balls}
The topological support of a centred Gaussian measure $\mu$ on $B$ is the
closure of the Cameron-Martin space in $B$. Furthemore the Cameron-Martin
space is dense in $X$. Therefore all balls in $B$ have positive $\mu$-measure.
\end{theorem}

Since the Cameron-Martin space of Gaussian measure $\mu$
is independent of the space on which we
view the measure as living, this following useful theorem shows that the
unit ball in the Cameron-Martin space is compact in any separable
Banach space $X$ for which $\mu(X)=1:$

\begin{theorem}
\label{t:CMcompact}
The closed unit ball in the Cameron-Martin space $\cH_{\mu}$ is compactly 
embedded into the separable Banach space $\CB$.
\end{theorem}

In the setting of Gaussian measures on a separable Banach space, all balls
have positive probability. 
The Cameron-Martin norm is useful in the characterization of small-ball
properties of Gaussians.
Let $B^\delta(z)$ denote a ball of radius $\delta$ in $\CB$ centred at
a point $z \in \cH_{\mu}$.

\begin{theorem}
\label{t:smallball}
The ratio of small ball probabilities under Gaussian measure $\mu$ satisfy
\begin{equation*}
    \lim_{\delta\to 0}
    \frac{\mu\bigl( B^\delta(z_1) \bigr)}
    {\mu\bigl(B^\delta(z_2)\bigr)}
    = \exp\left( \frac12\|z_2\|_\mu^2-\frac12\|z_1\|_\mu^2 \right).
\end{equation*}
\end{theorem}

\begin{example}
\label{ex:sb}
Let $\mu$ denote the Gaussian measure $N(0,K)$ on $\R^n$ with $K$
positive definite. Then Theorem \ref{t:smallball} is the statement that
\begin{equation*}
    \lim_{\delta\to 0}
    \frac{\mu\bigl( B^\delta(z_1) \bigr)}
    {\mu\bigl(B^\delta(z_2)\bigr)}
    = \exp\left( \frac12|K^{-\frac12} z_2|^2-\frac12|K^{-\frac12} z_1|^2 \right)
\end{equation*}
which follows directly from the fact that the Gaussian measure
at point $z \in \R^n$ has Lebesgue density proportional to
$\exp\left(-\frac12|K^{-\frac12} z|^2\right)$ and the fact that the
Lebesgue density is a continuous function. $\quad\Box$
\end{example}

\subsubsection{Separable Hilbert Space Setting}
\label{sec:GMH}

In these notes our approach is primarily based on defining Gaussian
measures on Hilbert space; the Banach spaces which are of full measure
under the Gaussian are then determined via the \as{Kolmogorov continuity
theorem}\index{Kolmogorov!continuity criterion}. In this subsection we develop the theory of Gaussian measures
in greater detail within the Hilbert space setting. Throughout
$\bigl(\cH, \langle \cdot,\cdot \rangle,\|\cdot\|\bigr)$ 
denotes the separable Hilbert space on which the Gaussian is
constructed. 
Actually, in this Hilbert space setting
the \as{covariance operator}\index{covariance operator} $C_\mu$ has 
considerably more structure than just the boundedness implied by
\eqref{eq:Gmoment}: it is trace-class and hence necessarily compact on $\CH$: 

\begin{lemma}\label{prop:trace}
A Gaussian measure $\mu$ on a a separable Hilbert space $\CH$, 
has \as{covariance operator}\index{covariance operator} $C_\mu \colon \CH \to \CH$ 
which is trace class and satisfies 
\begin{equation}\label{e:expnorm}
\int_\CH \|x\|^2\,\mu(dx) = \tr C_\mu\;.
\end{equation}
Conversely, for every positive trace class symmetric operator $K\colon \CH \to \CH$, there exists a Gaussian measure $\mu$ on $\CH$ such
that $C_\mu = K$.
\end{lemma}

Since the \as{covariance operator}\index{covariance operator} 
$C_{\mu}:\cH \to \cH$ of a Gaussian on $\CH$
is a compact operator it follows that if 
operator $C_{\mu}:\cH \to \cH$
has an inverse then it will be a densely-defined unbounded operator on $\cH$;
we call this the {\em \as{precision operator}}\index{precision operator}.
Both the covariance and the precision operators are self-adjoint on appropriate
domains, and fractional powers of them may be defined via the spectral theorem.

\begin{theorem}[{\bf \as{Cameron-Martin Space}\index{Cameron-Martin!space} on Hilbert Space}]
\label{t:cmh}
Let $\mu$ be a Gaussian measure on a Hilbert space $\CH$ with 
strictly positive \as{covariance operator}\index{covariance operator} $K$. Then the the Cameron-Martin
space $\CH_\mu$ consists of the image of $\CH$ under $K^{1/2}$
and the Cameron-Martin norm is given by  
$\|h\|_\mu^2=\|K^{-\frac12} h\|^2$.
\end{theorem}

\begin{example} Consider two Gaussian measures $\mu_i$
on $\cH=L^2(J), J=(0,1)$ both with \as{precision operator}\index{precision operator} 
$L=-\frac{d^2}{dx^2}$ where $\CD(L)=H^1_0(J) \cap H^2(J)$.
(Informally $-L$ is the Laplacian on $J$ with homogeneous
Dirichlet boundary conditions.) 
Assume that $\mu_1 \sim \cN(m,C)$
and $\mu_2 \sim \cN(0,C)$. 
Then $\cH_{\mu_i}$ is the image of $\cH$ under $\cC^{\frac12}$
which is the space $=H^1_0(J)$.
It follows that the measures are equivalent
if and only if $m \in H^1_0(J)$.
If this condition is satisfied then,
from Theorem \ref{t:cmh}, 
the Radon-Nikodym derivative between the two measures is given by
$$\frac{d\mu_1}{d\mu_2}(x) =\exp\Bigl(\langle m,x\rangle_{H^1_0}-
\frac12\|m\|_{H^1_0}^2\Bigr).\quad\Box$$
\end{example}

We now turn to the Feldman-H\'ajek\index{Feldman-H\'ajek Theorem} 
theorem in the Hilbert Space setting.
Let $\{\varphi_j\}_{j=1}^{\infty}$ denote an orthonormal basis for
$\CH$. Then the \textit{\as{Hilbert-Schmidt norm}}\index{Hilbert-Schmidt norm}
of a linear operator $L:\CH \to \CH$ is defined by
$$\|L\|_{\tiny {\rm HS}}^2:=\sum_{j=1}^{\infty} \|L\varphi_j\|^2.$$
The value of the norm is, in fact, independent of the choice of
orthonormal basis.
In the finite dimensional setting the norm is known 
as the \textit{\as{Frobenius norm}}\index{Frobenius norm}.

\begin{theorem}
[{\bf \as{Feldman-H\'ajek} on Hilbert Space}\index{Feldman-H\'ajek Theorem}]
\label{t:GAC}
Let $\mu_i$ with $i=1,2$ be two centred Gaussian measures on some
fixed Hilbert space $\CH$ with means $m_i$ and
strictly positive \as{covariance operators}\index{covariance operator} $C_i$.
Then the following hold:
\begin{enumerate}
\item $\mu_1$ and $\mu_2$ are either singular or equivalent.
\item The measures $\mu_1$ and $\mu_2$ are equivalent Gaussian measures 
if and only if:
\begin{enumerate}
\item The images of $\CH$ under $C_i^{\frac12}$ coincide for $i=1,2$,
and we denote this common image space by $E$;
\item $m_1-m_2 \in E$;
\item  $\|(C_1^{-1/2}C_2^{1/2})(C_1^{-1/2}C_2^{1/2})^* - I\|_{\tiny {\rm HS}} < \infty$. 
\end{enumerate}
\end{enumerate}
\end{theorem}

\begin{remark}
\label{rem:hs} The final condition may be replaced by the condition that
$$\|(C_1^{1/2}C_2^{-1/2})(C_1^{1/2}C_2^{-1/2})^* - I\|_{\tiny {\rm HS}} < \infty$$
and the theorem remains true; this formulation is sometimes useful.$\quad\Box$
\end{remark}

\begin{example} Consider two mean-zero Gaussian measures $\mu_i$
on $\cH=L^2(J), J=(0,1)$ with \as{precision operators}\index{precision operator}
$L_1=-\frac{d^2}{dx^2}+I$ and
$L_2=-\frac{d^2}{dx^2}$ respectively, both with
domain $H^1_0(J) \cap H^2(J)$.
The operators $L_1,L_2$ share the same eigenfunctions
$$\phi_k(x)=\sqrt 2 \sin \left(k\pi x\right)$$
and have eigenvalues
$$\lambda_k(1)=\lambda_k(2)+1, \quad \lambda_k(2)=k^2 \pi^2,$$
respectively.
Thus $\mu_1 \sim \cN(0,C_1)$ and $\mu_2 \sim \cN(0,C_2)$
where, in the basis of eigenfunctions, $C_1$ and $C_2$ are
diagonal with eigenvalues
$$\frac{1}{k^2\pi^2 +1}, \quad \frac{1}{k^2\pi^2}$$
respectively.
We have, for $h_k=\langle h,\phi_k\rangle$,
$$\frac{\pi^2}{\pi^2+1} \le \frac{\langle h,\cC_1 h \rangle}{\langle h,\cC_2 h \rangle}
=\frac{\sum_{k \in \bbZ^+} (1+k^2\pi^2)^{-1}h_k^2}{\sum_{k \in \bbZ^+}
(k\pi)^{-2}h_k^2} \le 1.$$
From this it follows that the Cameron-Martin spaces of the two measures
coincide, and are equal to $H^1_0(J)$.
Notice that
$$T=C^{-\frac12}_1C_2C^{-\frac12}_1-I$$
is diagonalized in the same basis as the $C_i$ and has eigenvalues
$$\frac{1}{k^2\pi^2}.$$
These are square summable and so by Theorem \ref{t:GAC}
the two measures are absolutely continuous with respect to one another.
$\quad\Box$
\end{example}

\subsection{Wiener Processes in Infinite Dimensional Spaces}
\label{sec:Ito}

Central to the theory of stochastic PDEs is the notion of a \index{Wiener!process (cylindrical)}\textit{\as{cylindrical Wiener process}},
which can be thought of as an infinite-dimensional generalisation of a standard $n$-dimensional Wiener process. 
This leads to the notion of the $A$-\as{Wiener process}\index{Wiener!process ($A$-)}
for certain classes of operators $A$.  
Before we proceed to the definition and construction of such Wiener processes
in separable Hilbert spaces, 
let us recall a few basic facts about stochastic processes in general.

In general, a stochastic process $u$ over a probability space $(\Omega,{\cal F},\P)$ and taking values in a separable Hilbert space $\CH$
is nothing but a collection $\{u(t)\}$
of $\CH$-valued random variables indexed by time $t \in \R$ (or taking values in some subset of $\R$). 
By \as{Kolmogorov's Extension Theorem}\index{Kolmogorov!extension theorem} \ref{theo:extension}, 
we can also view this as a map 
$u \colon \Omega \to \CH^\R$, where $\CH^\R$ is endowed
with the product sigma-algebra.
A notable special case which will be of interest here is
the case where the probability space is taken to be 
$\Omega = \CC([0,T], \CH)$ (or some other space of $\CH$-valued continuous
functions) endowed with some Gaussian measure $\P$ and where the process $X$ is given by
\begin{equation*}
u(t)(\omega) = \omega(t)\;,\qquad \omega \in \Omega\;.
\end{equation*}
In this case, $u$ is called the canonical process on $\Omega$. 

The usual (one-dimensional) \as{Wiener process}\index{Wiener!process} is a real-valued centred Gaussian process $B(t)$ such that
$B(0) = 0$ and $\bbE |B(t) - B(s)|^2 = |t-s|$ for any pair of times $s,t$. From our point of view,
the Wiener process on any finite time interval $I$ can always be realised as the canonical process for the Gaussian measure on $\CC(I,\R)$ with
\as{covariance function}\index{covariance function}
$C(s,t) = s \wedge t = \min\{s,t\}$. Note that such a 
measure exists by the \as{Kolmogorov continuity test}\index{Kolmogorov!continuity criterion}, and
Corollary ~\ref{cor:Kolmogorov} in particular.

The standard $n$-dimensional Wiener process $B(t)$
is simply given by $n$ independent copies of a standard one-dimensional
Wiener process $\{\beta_j\}_{j=1}^n$, so that its covariance is given by
\begin{equation*}
\bbE \beta_i(s) \beta_j(t) =  (s\wedge t)\delta_{i,j}\;.
\end{equation*}
In other words, if $u$ and $v$ are any two elements in $\R^n$, we have
\begin{equation*}
\bbE \langle u,B(s) \rangle \langle B(t),v \rangle =  (s\wedge t)\langle u,v \rangle\;.
\end{equation*}
This is the characterisation that we will now extend to an arbitrary separable Hilbert space $\CH$. One natural way of
constructing such an extension is to fix an orthonormal basis $\{e_n\}_{n \ge 1}$ of $\CH$ and a countable
collection $\{\beta_j\}_{j=1}^{\infty}$ 
of independent one-dimensional Wiener processes, and to set
\begin{equation}\label{e:defCylLoose}
B(t) := \sum_{n=1}^{\infty} \beta_n(t)\,e_n\;.
\end{equation}
If we define 
\begin{equation*}
B^N(t) := \sum_{n=1}^{N} \beta_n(t)\,e_n\;
\end{equation*}
then clearly $\bbE \|B^N(t)\|_{\CH}^2=tN$ and so the series will
not converge in $\CH$ for fixed $t>0.$ However the expression
\eqref{e:defCylLoose} is nonetheless the right way to
think of a cylindrical Wiener process on $\CH$; indeed for fixed $t>0$
the truncated series for $B^N$ will converge in a larger space containing
$\CH$. We define the following scale of Hilbert subspaces, 
for $r>0$, by
$$\mathcal{X}^r=\{u\in\mathcal{H}\big|\sum_{j=1}^{\infty} j^{2r}|\langle u,\phi_j\rangle|^2<\infty\}$$
and then extend to superspaces $r<0$ by duality. We use 
$\|\cdot\|_r$ to denote the norm induced by the inner-product 
$$\langle u,v \rangle_r=\sum_{j=1}^{\infty} j^{2r} u_j v_j$$
for $u_j=\langle u,\phi_j \rangle$ and $v_j=\langle v,\phi_j \rangle$.
A simple argument, similar to that used to prove Theorem \ref{t:2.4b},
shows that $\{B^N(t)\}$ is, for fixed $t>0$, Cauchy in $\mathcal{X}^r$
for any $r<-\frac12.$ In fact it is possible to construct a
stochastic process as the limit of the truncated series, living on
the space $\CC([0,\infty),\mathcal{X}^r)$ for any $r<-\frac12$, by
the \as{Kolmogorov Continuity Theorem} \ref{theo:Kolmogorov}
\index{Kolmogorov!continuity criterion}
in the setting where $D=[0,T]$ and $X=\mathcal{X}^r.$
We give details in the more general setting that follows.

Building on the preceding we now discuss construction of a
$\cC$-Wiener process $W$, using
the finite dimensional case described in Remark \ref{rem:cwp} to guide us.
Here $\cC:\CH \to \CH$ is assumed to be trace-class with eigenvalues 
$\gamma_j^2$. Consider the cylindrical Wiener
process given by 
$$B(t)=\sum_{j=1}^\infty \beta_j e_j,$$ 
where $\{\beta_j\}_{j=1}^\infty$ is an i.i.d.\ family of unit Brownian motions
on $\R$ with $\beta_j\in \CC([0,\infty);\R)$. We note that
\begin{equation}
\label{eq:BMH}
\bbE|\beta_j(t)-\beta_j(s)|^2 =|t-s|.
\end{equation}

Since
$\sqrt{\mathcal{C}}e_j=\gamma_j e_j$, the $\cC$-Wiener process
$W=\sqrt{\mathcal{C}}B$ is then
\begin{equation}
W(t)=\sum_{j=1}^\infty\gamma_j\beta_j(t)e_j.
\label{eq:cwp2}
\end{equation}
The following formal
calculation gives insight into the properties of $W$: 
\begin{align*}
\bbE W(t)\otimes W(s)
&=\bbE\Bigl(\sum_{j=1}^\infty \sum_{k=1}^\infty \gamma_j\gamma_k
\beta_j(t) \beta_k(s) e_j\otimes e_k\Bigr)\\
&=\Bigl(\sum_{j=1}^\infty \sum_{k=1}^\infty \gamma_j\gamma_k
\bbE\bigl(\beta_j(t) \beta_k(t)\bigr)e_j\otimes e_k\Bigr)\\
&=\Bigl(\sum_{j=1}^\infty \sum_{k=1}^\infty \gamma_j\gamma_k
\delta_{jk}(t \wedge s) e_j\otimes e_k\Bigr)\\
& =\sum_{j=1}^\infty
\Bigl(\gamma_j^2\phi_j\otimes\phi_j\Bigr) t \wedge s\\
&={\mathcal C}\, (t \wedge s).
\end{align*}
Thus the process has the covariance structure of Brownian motion in time,
and \as{covariance operator}\index{covariance operator} ${\mathcal C}$ in space.
Hence the name ${\mathcal C}$-\as{Wiener process}\index{Wiener!process}.

Assume now that the sequence $\gamma=\{\gamma_j\}_{j=1}^{\infty}$ is such
that $\sum_{j=1}^{\infty} j^{2r}\gamma_j^2=M<\infty$ for some $r \in \R.$
For fixed $t$ it is then possible to construct a
stochastic process as the limit of the truncated series
\begin{equation*}
W^N(t)=\sum_{j=1}^N\gamma_j\beta_j(t)e_j,
\end{equation*}
by means of a Cauchy sequence argument
in $L^2_{\P}(\Omega;\mathcal{X}^r)$. 
Similarly $W(t)-W(s)$ may be defined for any $t,s$.
We may then also discuss the regularity of this process in time.
Together equations \eqref{eq:BMH},\eqref{eq:cwp2} give
$\bbE\|W(t)-W(s)\|_r^2=M^2|t-s|.$
It follows that
$\bbE\|W(t)-W(s)\|_r \le M|t-s|^{\frac12}.$ Furthermore, since $W(t)-W(s)$ is
Gaussian, we have by \eqref{eq:Gmoment} that
$\bbE\|W(t)-W(s)\|_r^{2q} \le K_q|t-s|^{q}.$
Applying the \as{Kolmogorov continuity test}\index{Kolmogorov!continuity criterion} 
of Theorem \ref{theo:Kolmogorov} 
then demonstrates that the process given by \eqref{eq:cwp2} may be
viewed as an element of the space $C^{0,\alpha}([0,T];\mathcal{X}^r)$ for
any $\alpha<\frac12.$
Similar arguments may be used to study the cylindrical Wiener process,
showing that it lives in $C^{0,\alpha}([0,T];\mathcal{X}^r)$
for $\alpha<\frac12$ and $r<-\frac12.$

\subsection{Bibliographical Notes}

\begin{itemize}

\item Subsection \ref{ssec:ifs} introduces various Banach
and Hilbert spaces, as well as the notion of 
separablity;  see \cite{Yos95}. 
In the context of PDEs, see \cite{evans} and \cite{robinson1},
for all of the function spaces defined in 
subsections \ref{sssec:elp}--\ref{sssec:sob}; 
Sobolev spaces are developed in detail in \cite{AF03}.
The nonseparability of the H\"older spaces $C^{0,\beta}$ and
the separability of $C^{0,\beta}_0$ is discussed in \cite{hairernotes}.
For asymptotics of the eigenvalues of the Laplacian operator see
\cite[Chapter 11]{Str08}.
For discussion of the more general spaces of $E$-valued functions 
over a measure space $({\mathcal M},\nu)$ we refer the reader to \cite{Yos95}.
Subsection \ref{ssec:iise} concerns Sobolev embedding theorems,
building rather explicitly on the case of periodic functions.
The corresponding embedding results in domains 
with more general boundary conditions or even on more general manifolds
or unbounded domains, we refer to the 
comprehensive series of monographs \cite{Tri83,Tri92,Tri06}.
The interpolation inequality of \eqref{eq:interp} and Lemma \ref{l:intA}
may be found in \cite{robinson1}; see also Proposition 6.10
and Corollary 6.11 of \cite{hairernotes}. The proof of Theorem \ref{t:sob}
closely follows that given in \cite[Theorem 6.16]{hairernotes}, and is
a slight generalization to the \as{Hilbert scale}\index{Hilbert space} 
setting used here.

\item Subsection~\ref{ssec:pmii} briefly introduces the theory of
probability measures on infinite dimensional spaces. 
We refer to the extensive treatise by Bogachev \cite{BogMT},
and to the much shorter but more
readily accessible book by Billingsley \cite{Bill68}, 
for more details.
The subject of independent sequences of random variables, as overviewed
in subsection \ref{ssec:iid} in the i.i.d.\ case, 
is discussed in detail in \cite[section 1.5.1]{Giu06}.
The \as{Kolmogorov Extension Theorem} \ref{theo:extension}\index{Kolmogorov!extension theorem} is proved in numerous
texts in the setting where $X=\R$ \cite{Oks98}; since any Polish space
is isomorphic to $\R$ it may be stated as it is here.
Proofs of Lemmas \ref{prop:determinefinitedim} and \ref{prop:FT}
may be found in \cite{hairernotes}, where they appear as Proposition 3.6
and Propostion 3.9 respectively.
For (\ref{e:FTG}) see \cite[Chapter 2]{DapZab92}.
In subsection \ref{ssec:measures} we introduce the Bochner
integral; see  \cite{Bochner,Hildebrandt} for further details. 
Lemma \ref{lem:BC} and the resulting Corollary \ref{cor:var}
are stated and proved in \cite{bog98}. 
The topic of metrics on probability measures, introduced
in subsection \ref{ssec:metricpm} 
is overviewed in \cite{GS02}, where detailed references to the
literature on the subject may also be found; the second
inequality in Lemma \ref{lem:klbnd} is often termed the \textit{Pinsker
inequality} and can be found in \cite{CK11}.
Note that the choice of normalization constants
in the definitions of the \as{total variation}\index{total variation distance} 
and \as{Hellinger metrics}\index{Hellinger distance}
differs in the literature.
For a more detailed account of material on weak convergence
of probability measures we refer, for example,
to \cite{Bill68,BogMT,Vil03}.
A proof of the \as{Kolmogorov continuity 
test}\index{Kolmogorov!continuity criterion}
as stated in Theorem \ref{theo:Kolmogorov}
can be found in \cite[p.~26]{RevYor} for simple
case of $D$ an interval and $X$ a separable Banach space;
the generalization given here may be found in a forthcoming
uptodate version of \cite{hairernotes}.

\item The subject of Gaussian measures, as introduced in subsection \ref{sec:GM},
is comprehensively studied in \cite{bog98} in the setting of locally convex
topological spaces, including separable Banach spaces as a special case. See also 
\cite{Lif95} which is concerned with Gaussian random functions.
The Fernique Theorem \ref{theo:Fernique} 
is proved in \cite{Fernique} and the reader
is directed to \cite{hairernotes} for a very clear exposition. 
In \theo{theo:Fernique} it is possible to take for $\alpha$ any value smaller than ${1/(2\|C_\mu\|)}$ and this value is sharp: see \cite[Thm~4.1]{Ledoux}.
See \cite{bog98,Lif95} for more details on the Cameron-Martin space, and proof
of Theorem \ref{theo:CM}.
Theorem \ref{theo:balls} follows from Theorem 3.6.1 and Corollary 3.5.8 of
\cite{bog98}: Theorem 3.6.1 shows that the topological support is the closure
of the Cameron-Martin space in $B$ and Corollary 3.5.8 shows that the
Cameron-Martin space is dense in $B$.
The \index{reproducing kernel}\textit{\as{reproducing kernel} Hilbert space} 
for $\mu$ (or just \textit{reproducing kernel} for short) appears widely
in the literature and is isomorphic
to the Cameron-Martin space in a natural way. 
There is considerable confusion between the two  as a result.
We retain in these notes the terminology from \cite{bog98}, but the reader
should  keep in mind that there are authors who use a slightly different terminology.
Theorem \ref{t:smallball} as stated is a consequence of Proposition 3 in
section 18 in \cite{Lif95}. 
Turning now to the Hilbert space setting we note that Lemma \ref{prop:trace}
is proved as Proposition 3.15, and Theorem \ref{t:cmh} appears as Exercise 3.34,
in \cite{hairernotes}.
See \cite{bog98,DapZab92,MR0488264} for alternative developments of
the Cameron-Martin\index{Cameron-Martin!Theorem} and Feldman-H\'ajek theorems.\index{Feldman-H\'ajek Theorem}
The original statement of the \as{Feldman-H\'ajek} Theorem \ref{t:GAC}\index{Feldman-H\'ajek Theorem} can be found in \cite{Feldman,Hajek}. Our statement 
of Theorem \ref{t:GAC} mirrors
Theorem 2.23 of \cite{DapZab92} and Remark \ref{rem:hs} is Lemma 6.3.1(ii)
of \cite{bog98}. 
Note that we have not stated a result analogous to Theorem \ref{theo:CM}
in the case where of two equivalent Gaussian measures with differing
covariances. Such a result can be stated, but is technically
complicated in general because the ratio of normalizations constants
of approximating finite dimensional measures can blow-up as the
limiting infinite dimensional Radon-Nikodym derivative is attained;
see Corollary 6.4.11 in \cite{bog98}.

\item Subsection \ref{sec:Ito} contains a discussion of cylindrical
and $\cC$-Wiener processes. The development is given in more detail in
section 3.4 of \cite{hairernotes}, and in section 4.3 of \cite{DapZab92}.

\end{itemize}

\noindent{\bf Acknowledgements}
The authors are indebted to Martin Hairer for help in the development
of these notes, and in particular for considerable help in
structuring the Appendix, for the proof of Theorem \ref{t:sob}
(which is a slight generalization to \as{Hilbert scales}\index{Hilbert scale}
of Theorem 6.16 in \cite{hairernotes}) and for the proof of Corollary \ref{cor:Kol}
(which is a generalization of Corollary 3.22 in \cite{hairernotes} to the 
non-Gaussian setting and to H\"older, rather than Lipschitz, functions \{$\psi_k$\}).
They are also grateful to Joris Bierkens, Patrick Conrad, Matthew Dunlop, 
Shiwei Lan, Yulong Lu, Daniel Sanz-Alonso,
Claudia Schillings and Aretha Teckentrup for
careful proof-reading of the notes and related comments.
AMS is grateful for various hosts who gave him the
opportunity to teach this material in short course form
at TIFR-Bangalore (Amit Apte), G\"ottingen (Axel Munk), PKU-Beijing (Teijun Li), ETH-Zurich (Christoph
Schwab) and Cambridge CCA (Arieh Iserles), 
a process which led to refinements of the material; the authors
are also grateful to the students
on those courses, who provided useful feedback.
The authors would also like to thank Sergios Agapiou and Yuan-Xiang Zhang for
help in the preparation of these lecture notes, including type-setting,
proof-reading, providing the proof of Lemma \ref{l:2.3} and delivering 
problems classes related to the short courses. AMS is also
pleased to acknowledge the financial support of EPSRC, ERC and ONR 
over the last decade whilst the research that underpins this work 
has been developed.

\end{document}